\documentclass[11pt,a4paper]{article}
\usepackage[a4paper, left=2.5cm,right=2.5cm,top=2.5cm,bottom=2.5cm]{geometry}
\usepackage[T1]{fontenc}
\usepackage{amsmath}
\usepackage{amsfonts}
\usepackage{amssymb}
\usepackage{amsthm}
\usepackage{mathtools}
\usepackage{relsize,scalerel}
\usepackage{graphicx}
\usepackage{fancyhdr}
\usepackage{enumitem}
\usepackage{stmaryrd}
\usepackage[dvipsnames]{xcolor}
\usepackage[nottoc, notlof, notlot]{tocbibind}
\usepackage{dsfont}
\usepackage{enumitem}
\usepackage{euscript}
\usepackage{mathrsfs}
\usepackage{subcaption}

\usepackage{emptypage}

\usepackage{appendix}
\usepackage{makeidx} 
\makeindex

\usepackage{minitoc} 

\usepackage{soul}

\usepackage{xcolor}
\definecolor{unbleu}{rgb}{0.03, 0.15, 0.4}
\definecolor{monvert}{rgb}{0.0,.5,0.0}
\definecolor{britishracinggreen}{rgb}{0.0, 0.26, 0.15}
\definecolor{monbleu}{rgb}{0,.2,.8}
\definecolor{monautrebleu}{rgb}{0,0.4,.75}
\definecolor{applegreen}{rgb}{0.55, 0.71, 0.0}
\definecolor{monrouge}{rgb}{0.8, 0.0, 0.0} 
\definecolor{cadmiumgreen}{rgb}{0.0, 0.42, 0.24}
\definecolor{royalblue(traditional)}{rgb}{0.0, 0.14, 0.4}
\definecolor{black}{rgb}{0.0, 0.0, 0.0}
\definecolor{sepia}{rgb}{0.44, 0.26, 0.08}
\definecolor{teagreen}{rgb}{0.82, 0.94, 0.75}
\definecolor{yellow-green}{rgb}{0.6, 0.8, 0.2}
\definecolor{azure(colorwheel)}{rgb}{0.0, 0.5, 1.0}
\definecolor{awesome}{rgb}{1.0, 0.13, 0.32}
\definecolor{cadmiumyellow}{rgb}{1.0, 0.96, 0.0}
\definecolor{carrotorange}{rgb}{0.93, 0.57, 0.13}
\definecolor{green-yellow}{rgb}{0.68, 1.0, 0.18}
\definecolor{huntergreen}{rgb}{0.21, 0.37, 0.23}
\definecolor{darkorange}{rgb}{0.85, 0.4, 0.0}
\definecolor{carrotorange}{rgb}{0.93, 0.57, 0.13}
\definecolor{turquoise}{rgb}{0.25, 0.88, 0.82}

\usepackage{hyperref}
\hypersetup{linkcolor=royalblue(traditional),citecolor=cadmiumgreen, urlcolor=azure(colorwheel), colorlinks=true}

\usepackage{pgfplots}
\pgfplotsset{compat=1.18}

\usepackage{tikz}
\usetikzlibrary{arrows.meta, positioning}
\usetikzlibrary{fit}
\usetikzlibrary{calc}
\usetikzlibrary{calligraphy}
\usetikzlibrary{decorations.pathreplacing}

\usepackage{framed}
\usepackage[most]{tcolorbox}
\tcolorboxenvironment{box-thm}{
  boxrule=0.5pt,
  colback=blue!5!white,
  colframe=blue!75!black,
  enhanced,
  sharp corners,
  before skip=10pt,
  after skip=10pt,
}

\allowdisplaybreaks

\numberwithin{equation}{section}
\input cyracc.def

\newcommand{\widesim}{\scaleobj{1.3}{\sim}}

\makeatletter
\newcommand{\uset}[3][0ex]{%
  \mathrel{\mathop{#2}\limits_{
    \vbox to#1{\kern-6\ex@
    \hbox{$\scriptstyle#3$}\vss}}}}
\makeatother

\newcommand{\id}{\operatorname{id}}

\newcommand{\lip}{\operatorname{Lip}}
\newcommand{\var}{\operatorname{var}}

\newcommand{\dd}{\mathrm{d}}

\newcommand{\fPp}{\operatorname{FPP}}
\newcommand{\cfPp}{\operatorname{CFPP}}
\newcommand{\geo}{\operatorname{Geo}}
\newcommand{\RPP}{\operatorname{RPP}}
\newcommand{\DRPP}{\operatorname{DRPP}}
\newcommand{\PPP}{\operatorname{PPP}}
\newcommand{\CPPP}{\operatorname{CPPP}}

\newcommand{\Par}{\operatorname{Par}}
\newcommand{\lr}{\scaleto{r}{5pt}}

\newcommand{\pr}{\operatorname{pr}}
\newcommand{\RV}{\mathrm{RV}}

\newcommand{\isEquivTo}[1]{\underset{#1}{\sim}} 
 
	\theoremstyle{definition}
	\newtheorem{defn}{Definition}[section]

	\theoremstyle{plain}
	\newtheorem{lem}{Lemma}[section]
	\newtheorem{thm}{Theorem}[section]
	\newtheorem{cor}{Corollary}[section]
    \newtheorem{prop}{Proposition}[section]

	\newtheorem{assumption}{Assumption}[section]

	\theoremstyle{plain}
	
	\newtheorem{rem}{Remark}[section]

\AtBeginDocument{
   \def\MR#1{}
}

\usepackage{etoolbox}
\makeatletter
\patchcmd{\tableofcontents}{\@starttoc{toc}}{\vspace*{-2em}\@starttoc{toc}}{}{}
\makeatother

\begin{document}

\title{Recurrence to rare events for all points in countable Markov shifts}


\author{Dylan Bansard-Tresse \vspace{0.2cm}\\ 
CPHT, CNRS, École polytechnique, \\ Institut Polytechnique de Paris, 91120 Palaiseau, France\\
\texttt{dylan.bansard-tresse@polytechnique.edu}\\
\small{\url{https://sites.google.com/view/bansard-tresse}}}


\maketitle
\begin{abstract}
    We study quantitative recurrence to rare events in Countable Markov Shifts (CMS) with recurrent potentials, focusing on return-time statistics to natural target sets for every point. In the positive recurrent case, return-time processes associated with nonperiodic points converge to a standard Poisson process, while those for periodic points converge to a compound Poisson limit. In the null-recurrent regime, three distinct behaviors arise. Points satisfying certain combinatorial conditions—in particular, the class is generic in the measure-theoretic sense— exhibit fractional Poisson limits, whereas periodic points yield compound fractional Poisson limits. For all remaining points, we describe the full family of possible limit laws, each Pareto-dominated. This classification is sharp: every limit behavior in this family can be realized by an explicit system. For canonical families of null-recurrent CMS, we identify the limit process for every point, providing a complete description of return-time statistics in these systems.
\end{abstract}

\textbf{Keywords:} infinite ergodic theory, rare events, point processes, fractional Poisson process, renewal processes, Pareto law

{\small
\tableofcontents
}

\section{Introduction}

\subsection{Context}


\paragraph{Returns to rare events in dynamical systems.} The study of hittings and returns in dynamical systems is an active research domain. The goal is to study limiting behaviors of the distribution of returns to small targets in the phase space. More explicitly, for a dynamical system $(X, \mathscr{B},\mu, T)$ and a sequence of asymptotically rare events $(B_n)_{n\geq 1}$ included in $X$ (\textit{i.e.} $\mu(B_n) \to 0$), one wishes to understand the behavior of 
\[r_{B_n} (=: r_{B_n}^{(1)}) : x \mapsto \inf \{j \geq 1 \;|\; T^j(x) \in B_n\} \]
and more generally of the successive returns
\[r_{B_n}^{(k)} : x \mapsto \inf \{j > r_{B_n}^{(k-1)}(x) \;|\; T^j(x) \in B_n\} \quad \text{for $k\geq 2$}.\]
In the late 1990's, it was shown that for uniformly expanding systems of the interval endowed with their unique absolutely continuous invariant probability (a.c.i.p.), the limiting behavior of $\mu(B_n)\,r_{B_n}$ is exponential provided that $B_n$ is a sequence of nested ``natural'' targets (\textit{i.e.} metric balls or dynamically defined cylinders) centered at a generic point $x \in X$ (in the sense that it works for $\mu$-a.e. $x\in X$). This property has been shown to still hold for more probability preserving systems including non-uniformly expanding and hyperbolic ones \cite{CampaninoIsola95,Col02, AS11, ChazottesCollet13, AS16}. When we keep track of the successive returns, via the study of the following random measure, called Rare Event Point Process (REPP)
\begin{align*}
    N_{B_n}^{\id} := \sum_{k\geq 1} \delta_{\mu(B_n)\,r_{B_n}^{(k)}},
\end{align*}
the limit obtained is the standard Poisson point process on the half-line (see \cite{ChazottesCollet13, LFFF16} and the references therein). \\

However, for many applications, it is often important to understand the behavior of the system at specific points $x\in X$, rather than inside a set of full measure. In particular, early work revealed that for periodic points, convergence to a standard Poisson point process is impossible \cite{Hirata93}. This failure is due to clustering of rare events arising from the shorts returns that may appear from the special periodic behavior of the orbit of $x$. In such cases, the limit of the rare event point process is instead a compound Poisson process, whose multiplicity distribution is determined by a parameter $\theta$ depending solely on the periodic point $x$ considered. By analogy with extreme value theory, $\theta$ is known as the extremal index (see \cite{LFFF16} for a more detailed introduction). \\

Over the past decade, a central question has been to determine whether limit laws could be proven for all points in a given dynamical system. When a system allows the derivation of such limit laws at every point, it is said to satisfy the \textit{all-point REPP} property. Remarkably, for some \textit{chaotic} systems, it has been shown that periodic points are the only ones exhibiting non-standard Poissonian behavior: all other points give rise to a standard Poisson point process in the limit and thus it provides an \textit{all-point REPP} property with a \textit{dichotomy} between periodic and non-periodic points (see \cite{AFV15, LFFF16} and the references therein). It has been first established for uniformly expanding systems \cite{AFV15} and latter proven for some non-uniformly expanding maps, such as the Manneville–Pomeau map \cite{FFTV16} or quadratic maps with Misiurewicz parameters \cite{BF23}.\\ 

To date, several approaches coexist in the literature to prove such results for probability preserving systems: 
\begin{itemize}
    \item Direct mixing-based methods exploit explicit decay of correlations to compare the dynamical system with the i.i.d. setting, providing sufficient conditions for convergence (see \cite{Haydn13_EntryAndReturnTimesDistribution, Sau09_Survey_AnIntroductionToQuantitativeRecurrenceInDynamicalSystems} and references therein).
    \item Connection with extreme value theory where sufficient conditions can be adapted from the theory of maxima of stochastic processes to the dynamical context \cite{FFT10, FFT13, LFFF16}.
    \item Spectral methods where we rely on good spectral properties of the transfer operator (typically a spectral gap) and stability under small perturbations, to deduce Poissonian limits \cite{KL09, Kel12, Zha21, atnip2023compound}.
    \item Leverage of the equivalence between hitting time statistics and return time statistics \cite{HLV05, Mar17, Zwe16} to build sufficient conditions for the convergence towards (compound) Poisson point processes \cite{Zwe22}.
\end{itemize}

In general, these methods are directly applicable primarily to systems exhibiting pronounced chaotic characteristics. Within the domain of ergodic theory, a pivotal instrument is the use of inducing. Regarding the issue of returns, it has been established that convergence for the induced dynamical system is equivalent to the convergence for the original system \cite{HWZ14, FFTV16}, assuming that the inducing process is executed via a first return. This approach was subsequently modified and implemented to demonstrate the \textit{all-point REPP} property for certain non-uniformly hyperbolic systems \cite{FFTV16, BF23}.  \\

Nevertheless, all these results require the dynamical system to be probability preserving while there are a lot of natural dynamical systems where the meaningful preserved measure is infinite (while remaining $\sigma$-finite). Examples encompass null-recurrent Markov chains, interval maps with indifferent fixed points, $\mathbb{Z}^d$-extensions of probability preserving systems or billiards with cusps. The question of hittings and returns for such systems is more limited but is of growing interest. Contrary to the finite measure case where the scaling $\mu(B_n)$ is computed from Kac's theorem, this is no longer the case in the infinite setting. Hence, the REPP is adapted with a suitable scaling $\gamma$ that may depend on the system and it is thus defined in the following way:
\[N_{B_n}^{\gamma} := \sum_{k\geq 1} \delta_{\gamma(\mu(B_n))\,r_{B_n}^{(k)}}.\]
The convergence of first hitting and return times has been investigated, revealing non-exponential limiting laws for various infinite measure preserving dynamical systems and natural targets \cite{BZ01, PeneSaussol10_BackToBallsInBilliards, RZ20, Yas18, 
Yas24_quantitativerecurrencezextensionthreedimensional}, with fewer studies considering 
the whole sequence of successive returns \cite{PSZ13, BT24_FractionalPoisson}. For the natural class of pointwise dual ergodic systems (see Definition \ref{defn:pointwise_dual_ergodicity}) satisfying some regular variation hypothesis, the Fractional Poisson Process (FPP) was shown to emerge as the pivotal limit distribution, acting as the analog of the standard Poisson point process in the infinite measure preserving setting \cite{BT24_FractionalPoisson}. Note that the fractional Poisson process was first defined, far from its emergence in ergodic theory, as a generalization of the Poisson point process to model systems presenting long-term temporal correlations, preventing exponential behavior \cite{Las03}. The study of its properties and applications in probability theory is also an active domain of research (see \textit{e.g.} \cite{MS19_Stochastic_models_for_fractional_calculus} and the references therein). \cite{BT24_FractionalPoisson} is also the first to show an \textit{all-point REPP} property in the infinite setting with a proof adapted to the specific example of the Manneville-Pomeau family and targets that are only dynamically defined cylinders, sometimes only left or right neighborhoods of the point considered. 

\paragraph{Symbolic dynamics and Countable Markov Shifts.}
Symbolic dynamics arise naturally in statistical physics, where the phase space is represented as a subset of sequences over an alphabet. For instance, one may model a one-dimensional lattice gas by letting the alphabet encode the possible local states. This symbolic viewpoint is equally powerful in deterministic dynamics, as it allows one to encode chaotic systems and deduce their statistical properties from the corresponding symbolic representation. This idea was popularized by Bowen, who showed that Axiom A diffeomorphisms are conjugate to subshifts of finite type via Markov partitions, and thereby proving strong stochastic properties for such systems \cite{Bowen75}. Subshifts of finite type are now well understood. However, since they rely on a finite alphabet, they cannot capture many natural non-uniformly hyperbolic systems. This motivated the study of the broader class of topological Markov shifts, which allow a countable alphabet and encompass fundamental examples \cite{tdfsurvey}. Recently, such shifts have proved to be a key ingredient in the thermodynamical formalism for surface diffeomorphisms \cite{BCS22}. Furthermore, they also are a natural generalization of Markov chains, accommodating stronger—potentially infinite—dependencies between successive symbols. General topological Markov shifts may even display invariant measures with an infinite mass, which lies far beyond what is possible for subshifts of finite type and lead to completely different behaviors.\\

Despite the importance of this class of dynamical systems, the question of quantitative recurrence for topological Markov shifts remains largely open, except in the simpler setting of subshifts of finite type \cite{Coelho00}. This gap motivates the present work, where we investigate recurrence statistics and their limiting laws in this broader symbolic setting.

\subsection{Contributions}
    This paper provides an in-depth study of quantitative recurrence for topological Markov shifts. First, we show that topological Markov shifts associated to a positive recurrent potential satisfy the standard \textit{all-point REPP} with a dichotomy between periodic and non-periodic points, generalizing the result known for subshifts of finite type and showing that no new behavior emerges even though some properties are relaxed (Theorem \ref{thm:all-point_REPP_positive_recurrent_CMS}). The method of proof is based on the techniques developed in \cite{Zwe22} and the use of inducing. \\
    
    Then, we study the more delicate case of countable Markov shifts endowed with a null recurrent potential. We obtain the following contributions:
\begin{itemize}
    \item We get combinatorial conditions for the convergence towards the fractional Poisson process. If a point $x$ of the phase space is non periodic and is infinitely recurrent for at least one letter of the alphabet, then the REPP associated to cylinders shrinking to $x$ converges in distribution towards a fractional Poisson process (Theorem \ref{thm:HTS_infinite_FPP_infinitely_recurrent}). In particular, this properties is more general than having the result for a set of full measure as it makes explicit the points that have such a property. 
    \item If the point considered is periodic, then the associated REPP converges in law towards a compound fractional Poisson process which is driven by an explicit extremal index (Theorem \ref{thm:HTS_infinite_CFPP_periodic_points}).
    \item For the remaining points, we are able to get necessary and sufficient conditions for the convergence of the REPP. These conditions are express in terms of limit distribution of the first hit of a \textit{non}-rare event (Theorem \ref{thm:HTS_infinite_points_that_never_come_back_statements}). Furthermore, we are able to characterize the family of reachable limits in this case (Proposition \ref{prop:only_possible_limits_for_waiting_times_in_G}). We even show that this family is sharp in the sense for every law, it is always possible to build a CMS that have this law that emerges at least for neighborhoods of a point in the phase space (Proposition \ref{prop:Every_waiting_time_achievable_when_CMS_well_chosen}). It shows that new behaviors may emerge in the CMS context, breaking the analogy with the finite measure case, where such limits are forbidden.  
    \item For the most studied examples of null-recurrent systems in the literature, like $\mathbb{Z}$-extensions over subshifts of finite type or tower shifts, we are able to prove the convergence at every point and explicitly identify the limit point processes, hence proving the \textit{all-point REPP}. In particular, it shows distinct behaviors. While for $\mathbb{Z}$-extensions, we get a dichotomy between non-periodic points---with a fractional Poisson process---and periodic point---with a compound fractional Poisson process (in particular generalizing the result of \cite{PS23} that only holds for a set of full-measure), a third class of limit point processes emerges for a countable set of points in the case of tower shifts (Theorem \ref{thm:all_points_HoC_structure} for the tower shift and Theorem \ref{thm:all_points_Z_extension} for $\mathbb{Z}$-extensions). 
    \item Finally, we also bring some contributions to the study of other families of asymptotically rare events that are not necessarily cylinders shrinking towards a point. In particular, we show that returns close to an embedded subshift of finite type converges in distribution towards a compound fractional Poisson process with an extremal index driven by the relative induced pressure (Theorem \ref{thm:REPP_subshift_CFPP}). 
\end{itemize}

Along the way, we explicit abstract sufficient conditions for the convergence towards new point processes, different from (compound) fractional Poisson process (Theorem \ref{thm:sufficient_conditions_convergence_other_point_processes}). This abstract theorem works in higher generality than Markov shifts and could be used to derive limit point process for other examples of infinite measure preserving systems and other classes of asymptotically rare events in the future.

\paragraph{Organization of the paper.} The article is organized as follows. Section \ref{section:preliminaries} defines the main objects from the theory of topological Markov shifts and quantitative recurrence in infinite ergodic theory. In Section \ref{section:positive_recurrent_CMS}, we tackle the case of probability preserving topological Markov shift, \textit{i.e} positive recurrent potentials. Section \ref{section:quatitative_recurrence_null-recurrent_CMS} is devoted to the null recurrent case where we exhibit sufficient combinatorial conditions for convergence towards fractional Poisson process as well as necessary and sufficient conditions for the convergence to other point processes for shrinking cylinders. In Section \ref{section:possible_limit_laws} we treat the question of the possible limit laws and build examples where they emerge. Section \ref{section:all-point_REPP_for_examples} is dedicated to the study of paradigmatic families of null-recurrent CMS for which we establish the \textit{all-point REPP} property. Finally the study of visits close to an embedded subshift of finite type is presented in Section \ref{section:convergence_embedded_SFT}. The technical proofs are postponed to the appendix.

\paragraph{Acknowledgements.} The present article is part of the author PhD thesis prepared at École Polytechnique. The author was partially financed by Portuguese public funds through FCT  -- Funda\c{c}\~ao para a Ci\^encia e a Tecnologia, I.P., in the framework of the projects PTDC/MAT-PUR/4048/2021 and CMUP's project with reference UID/00144/2025 (Doi: https://doi.org \\/10.54499/UID/00144/2025).

\section{Preliminaries}
\label{section:preliminaries}

Let $(X,\mathscr{B},\mu, T)$ be a measure-theoretic dynamical system. This means that $(X, \mathscr{B}, \mu)$ is a measure space and the self-map $T : X \circlearrowleft$ leaves the measure $\mu$ invariant (\textit{i.e.} the push forward $T_{\#}\mu$ is equal to $\mu$). We assume that $\mu$ is finite or $\sigma$-finite. The transfer operator $\widehat{T} : L^1(\mu) \to L^1(\mu)$ of the system is defined via the following identity: 
\[\int f\cdot(g\circ T)\,\dd\mu = \int (\widehat{T}f)\cdot g\,\dd\mu \quad \forall f\in L^1(\mu), \; \forall g\in L^{\infty}(\mu).\] 
We say that $(X,\mathscr{B},\mu, T)$ is a conservative ergodic measure preserving 
transformation (CEMPT for short) if $\sum_{k\geq 0} \widehat{T}^ku = +\infty$ $\mu$-a.e. for all $u \in L^1_+(\mu) := \big\{ u\in L^1(\mu)\;|\; u\geq 0,\, \int u\,\dd\mu > 0\big\}$ or, equivalently, if this is true for all
$u \in \mathcal{D}(\mu) := \big\{u \in L^1(\mu)\;|\; u\geq 0,\, \int u\,\dd\mu = 1\big\}$ \cite[Propostion 1.3.2]{Aar97}. It is easy to see that ergodic probability preserving systems are CEMPT, while the classical definition of ergodicity (\textit{i.e.} for all $A\in \mathscr{B}$, $T^{-1}A = A$ implies $\mu(A) = 0$ or $\mu(A^c) = 0$) is necessary but not sufficient to be a CEMPT when $\mu(X) = +\infty$. 

\paragraph{Countable Markov Shifts.} We now formally define countable Markov shifts and present the classification of potentials. For a more detailed introduction to countable Markov shifts and their properties, we refer to the survey \cite{tdfsurvey} and the references therein. Let $G(V, E)$ be a direct graph with the set of vertices $V$ being at most countable ($V$ is also called the alphabet and an element $v\in V$ is called a letter or a state). We define the usual transition matrix $A : V \times V \to \{0,1\}$ by $A(v,v') = 1$ if $(v,v') \in E$ and $A(v,v') = 0$ else and write $v \to v'$ if $A(v,v') = 1$. More generally, we say $v \xrightarrow[]{n} v'$ if it exists $w_1, \dots, w_{n-1} \in V$ such that $v \to w_1 \to \dots \to w_{n-1} \to v'$. In such case, we say that $(vw_1^{n-1}v')$ is an admissible sequence or word. Then, a topological Markov shift (TMS) associated to the graph $G$ is the topological dynamical system $(\Omega, T)$ where 
\begin{align*}
    \Omega := \Omega_G := \big\{x \in V^{\mathbb{N}}\; \big|\; A(x_n,x_{n+1}) = 1\quad \forall n\geq 0\big\}
\end{align*}
and $T : \Omega \circlearrowleft$ is the usual shift map defined by $T((x_n)_{n\geq 0}) = (x_{n+1})_{n\geq 0}$. We endow $\Omega$ with the geometric distance $d_{\eta}(x,y) = \eta^{\min\{n \geq 0\;|\; x_n \neq y_n\}}$ for some fixed $\eta \in (0,1)$. Then, $(\Omega, d_{\eta})$ is a bounded complete metric space and its topology is generated by the cylinders 
\begin{align*}
    [v_0^{n-1}] := \{x\in \Omega \;|\; x_0^{n-1} = v_0^{n-1}\}
\end{align*}
defined for all $n\geq 1$ and every admissible sequence $(v_0^{n-1})$. In this case, we say that $[v_0^{n-1}]$ is an $n$-cylinder and we call $n$ its depth (or length). By convention, if $(v_0^{n-1})$ is non admissible, $[v_0^{n-1}] := \emptyset$. Let $\mathcal{C}(n)$ be the collection of admissible $n$-cylinders. By construction, $\mathcal{C}(n)$ is a (generating) partition of $\Omega$ and we have
\begin{align*}
    \mathcal{C}(n) := \bigvee_{k=0}^{n-1} T^{-k}\mathcal{C}(1)\,.
\end{align*}
The phase space $(\Omega, d_{\eta})$ is compact if and only if the alphabet $V$ is finite and the TMS is called a subshift of finite type (SFT) in this case. When $V$ is infinite, the phase space is not compact and might not even be locally compact and we say that it is a countable Markov shift (CMS). \\

We will assume that our TMS are topologically mixing (\textit{i.e.} for every $v, w \in V$ there exists some $N \geq 0$ (that might depend on $v$ and $w$) such that for all $n\geq N$, $v\xrightarrow[]{n} w$). By spectral decomposition, topologically transitive (\textit{i.e.} for all $v, w \in V$, there exists $n\geq 1$ such that $v \xrightarrow[]{n} w$) should be enough to obtain similar results but, for simplicity, we will always topological mixing.\\ 



To define weights on our system and potentially characterize invariant measures, we need to endow our TMS with a potential which is a measurable function $\phi : \Omega \to \mathbb{R}$. Its regularity is characterized through its variations.

\begin{defn}[Variation]
    Let $\phi$ be a potential on $\Omega$. For all $n\geq 1$, we call the $n$-th variation of $\phi$ the quantity $\var_n(\phi) := \sup \big\{|\phi(x) - \phi(y)|\;\big|\; x_0^{n-1} = y_0^{n-1}\big\}$. Furthermore, if we assume some regularity on the potential, we say that $\phi$ is
    \begin{enumerate}[label = $(\roman*)$]
        \item Markovian if $\var_2(\phi) = 0$,
        \item weakly Hölder if there exists $C > 0$ and $0 < \eta < 1$ such that $\var_n(\phi) \leq C\eta^n$ for all $n\geq 2$,
        \item with summable variations if $\sum_{n\geq 2} \var_n(\phi) < +\infty$,
        \item satisfying Walters' condition (or simply is Walters) if for all $k\geq 1$,
        \begin{align*}
            W_k(\phi) := \sup_{n\geq 1} \var_{n+k}(S_n\phi) < +\infty \; \text{and}\; W_k(\phi) \xrightarrow[k \to +\infty]{} 0.
        \end{align*}
        where $S_n \phi := \sum_{k = 0}^{n-1} \phi \circ T^k$ is the classical Birkhoff sum.
    \end{enumerate}
\end{defn}

\begin{rem}
    We have the following implications concerning the regularity of the potential: $(i) \xRightarrow[]{} (ii) \xRightarrow[]{} (iii) \xRightarrow[]{} (iv) \xRightarrow[]{} \text{continuity}$.
\end{rem}

We can see the effect of the potential of the dynamics through its associated Ruelle's Perron-Frobenius operator.

\begin{defn}[Ruelle's Perron-Frobenius operator (RPF operator)] 
    For a continuous potential $\phi$, we define its associated Ruelle's Perron-Frobenius operator (RPF) operator $L_{\phi}$ by
    \begin{align*}
        L_{\phi}f(x) := \sum_{Ty = x} e^{\phi(y)}f(y)\,, \quad x \in \Omega.
    \end{align*}
\end{defn}


We say that the potential has a finite Gurevich pressure if the quantity $P_G(\phi):= \lim_{n\to +\infty} \allowbreak n^{-1}\log(Z_n(\phi, v))$ is finite, where for some $v \in V$, $Z_n(\phi, v) = \sum_{T^n x = x} e^{S_n\phi(x)} \mathbf{1}_{[v]}(x)$. Note the limit does not depend on the choice of the state $v$. Then, for such potentials, one is able to give a classification, generalizing the one existing for Markov chains.

\begin{defn}[Classification of potentials]
\label{defn:CMS_classification_of_potentials}
Assume $\phi$ is a potential satisfying Walters' condition on a (topologically mixing) TMS such that $P_G(\phi) < +\infty$. Furthermore, for all $v \in V$, set $Z_n^*(\phi, v) := \sum_{T^nx = x} e^{S_n\phi(x)}\mathbf{1}_{\{r_{[v]} = n\}}(x)\mathbf{1}_{[v]}(x)$.
 Let $\lambda_{\phi} := e^{P_G(\phi)}$ and fix $v\in V$. We say that
 \begin{enumerate}[label = (\roman*)]
     \item $\phi$ is positive recurrent if 
     $
         \sum_{n\geq 1} \lambda_{\phi}^{-n}Z_n(\phi, v) = +\infty$ and $\sum_{n\geq 1} n \lambda_{\phi}^{-n}Z_n^*(\phi, v) < +\infty.
     $
     \item $\phi$ is null recurrent if 
     $
         \sum_{n\geq 1} \lambda_{\phi}^{-n}Z_n(\phi, v) = +\infty$ and $\sum_{n\geq 1} n \lambda_{\phi}^{-n}Z_n^*(\phi, v) = +\infty.
     $
     \item $\phi$ is transient if
     $
         \sum_{n\geq 1} \lambda_{\phi}^{-n}Z_n(\phi, v) < +\infty.
     $
 \end{enumerate}
\end{defn}

Such a classification can be equivalently characterized through spectral properties of the RPF operator, more suitable to the ergodic theory context (see \cite[Section 4]{tdfsurvey} for a survey on this result).

\begin{thm}[Generalized Ruelle-Perron-Frobenius (GRPF) Theorem \textup{\cite[Theorem 1]{Sar01_NullRecurrentPotentials}, \cite[Theorem 4.9]{tdfsurvey}}] 
    \label{thm:GRPF}
	Let $(\Omega, T)$ be a topologically mixing TMS and $\phi : \Omega \to \mathbb{R}$ a Walters potential with $P_G(\phi) < +\infty$. Then,
	\begin{enumerate}[label = (\roman*)]
	 	\item $\phi$ is positive recurrent if and only if there exist $\lambda_{\phi}$, $h_{\phi} : \Omega \to \mathbb{R}$ continuous strictly positive and $\nu_{\phi}$ a conservative measure finite on cylinders such that $L_{\phi}h_{\phi} = \lambda_{\phi}h_{\phi}$, $L_{\phi}^*\nu_{\phi} = \lambda_{\phi}\nu_{\phi}$ and $\int h_{\phi}\,\dd\nu_{\phi} = 1$. In this case, the measure $\mu_{\phi} \ll \nu_{\phi}$ with $\dd \mu_{\phi}/\dd\nu_{\phi} = h_{\phi}$ is a ergodic shift invariant probability. 
     
        \item $\phi$ is null recurrent if and only if there exist $\lambda_{\phi}$, $h_{\phi} : \Omega \to \mathbb{R}$ continuous strictly positive and $\nu_{\phi}$ a conservative measure finite on cylinders such that $L_{\phi}h_{\phi} = \lambda_{\phi}h_{\phi}$, $L_{\phi}^*\nu_{\phi} = \lambda_{\phi}\nu_{\phi}$ and $\int h_{\phi}\,\dd\nu_{\phi} = +\infty$. In this case, the measure $\mu_{\phi} \ll \nu_{\phi}$ with $\dd \mu_{\phi}/\dd\nu_{\phi} = h_{\phi}$ is a conservative ergodic shift invariant $\sigma$-finite measure. 
        \item $\phi$ is transient if there are no conservative measure $\nu$ such that $L_{\phi}^*\nu = \lambda\nu$ for some $\lambda > 0$.
	\end{enumerate} 
\end{thm}

\begin{rem}
    \label{rem:regularity_log_h}
    In the recurrent case, the proof of the generalized Ruelle-Perron-Frobenius theorem gives more information on the regularity of $\log(h_{\phi})$, and in particular,
    \begin{align*}
        \var_k\log(h_{\phi}) \leq W_k(\phi) \quad \forall k \geq 1.
    \end{align*}
\end{rem}

\begin{rem}
    \label{rem:definition_phi_*}
    In the recurrent case, we have $\widehat{T}_{\mu_{\phi}} = L_{\phi_*}$ with $\phi_* := \phi + \log h_{\phi} - \log h_{\phi}\circ T  - P_{G}(\phi)$. In particular, $L_{\phi_*}1 = 1$. By the previous remark, $\phi_*$ also satisfies Walters' condition and $P_G(\phi_*) = 0$. This particular potential $\phi_*$ associated to $\phi$ will often be more suitable for the properties we will be presenting later on.
\end{rem}

\noindent \textbf{Notations.} To alleviate notations, we will drop the $\phi$ in the indices when the potential is clear from the context. Furthermore, when restricting to a one-cylinder $[v]$ for $v \in V$, we will sometimes omit the brackets in the indices.

\paragraph{Returns to rare events and infinite ergodic theory.} For a subset $B \in \mathscr{B}$ and $x \in X$, we set $r_B(x) := r_B^{(1)}(x) := \inf \{n \geq 1 \;|\; T^nx \in B\}$ the first return time to $A$ of $x$. By induction, we also define $r_B^{(k+1)}(x) := \inf \{n > r_B^{(k)}(x) \;|\;T^nx \in B\}$ for all $k\geq 1$ and with the convention that $\inf \emptyset = +\infty$ and $r_B^{(k')}(x) = +\infty$ if $r_B^{(k)}(x) = +\infty$ for some $k \leq k'$. Note that if $(X, \mathscr{B}, \mu, T)$ is a CEMPT, that $r_B^{(k)}$ is finite almost surely for all $k\geq 1$ provided that $\mu(B) > 0$. One way of keeping track of every (scaled) return at the same time is to look at the counting measure $N_{B_n}^{\gamma}$ on $\mathbb{R}_+$ defined as 
\begin{align}
    \label{eq:defn_EPP}
        N_{B}^{\gamma} := \sum_{k \geq 1} \delta_{\gamma(\mu(B))\,r^{(k)}_{B}}.
\end{align}
The measure is called (by a slight abuse of notation) the event point process (EPP) associated to $B$ and $\gamma : \mathbb{R}_+ \to \mathbb{R}_+$ is called the scaling function. However, our goal is to understand the evolution of the EPP when the measure of $B$ gets smaller and smaller. We say that a sequence $(B_n)_{n\geq 1} \in \mathscr{B}^{\mathbb{N}}$ is a sequence of asymptotically rare events if $\mu(B_n) > 0$ for all $n \geq 0$ and $\mu(B_n) \to 0$ as $n \to +\infty$.

\begin{defn}[Rare Event Point Processes]
    \label{defn:REPP}
    Let $(B_n)_{n\geq 0}$ be a sequence of asymptotically rare events. We call Rare Event Point Processes (REPP) (with scaling $\gamma$) the sequence of EPP associated to $(B_n)_{n\geq 0}$ (with scaling $\gamma$). 
\end{defn}

As we aim to understand the limit behavior of the REPP, from a statistic point of view, we need to properly choose our starting measured space. There are two main ways of doing it. Either we take a point randomly chosen in the whole phase space, either we can consider only the points that started directly in the targets and thus driven by the induced measure $\mu_{B} := \mu(\cdot \cap B)/\mu(B)$. This is the reason why we introduce the following definition.

\begin{defn}[Hitting and return REPP]
    \label{defn:hitting_and_return_REPP}
    The hitting REPP for $(B_n)_{n\geq 1}$ is the sequence of random variables $N_{B_n}^{\gamma}$ on the probability space $(X, \mathscr{B}, \nu)$ for some $\nu \ll \mu$ with $\nu(X) = 1$. \\
    The hitting REPP for $(B_n)_{n\geq 1}$ is the sequence of random variables $N_{B_n}^{\gamma}$ on the (changing with $n$) probability space $(B_n, \mathscr{B} \cap B_n, \mu_{B_n})$.
\end{defn}

\begin{rem}
    \label{rem:REPP_random_variables_vs_probabilities}
    In fact, as we are only interested by convergence in law in this work, we will confuse the random variables with their associated distribution, \textit{i.e} we may talk of the hitting REPP as the sequence of probabilities $(N_{B_n}^{\gamma})_{\#} \nu$ and the return REPP as the sequence of probabilities $(N^{\gamma}_{B_n})_{\#}\mu_{B_n}$.
\end{rem}

\begin{rem}
    For the hitting REPP, it can seem arbitrary to choose a specific probability $\nu \ll \mu$ because for a fixed set $B$, it will change the distribution of the EPP. However, \cite[Corolarry 6]{Zwei07} ensures that if a limit distribution exists for a specific probability $\nu \ll \mu$, it exists for every probability $\nu' \ll \mu$ and the limit is the same. Of course, when $\mu(X) = 1$, we can simply choose $\nu = \mu$ in the previous definition. 
\end{rem}

Now comes the question of the right choice of the scaling $\gamma$. When the system is probability preserving, the scaling directly comes from Kac's Theorem which ensures that for all $B \in \mathscr{B}$ with $\mu(B) > 0$, $\mathbb{E}_{\mu_B}[r_B] = +\infty$. Thus, the scaling $\gamma := \id$ is the natural universal choice. Nevertheless, when $\mu(X) = +\infty$, the answer to that question is trickier. Kac's Theorem only states the infiniteness of the expected value and it turns out that no universal scaling can be found \cite[Theorem 2.1]{RZ20}. In order to find adapted scaling, we need to introduce more properties on the system. We briefly present the main properties needed but for a more detail introduction of such systems and properties in infinite ergodic theory, we refer to \cite{Aar97}. First, we will assume that our systems are pointwise dual ergodic (PDE).

\begin{defn}[Pointwise dual ergodicity (PDE)]
    \label{defn:pointwise_dual_ergodicity}
    A CEMPT $(X, \mathscr{B}, \mu, T)$ is said to be pointwise dual ergodic (PDE) if there exists a sequence $(a_n)_{n\geq 0}$ such that 
    \begin{align*}
        \frac{1}{a_n} \sum_{k = 0}^{n-1} \widehat{T}^k f \xrightarrow[n \to +\infty]{} \int f\,\dd\mu, \quad \mu-a.e. ,\quad \forall f \in L^1(\mu).
    \end{align*}
    In this case, we call $(a_n)_{n\geq 1}$ a normalizing sequence.
\end{defn}

Note that, although it seems like an equivalent of Birkhoff's Theorem in the infinite context (with iterations of the transfer operator instead), it is a property that is not satisfied for every CEMPT. Furthermore, due to the infiniteness of $\mu$, we necessarily have $a_n = o(n)$. The PDE property is equivalent to the existence of uniform sets, which are of paramount importance for our theory.

\begin{defn}[Uniform set]
    A set $Y \in \mathscr{B}$ with $\mu(Y) > 0$ is said to be uniform if there exists a function $f \in L^1(\mu)$ with $f\geq 0$ and $\int_Y f\,\dd\mu$ and a sequence $(a_n)_{n\geq 1}$ such that 
    \begin{align*}
        \bigg\| \;\frac{1}{a_n} \sum_{k = 0}^{n-1} \widehat{T}^k f - \int f\,\dd\mu\;\bigg\|_{L^{\infty}(\mu_Y)} \xrightarrow[n \to +\infty]{} 0.
    \end{align*}
    We also say that the set $Y$ is $f$-uniform. In particular, if $Y$ is uniform for $\mathbf{1}_Y$, one says that $Y$ is a Darling-Kac set. 
\end{defn}

In fact, the existence of uniform set and the PDE property are equivalent. The  implication follows from Egorov's theorem, while the reciprocal can be found in \cite[Proposition 3.7.5]{Aar97}. \\

Together with the PDE property, we will assume that our normalization satisfies some regular variation hypothesis. We say that a function $f : \mathbb{R}_+ \to \mathbb{R}_+$ is regularly varying (at $+\infty$) of index $\alpha \in \mathbb{R}$ if for every $\lambda > 0$, $f(\lambda z)/f(z) \xrightarrow[z\to +\infty]{} \lambda^{\alpha}$. This notion is generalized for sequences where we say that $(u_n)_{n\geq 1}$ is regularly varying if $u : z \mapsto u_{\lfloor z\rfloor}$ is regularly varying. In our context, we will always assume that the normalizing sequence $(a_n)_{n\geq 1}$ coming from the PDE property is regularly varying for some index $0 < \alpha \leq 1$. When both properties (PDE $\&$ regular variation) are verified, the scaling 
\begin{align}
    \label{eq:defn_gamma}
    \gamma(s) := \frac{1}{a^{\leftarrow}\big(\frac{1}{s}\big)}
\end{align}
is the right scaling for sequence of asymptotically rare events $(B_n)_{n\geq 1}$ remaining inside a uniform set $Y$ \cite[Theorem 2.2]{RZ20}. Here $f^{\leftarrow}$ stands for the generalized inverse of the function $f$ and is defined by $f^{\leftarrow}(x) := \inf \{y\geq 0\;|\; f(y) > x\}$ (again $(a_{n})_{n\geq 1}$ is only a sequence but it can be seen as a function by $a : s \mapsto a_{\lfloor s\rfloor}$ and we keep the same notation $a$ for both). Thus, in the definition of the point process \eqref{eq:defn_EPP}, we will always assume that $\gamma$ is defined from \eqref{eq:defn_gamma}. Note that by the standard theory of regular variation, if $(a_n)_{n\geq 1} \in \RV(\alpha)$ then $\gamma \in \RV_0(1/\alpha)$ where $f \in \RV_0(\beta)$ if $s \mapsto f(1/s)^{-1} \in \RV(\beta)$ \cite[Theorem 1.5.12]{Bingham89_RegularVariation}. \\

In this context, it has been established that the convergence of the hitting REPP and the return REPP are equivalent \cite[Theorem 2.1 and Corollary 2.1]{BT24_FractionalPoisson}. \\

Going back to the generalized Ruelle's Perron-Frobenius theorem (Theorem \ref{thm:GRPF}), we are able to state more properties for null recurrent CMS.

\begin{thm}
    \label{thm:GRPF_null_recurrent}
    Let $(\Omega, T)$ be a topologically mixing TMS and $\phi : \Omega \to \mathbb{R}$ a potential satisfying Walters' condition and with $P_G(\phi) < +\infty$. Assume that $\phi$ is null-recurrent and consider $\mu_{\phi}$ defined from Theorem \ref{thm:GRPF}-$(ii)$. Then, $(\Omega, T, \mu_{\phi})$ is PDE with normalizing sequence $(a_n)_{n\geq 1}$ such that for every $v \in V$,
        \begin{align*}
            a_n \uset{\widesim}{n\to +\infty} \frac{1}{\mu_{\phi}[v]} \sum_{k = 1}^{n} \lambda_{\phi}^{-k}Z_k(\phi, v)\,.
        \end{align*}
    Furthermore, for all $k\geq 1$ and  $(a_0^{k-1}) \in V^k$ admissible, $[a_0^{k-1}]$ is a Darling-Kac set.
\end{thm}

\paragraph{Point Processes.} We present now few families of point processes that will appear in the following section of this paper. For a more detailed introduction to point processes, we refer to \cite{DaleyVereJones03_VolumeI, DaleyVereJones08_VolumeII} or \cite{LastPenrose18_LecturesOnPoissonProcess}. On $\mathbb{R}_+$, we define the set of boundedly finite counting measure $\mathcal{N}_{\mathbb{R}_+}^{\#}$ as the subset of boundedly finite measure on $\mathbb{R}_+$ (\textit{i.e.} $m(A) < +\infty$ for all $A$ bounded) with the additional property that $m(A) \in \{0, 1, 2, \dots\}$ for all $A$ bounded. Endowed with the weak$^*$ topology, $\mathcal{N}_{\mathbb{R}_+}^{\#}$ is a Polish space. In probability theory, we call a point process (on $\mathbb{R}_+$) a random variable taking values in $\mathcal{N}_{\mathbb{R}_+}^{\#}$. By Definition \ref{defn:hitting_and_return_REPP}, both the hitting and the return REPP are sequences of point processes taking values in $\mathcal{N}_{\mathbb{R}_+}^{\#}$.\\

Most of the processes that will appear can be expressed as renewal point processes.

\begin{defn}[Renewal point process (RPP)] 
\label{defn:renewal_point_process}
Let $W$ be a non negative random variable such that $\mathbb{P}(W > 0) > 0$ and $\mathbb{P}(W < +\infty) = 1$. Let 
$(W_i)_{i\geq 1}$ be i.i.d. random variables having the same law as $W$. The renewal 
point process with waiting time $W$ is defined by
\[
\RPP(W) \overset{\text{(law)}}{=} \sum_{i = 1}^{+\infty} \delta_{T_i},
\]
where $T_{i+1} - T_{i} = W_{i+1}$ for all $i\geq 0$ and with the convention $T_0 = 0$.
Note that $\mathbb{P}(\RPP(W) \in \mathcal{N}_{\mathbb{R}_+}^{\#\infty}) = 1$ where $\mathcal{N}_{\mathbb{R}_+}^{\#\infty}$ is the set of counting measures $m$ such that $m(\mathbb{R}_+) = +\infty$.
\end{defn}

Sometimes, the first waiting time $W_1$ is particular and it can have a law different from the other waiting times $(W_i)_{i\geq 2}$ and we also introduce the notion of delayed renewal processes.

\begin{defn}[Delayed renewal point process (DRPP)]
\label{defn:delayed_RPP}
Let $V,W$ be two non negative random variables such that $\mathbb{P}(V > 0), 
\mathbb{P}(W > 0) > 0$ and $W, V < +\infty$ almost surely. Let $(W_i)_{i\geq 1}$ be i.i.d. random variables with the same law as $W$ and independent from $V$. The delayed renewal point process with delay $V$ and waiting time $W$ 
is defined by
\[
\DRPP(V,W) \overset{\mathrm{(law)}}{=} \sum_{i=1}^{+\infty} \delta_{T_i},
\]
where $T_{i+1} - T_i = W_{i+1}$ for all $i\geq 1$ and $T_1 = V$. Note that $\DRPP(V,W) \in \mathcal{N}_{\mathbb{R}_+}^{\#\infty}$ almost surely.
\end{defn}

Finally, the last class of point process that we define here are compound Point process that are a common way to add mass to a specific atom. This construction is particularly important when one wants to take into account clusters that may occur at the same time.

\begin{defn}[Compound point process]
\label{defn:compound_point_process}
For a simple point process $P = \sum_{i=1}^{+\infty}\delta_{T_i}$, we define the 
associated compound point process $c(P)(\pi)$ of multiplicity $\pi$ (where $\pi$ is a 
probability distribution on $\mathbb{N}$) as
\[
c(P)(\pi) \overset{\mathrm{(law)}}{=} \sum_{i=1}^{+\infty}X_i\delta_{T_i},
\]
where $(X_i)_{i\geq 1}$ are i.i.d. random variables distributed according to $\pi$ and
independent of $(T_i)_{i\geq 1}$. 
\end{defn}

\begin{rem}
Note that this is not the standard approach to defining a compound process. Here, we utilize 
the structure of \(\mathbb{R}_+\) and the fact that the multiplicity \(\pi\) takes 
integer values and this will be sufficient in the following. However, compound point processes are typically constructed in a more general framework, particularly the compound Poisson process, as developed in \cite[Chapter 15]{LastPenrose18_LecturesOnPoissonProcess} for example. 
\end{rem}

\paragraph{All-point REPP property.} In this article, we are interested by finding the asymptotic behavior of the hitting REPP and the return REPP when targets shrink towards a point $x$ in the phase space. Since, we work with TMS, the natural targets are the cylinders defining the point $x$, \textit{i.e.} we consider $B_n = [x_0^{n-1}]$ for all $n\geq 1$. By the GRPF Theorem (Theorem \ref{thm:GRPF}, we know that, in the recurrent case, the invariant measures are non-atomic and thus $B_n$ are well defined asymptotically rare events. Beyond having results that would be true only for points $x$ in a set of full measure, we focus on finding a behavior for \textit{every} point $x$ in the system that we call the all-point REPP.

\begin{defn}[All-point REPP]
    We say that a TMS satisfies the \textit{all-point REPP} property if for every point $x$ of the phase space the hitting REPP and the return REPP associated to the sequence of asymptotically rare events $B_n := [x_0^{n-1}]$ converge weakly as $n$ goes to $+\infty$.
\end{defn}

Of course, beyond proving the existence of such limits, a more interesting point is to actually find the explicit limiting laws. Furthermore, by Theorem \ref{thm:GRPF_null_recurrent} and \cite[Corollary 2.1]{BT24_FractionalPoisson}, having the convergence for the hitting REPP or the return REPP is sufficient to have the convergence for both.





\section{Convergence at all points for positive recurrent potentials} 
\label{section:positive_recurrent_CMS}
In the positive recurrent case, we obtain the \textit{all-point REPP} with a dichotomy between non-periodic and periodic points, hence a result similar to the one for expanding systems \cite{LFFF16}. Except for the SFT case, this result for general positive recurrent TMS is new, so we provide a detailed treatment here. Even if the proof are easier and more standard, the positive recurrent case will serve as a preparatory step, offering insights for tackling the more delicate null recurrent case. Here, we observe convergence towards the Poisson point process for non periodic points and the compound Poisson point process for periodic points. We start by properly defining these point processes.

\begin{defn}[(Compound) Poisson point process ((C)PPP)]
    \label{defn:(compound)_Poisson_point_process}
    The standard Poisson point process $\PPP(\lambda)$ (of parameter $\lambda > 0$) is defined as a special renewal point process where the waiting times are following an exponential law (of parameter $\lambda$), \textit{i.e.} for $\mathcal{E}(\lambda)$ the exponential law of parameter $\lambda$,
    \[\PPP(\lambda) \overset{\text{(law)}}{=} \RPP(\mathcal{E}(\lambda)).\]
    The compound Poisson process $\CPPP(\lambda, \pi)$ of parameter $\lambda$ and multiplicity $\pi$ is simply $c(\PPP(\lambda))(\pi)$ with the notations from Definition \ref{defn:compound_point_process}. \\Note that in the particular case where $\pi = \geo(\theta)$ for some $0 < \theta \leq 1$ (\textit{i.e.} $\pi(k) = \theta (1 - \theta)^{k-1}$ for $k\geq 1$), $\CPPP(\lambda, \geo(\theta))$ can also be defined as a delayed renewal point process. Indeed, let $W_{1, \theta}(\lambda)$ be the non-negative random variable with distribution function
    \begin{align*}
        \mathbb{P}(W_{1, \theta}(\lambda) \leq t) = 1 - \theta + \theta(1 - e^{-\lambda t}),
    \end{align*}
    then, 
    \begin{align*}
        \CPPP(\lambda, \geo(\theta)) \overset{\text{(law)}}{=} \DRPP( \mathcal{E}(\lambda), W_{1, \theta}(\lambda)). 
    \end{align*}
\end{defn}

The following theorem establishes the \textit{all-point REPP} property for positive recurrent TMS with a dichotomy between periodic and non-periodic points.

\begin{thm}
    \label{thm:all-point_REPP_positive_recurrent_CMS}
    Let $(\Omega, T)$ be a topologically mixing TMS endowed with a positive recurrent potential $\phi$ where $\phi$ is Walters and $P_G(\phi) < +\infty$. Then, for all $x \in \Omega$ and $B_n = [x_0^{n-1}]$ for $n\geq 1$,
    \begin{enumerate}
        \item If $x$ is non periodic,
        \begin{align*}
            N_{B_n}^{\id} \xRightarrow[n \to +\infty]{\mu} \PPP \quad \text{and} \quad N_{B_n}^{\id} \xRightarrow[n \to +\infty]{\mu_{B_n}} \PPP.
        \end{align*}

        \item If $x$ is periodic of prime period $q$,
        \begin{align*}
            N_{B_n}^{\id} \xRightarrow[n \to +\infty]{\mu} \CPPP(\theta, \geo(\theta)) \quad \text{and} \quad N_{B_n}^{\id} \xRightarrow[n \to +\infty]{\mu_{B_n}} \RPP(W_{1,\theta}),
        \end{align*}
        where $\theta := 1 - \exp(S_q\phi(x) - qP_G(\phi))$ is the extremal index.
    \end{enumerate}
\end{thm}

The proof of the theorem can be found in Section \ref{section:proof_of_thm_all_point_REPP_positive_recurrent}. It takes advantage of \cite[Theorem 3.6 and 3.8]{Zwe22} giving sufficient conditions for a convergence towards Poisson point process and compound Poisson point processes (see Theorem \ref{thm:Roland_CPP_convergence_when_conditions_are_satisfied} for the formal statement) and we are able to show that these conditions are indeed satisfied for our targets. The crucial advantage in the finite measure case is that we are able to use inducing on a nice set---here it will be on one cylinders---to have even nicer properties on the system, such a fast mixing and bounded distortion to prove more easily the conditions to apply the previous theorems.

\section{Conditions of convergence for null-recurrent potentials} 
\label{section:quatitative_recurrence_null-recurrent_CMS}

While the positive recurrent case is more standard, the main focus of our work is on the null-recurrent case. As stated in \cite{BT24_FractionalPoisson} for cylinders in the Manneville-Pomeau map, the pivotal point process emerging as a limit is the fractional Poisson process.

\begin{defn}[Fractional Poisson Process]
\label{defn:Fractional_Poisson_Process}\leavevmode
The Fractional Poisson Process $\fPp_{\alpha}(\lambda)$ of parameters $\alpha \in 
(0,1]$ and $\lambda > 0$ is the renewal point process $\RPP(H_{\alpha}(\lambda))$ 
where $H_{\alpha}(\lambda)$ is a Mittag-Leffler law of the first type characterized 
by its Laplace transform
\[
\mathbb{E}[\exp(-sH_{\alpha}(\lambda))]
:= \frac{\lambda}{\lambda + s^{\alpha}} \,,\; \forall s \geq 0.
\]
\end{defn}

\begin{rem}
    There exists another equivalent definition of the fractional Poisson process through a random time reparameterization of a standard Poisson point process \cite{MNV11_FFPAndTheInverseStableSubordinator}. If $D_{\alpha}$ is an $\alpha$-stable subordinator (\textit{i.e.} it is a non-negative Lévy process with $\mathbb{E}[e^{-sD_{\alpha}(t)}] = \exp(-ts^{\alpha})$ for all $s, t \geq 0$) and $M_{\alpha}$ is an inverse stable subordinator (or Mittag-Leffler process) defined by 
    \begin{align*}
        M_{\alpha}(t) = D_{\alpha}^{\leftarrow}(t) := \inf \{u > 0 \;|\; D_{\alpha}(u) > t\} \quad \text{for $t \geq 0$,}
    \end{align*}
    Then $\fPp_{\alpha}(\lambda) \overset{\text(law)}{=} \PPP(\lambda) \circ M_{\alpha}$ where $\PPP(\lambda)$ and $M_{\alpha}$ are chosen independent.
\end{rem}

\begin{rem}
    Note that in the particular case $\alpha = 1$ we recover a standard Poisson point process.
\end{rem}

Going towards the \textit{all-point REPP} property for null-recurrent CMS, we are able to show that the fractional Poisson process is the asymptotic behavior for the hitting and return REPP associated to points that are infinitely recurrent for at least one letter of the alphabet $V$.

\begin{thm}
    \label{thm:HTS_infinite_FPP_infinitely_recurrent}
    Let $(\Omega, T)$ be a topologically mixing CMS endowed with a null recurrent potential $\phi$ where $\phi$ is Walters and $P_G(\phi) < +\infty$. Assume that the (non-decreasing) normalizing sequence $(a_n)_{n\geq 1}$ belongs to $\RV(\alpha)$ for some $0 < \alpha \leq 1$. Let $x \in \Omega$ and take $B_n := [x_0^{n-1}]$ for $n \geq 1$. Assume that
    \begin{itemize}
        \item $x$ is non periodic and there exists $v \in V$ such that $|\mathcal{O}(x) \cap [v]| = +\infty$.
    \end{itemize}
    Then,  
    \begin{align*}
        N_{B_n}^{\gamma} \xRightarrow[n \to +\infty]{\mathcal{L}(\mu)} \fPp_{\alpha}(\Gamma(1+\alpha)) \quad \text{and} \quad N_{B_n}^{\gamma} \xRightarrow[n \to +\infty]{\mu_{B_n}} \fPp_{\alpha}(\Gamma(1+\alpha)).
    \end{align*}
\end{thm}

The proof of can be found in Section \ref{section:proof_of_thm_FPP_infinite_recurrent_and_periodic}. Of course, the hypothesis on the point $x$ holds for a set of full measure (by conservativity it it even true that $|\mathcal{O}(x) \cap [v]| = +\infty$ \textit{for all} $v \in V$ $\mu$-almost everywhere), but its significance extends beyond this measure-theoretic statement. It is a combinatorial rather than a probabilistic property, which enables us to explicitly identify the points satisfying it. It is also worth emphasizing that in the assumption $|\mathcal{O}(x) \cap [v]| = +\infty$, no condition is imposed on the rate or frequency of returns. The fractional Poisson process nevertheless emerges as a limit, despite the fact that the returns to the symbol $v$ may become increasingly sparse. \\

The heuristic behind this result is the following. As in the finite measure case, one wishes to use inducing to recover nicer properties on the system. We can do the same procedure as in the finite case---induce on a $1$-cylinder---and have a Poisson point process as a limit for the induced system (preserving a probability). However, results from \cite{HLV05, FFTV16} ensuring equivalence between convergence for the induced system and the original system (and with the same limit) are not true in the infinite setting because the excursions are not integrable anymore. In this new setting, one would need to keep track of of the occupation time or local time at the induced set. This can be dealt with the functional Darling-Kac Theorem \cite[Theorem 6.1]{OwadaSamorodnitsky15} which ensures that the suitably normalized occupation time converges in law towards the Mittag-Leffler process. Then, the fractional Poisson process is obtained by the concatenation of the two previous limits, provided that the mechanisms are asymptotically independent. This would probably be the most natural way to tackle the problem and, in this approach, the main difficulty is to prove that we have the asymptotic independence. However, although this method could be a foreseeable way of proving the results, it seems that it could lack the crucial gap between sets of full-measure and sets satisfying combinatorial conditions.\\

For that, we are going to use the method developed in \cite{BT24_FractionalPoisson} (originated from \cite{RZ20} for the first return time) where the equivalence between convergence of the hitting REPP and convergence of the return REPP is used to construct sufficient conditions for convergence towards the fractional Poisson process. For that, we introduce a delay time $\tau_n$ that we have to wait before finding back a ``good measure''. Our proof method shows that we can choose our delay time in a universal way, by taking the first instant, after $n$, where a return towards a specific letter occurs. However, the delay time needs to be negligible in the limit and this is where some difficulty appears because we need to control it more precisely how it evolves as $n$ grows. The combinatorial conditions appears to tackle this issue at this point of the proof.\\

As it is also the main objective of this paper, we also aim to understand the behavior of the hitting and return REPP when $x$ is periodic or $|\mathcal{O}(x) \cap [v]| < +\infty$ for all $v \in V$. As first observed in \cite{BT24_FractionalPoisson} in the special easier case of the Manneville-Pomeau map, the behavior at periodic points changes from the behavior of the first class of points in the same way as in the probability preserving case, \textit{i.e.} there is a creation of clusters of returns due to the periodicity where the multiplicity is driven by a geometric law, itself driven by the extremal index. The compound fractional Poisson process $\cfPp_{\alpha}(\lambda, \pi)$ is hence defined as the compound point process $c(\fPp_{\alpha}(\lambda))(\pi)$. Again, as in Definition \ref{defn:(compound)_Poisson_point_process}, in the particular case where $\pi$ is $\geo(\theta)$, the compound fractional Poisson process can be expressed as a delayed renewal point process associated to the law $W_{\alpha, \theta}(\lambda)$ with distribution function
\begin{align*}
    \mathbb{P}(W_{\alpha, \theta}(\lambda) \leq t) = 1 - \theta + \theta \,\mathbb{P}(H_{\alpha}(\lambda) \leq t) \quad \text{for $t\geq 0$,}
\end{align*}
that is to say $\cfPp_{\alpha}(\lambda) \overset{\text{(law)}}{=} \DRPP(H_{\alpha}(\lambda), W_{\alpha, \theta}(\lambda))$. Thus, we obtain the following limits for periodic points.

\begin{thm}
    \label{thm:HTS_infinite_CFPP_periodic_points}
    Let $(\Omega, T)$ be a topologically mixing CMS endowed with a null recurrent potential $\phi$ where $\phi$ is Walters and $P_G(\phi) < +\infty$. Assume that the (non-decreasing) normalizing sequence $(a_n)_{n\geq 1}$ belongs to $\RV(\alpha)$ for some $0 < \alpha \leq 1$. Let $x \in \Omega$ and take $B_n := [x_0^{n-1}]$ for $n \geq 1$. Assume that
    \begin{itemize}
        \item $x$ is periodic of prime period $q$.
    \end{itemize}
    Then,  
    \begin{align*}
                N_{B_n}^{\gamma} \xRightarrow[n \to +\infty]{\mathcal{L}(\mu)} \cfPp_{\alpha}(\theta\Gamma(1+\alpha), \geo(\theta)) \quad \text{and} \quad N_{B_n}^{\gamma} \xRightarrow[n\to +\infty]{\mu_{B_n}} \RPP(W_{\alpha,\theta}(\theta\Gamma(1+\alpha)))\,.
    \end{align*}
    where $\theta := 1 - \exp(S_q\phi(x) - q P_G(\phi))$.
\end{thm}

The proof can be found in Section \ref{section:proof_of_thm_FPP_infinite_recurrent_and_periodic}. Finally, to get every point, it remains to study points which encounter each letter a finite number of time. Note that, depending on the graph structure of the CMS, this class of points may not even exist (like for the renewal shift, see the end of Section \ref{section:other_examples}) but in many meaningful examples they still exist (see Section \ref{section:all-point_REPP_for_examples}). For these particular points, we recognize that other limit laws, different from the fractional Poisson process or its compound versions, can appear as limits. Let us first define this class of new point process. They will be characterized as (delayed) renewal point process and thus we only need to define there respective waiting times.

\begin{defn}[Waiting random variables]
    \label{defn:waiting_random_variables_J_alpha}
    Let $\nu$ be a probability measure on $\mathbb{R}_+$ and $W$ be a random variable distributed according to $\nu$. We define the waiting random variables $J_{\alpha}(\nu)$ and $\widetilde{J}_{\alpha}(\nu)$ throughout their Laplace transform by 
    \begin{align*}
        \mathbb{E}[\exp(-s J_{\alpha}(\nu))] := \Big(\mathbb{E}[e^{-sW}] + s^{\alpha}\Gamma(1+\alpha)^{-1}\Big)^{-1}\,.
    \end{align*}
    and 
    \begin{align*}
        \mathbb{E}[\exp(-s \widetilde{J}_{\alpha}(\nu))] := \frac{\mathbb{E}[e^{-sW}]}{\mathbb{E}[e^{-sW}] + s^{\alpha}\Gamma(1+\alpha)^{-1}}\,.
    \end{align*}
\end{defn}

Note that, if $W \sim \nu$ is independent of $\widetilde{J}_{\alpha}(\nu)$, then $W + J_{\alpha}(\nu) \overset{\text{(law)}}{=} \widetilde{J}_{\alpha}(\nu)$. Furthermore, in the special case $W = 0$, we have $J_{\alpha}(\nu) \overset{\text{(law)}}{=} \widetilde{J}_{\alpha}(\nu) \overset{\text{(law)}}{=} H_{\alpha}(\Gamma(1+\alpha))$.\\

The next theorem provides a necessary and sufficient conditions to characterize the limit point process for both the hitting and return REPP.

\begin{thm}
    \label{thm:HTS_infinite_points_that_never_come_back_statements}
    Let $(\Omega, T)$ be a topologically mixing CMS endowed with a null recurrent potential $\phi$ where $\phi$ is Walters and $P_G(\phi) < +\infty$. Assume that the (non-decreasing) normalizing sequence $(a_n)_{n\geq 1}$ belongs to $\RV(\alpha)$ for some $0 < \alpha \leq 1$. Let $x \in \Omega$ and take $B_n := [x_0^{n-1}]$ for $n \geq 1$. Assume that
    \begin{itemize}
        \item $|\mathcal{O}(x) \cap [x_0] | < +\infty$.
    \end{itemize}
    Set $j_{x_0} := \max \{k \geq 0\;|\; T^kx \in [x_0]\}$. Then,
    \begin{align}
        \label{eq:convergence_law_tau_n_in_thm_statements}
        \gamma(\mu(B_n))\,r_{[x_0]} \circ T^{j_{x_0}} \xRightarrow[n \to +\infty]{\mu_{B_n}} W \sim \nu
    \end{align}
    if and only if 
    \begin{align*}
         N_{B_n}^{\gamma} \xRightarrow[n\to +\infty]{\mu_{B_n}} \RPP(\widetilde{J}_\alpha(\nu))\quad \text{and} \quad N_{B_n}^{\gamma} \xRightarrow[n\to +\infty]{\mathcal{L}(\mu)} \DRPP(J_{\alpha}(\nu), \widetilde{J}_\alpha(\nu))
    \end{align*}
    where $J_{\alpha}(\nu)$ and $\widetilde{J}_{\alpha}(\nu)$ are defined in Definition \ref{defn:waiting_random_variables_J_alpha}. 
\end{thm}

We provide a detailed proof in Section \ref{section:proof_points_that_never_come_back}. Note that in Theorem \ref{thm:HTS_infinite_points_that_never_come_back_statements}, we only assumed $|\mathcal{O}(x) \cap [x_0] | < +\infty$ and not $|\mathcal{O}(x) \cap [v] | < +\infty$ for all $v \in V$ and thus the equivalence is more general. Furthermore, this is not incompatible with Theorem \ref{thm:HTS_infinite_FPP_infinitely_recurrent} because $\fPp_{\alpha}(\Gamma(1+\alpha)) \overset{\text{(law)}}{=} \RPP(\widetilde{J}_{\alpha}(0)) \overset{\text{(law)}}{=} \DRPP(J_{\alpha}(0), \widetilde{J}_\alpha(0))$ and, hence, if $|\mathcal{O}(x) \cap [v]| = +\infty$ for some $v \in V \,\backslash\, \{x_0\}$, then we have $W = 0$ almost surely. \\

The proof of this result, relies on a generalization of the sufficient conditions so that we allow for non-negligible delay times. We thus need to define new assumptions to be able to tackle these new phenomenons.\\

\noindent  \textbf{Assumptions $B_{\alpha}(\nu)$.} A sequence $(B_n)_{n\geq 1} \in \mathscr{B}^{\mathbb{N}}$ of asymptotically rare events satisfies $B_{\alpha}(\nu)$ if it satisfies the following conditions : \index{conditions!$B_{\alpha}(\nu)$}
\begin{enumerate}[label = $B\arabic*_{\alpha}(\nu)$, leftmargin=*]
\item \label{cond_OtherLimits:good_density_after_tau_n} There exist a sequence of 
measurable functions $\tau_n : B_n \to \mathbb{N}$ and a compact subset $\mathcal{U}$ 
of $L^1(\mu)$ such that for all $n \geq 1$, for all $k \geq 0$ such that $B_n \cap \{\tau_n = k\} \neq \emptyset$, 
\[
\widehat{T}^{k}\left(\frac{\mathbf{1}_{B_n \cap \{\tau_n = k\}}}{\mu(B_n \cap \{\tau_n = k\})}\right)\in 
\mathcal{U}.
\]
\item \label{cond_OtherLimits:tau_n_small_enough} The sequence $(\tau_n)_{n\geq 0}$
satisfies $\gamma(\mu(B_n))\,\tau_n \xRightarrow[n\to +\infty]{\mu_{B_n}} W$, where 
$\gamma$ is the scaling for $T$ is defined from the normalizing sequence $(a_n)_{n\geq 0}$ by \eqref{eq:defn_gamma}.
\item \label{cond_OtherLimits:cluster_compatible_tau_n_no_cluster_from_Q} The sequence 
$(B_n)_{n\geq 0}$ is such that $\mu_{B_n}(\lr_{B_n} \leq \tau_n) \xrightarrow[n\to 
+\infty]{} 0.$
\end{enumerate}

Then, from these new conditions, we can prove an abstract theorem proving convergence towards the associated renewal processes.

\begin{thm}
\label{thm:sufficient_conditions_convergence_other_point_processes}
Let $(X,\mathscr{B},\mu, T)$ be a PDE CEMPT with its normalizing sequence $(a_n)_{n\geq1}$ belonging to $\RV(\alpha)$ for some $0 < \alpha \leq 1$. Let $Y \in \mathscr{B}$ be a uniform set with $\mu(Y) < +\infty$. Let $(B_n)_{n\geq 1}$ be a sequence of asymptotically rare events with $B_n \subset Y$ for all $n\geq 1$ and $B_{\alpha}(\nu)$ for some non negative random variable $W$. Then,
    \begin{align*}
        N_{B_n}^{\gamma} \xRightarrow[n\to +\infty]{\mathcal{L}(\mu)} \DRPP(J_{\alpha}(\nu), \widetilde{J}_\alpha(\nu)) \quad \text{and} \quad 
        N_{B_n}^{\gamma} \xRightarrow[n\to +\infty]{\mu_{B_n}} \RPP(\widetilde{J}_\alpha(\nu)).
    \end{align*}
\end{thm}

\begin{rem}
    In what follows, we apply this abstract theorem will be applied to the particular case of null-recurrent CMS. Nonetheless, it remains valid in the broader setting of PDE dynamical systems and may provide a fruitful framework for the study of other classes of systems. Moreover, while our main focus will be on rare events corresponding to cylinder sets, the theorem can be applied to more general types of rare events, some of which will be discussed in Section \ref{section:convergence_embedded_SFT}. 
\end{rem}

Theorem \ref{thm:sufficient_conditions_convergence_other_point_processes} allows us to prove Theorem \ref{thm:HTS_infinite_points_that_never_come_back_statements} and is itself proven in Section \ref{section:proof_CMS_allowing_for_non_zero_times}. Nevertheless, the necessary and sufficient conditions does not provide any information on the limits of the return or hitting REPP. It reformulates it into whether we are able to identify the limit \eqref{eq:convergence_law_tau_n_in_thm_statements} in the CMS context. 

\section{Possible limit laws}
\label{section:possible_limit_laws}

Here we focus on the possible limits for \eqref{eq:convergence_law_tau_n_in_thm_statements} and thus of the special point processes that can emerge as limits for the convergence in law of the hitting and return REPP. It turns out that the regular variation hypothesis constraints the class of potential limits, in the sense that they must be dominated in some sense by the Pareto distribution of index $\alpha$, where $\alpha$ is the regularly varying parameter.

\begin{defn}[Pareto distribution] \index{Pareto distribution}
    A non negative random variable $W$ is said to follow a Pareto law of index $\alpha > 0$ and parameter $\lambda > 0$ (we will write $W \sim \Par_{\alpha}(\lambda)$) if 
    \begin{align*}
        \mathbb{P}(W > t) = 1 \wedge \lambda t^{-\alpha} \quad \forall t \geq 0.
    \end{align*}
\end{defn}

\noindent Then, we define the following set of distribution $\mathcal{G}_{\alpha}$ whose tails are dominated by the Pareto law (of parameter $\sin(\pi \alpha)/(\pi \alpha)$). Letting $\mathbb{P}(\mathbb{R}_+)$ be the set of probability measures on $\mathbb{R}_+$, we set
\begin{align*}
    \mathcal{G}_{\alpha} := \bigg\{ \nu \in &\,\mathbb{P}(\mathbb{R}_+) \;|\; \forall s>t>0, \; \nu(]t, +\infty[) \leq 1 \wedge \frac{\sin(\pi \alpha)}{\pi\alpha}\, t^{-\alpha}\; \\
    &\text{and}\; \nu(]t, s]) \leq \frac{\sin(\pi\alpha)}{\pi\alpha}\big(t^{-\alpha} - s^{-\alpha}\big)\bigg\}.
\end{align*}

\begin{rem}
    In particular, for $\nu \in \mathcal{G}_{\alpha}$, if $\nu(\{x\}) > 0$ for some $x \geq 0$, then $x = 0$, \textit{i.e.} the only possible atom is $0$.
\end{rem}

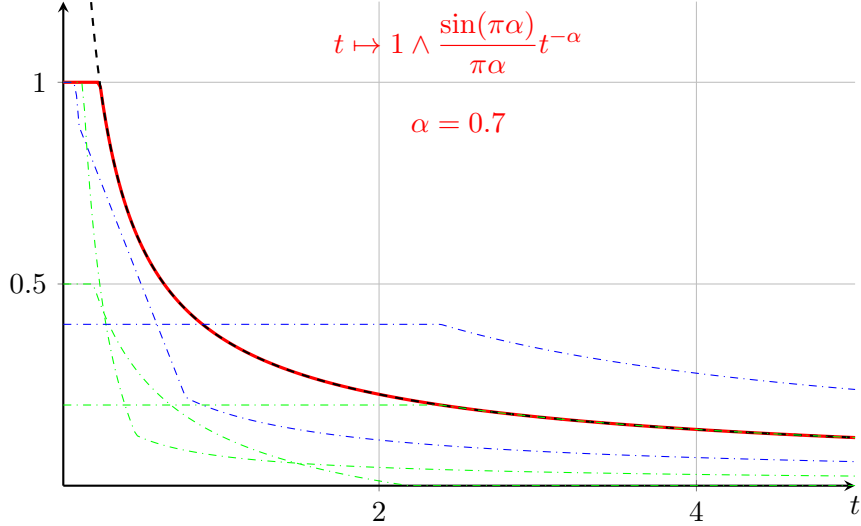
\begin{figure}[h]
    \centering
    \begin{tikzpicture}
      \begin{axis}[
        width=12cm,
        height=8cm,
        domain=0.01:5,
        samples=300,
        xlabel={$t$},
        ylabel={},
        axis lines=middle,
        ymin=0,
        ymax=1.2,
        xtick={0,2,4,6,8},
        ytick={0,0.5,1},
        thick,
        grid=both,
        every axis x label/.style={at={(current axis.right of origin)}, anchor=north},
        every axis y label/.style={at={(current axis.above origin)}, anchor=east},
      ]
        \def\zeta{0.7}
        \addplot[
          red,
          very thick,
        ]
        {min(1, (sin(deg(pi*\zeta))/(pi*\zeta))*x^(-\zeta))};
        
        \node at (axis cs:2.5,1.1) {\textcolor{red}{$t \mapsto 1 \wedge \dfrac{\sin(\pi\alpha)}{\pi\alpha} t^{-\alpha}$}};
        \node at (axis cs:2.5,0.9) {\textcolor{red}{$\alpha = \zeta$}};
        
        \addplot[
          black,
          thick,
          dashed,
        ]
        {(sin(deg(pi*\zeta))/(pi*\zeta))*x^(-\zeta))};

        \addplot[
          green,
          thin,
          dashdotted,
        ]
        {min(1, max( (sin(deg(pi*\zeta))/(pi*\zeta))*x^(-\zeta) - 0.5, (sin(deg(pi*\zeta))/(pi*\zeta))*x^(-\zeta)*0.2))};
        \addplot[
          green,
          thin,
          dashdotted,
        ]
        {min(0.2, (sin(deg(pi*\zeta))/(pi*\zeta))*x^(-\zeta)};
        
        \addplot[
          green,
          thin,
          dashdotted,
        ]
        {max(0, min(0.5, (sin(deg(pi*\zeta))/(pi*\zeta))*(x+ 0.2)^(-\zeta) - 0.2))};
        
        \addplot[
          blue,
          thin,
          dashdotted,
        ]
        {min(0.4, 2*(sin(deg(pi*\zeta))/(pi*\zeta))*x^(-\zeta)};

        \addplot[
          blue,
          thin,
          dashdotted,
        ]
        {min( min(1,(sin(deg(pi*\zeta))/(pi*\zeta))*x^(-\zeta) ), max(1-x, 0.5 *(sin(deg(pi*\zeta))/(pi*\zeta))*x^(-\zeta)))};
      \end{axis}
    \end{tikzpicture}
\caption{Examples of tails of laws belonging to $\mathcal{G}_{\alpha}$ (in green) and tails of laws that are not in $\mathcal{G}_{\alpha}$ (in blue). To be valid, they must be dominated (in some sense) by the tail of the Pareto law of index $\alpha$ and parameter $\sin(\pi \alpha)/(\pi\alpha)$ (in red).}
\label{fig:enter-label}
\end{figure}

It turns out that the set of distributions $\mathcal{G}_{\alpha}$ is the only set of reachable distributions for the limit \eqref{eq:convergence_law_tau_n_in_thm_statements}.

\begin{prop}
    \label{prop:only_possible_limits_for_waiting_times_in_G}
    Let $x \in X$ and $B_n := [x_0^{n-1}]$ for $n\geq 1$. Assume that $|\mathcal{O}(x) \cap [x_0]| < +\infty$ and set $j_{x_0} := \max \{k\geq 0\;|\; T^kx\in[x_0]\}$.  If 
    \begin{align*}
        \gamma(\mu(B_n))\, r_{[x_0]} \circ T^{j_{x_0}} \xRightarrow[n \to +\infty]{\mu_{B_n}} W \sim \nu \,,
    \end{align*}
    for some probability $\nu$, then $\nu \in \mathcal{G}_{\alpha}$.
\end{prop}

Hence, we get a class of reachable limit distributions. The next result shows that this class of distributions in sharp. Indeed, for every $\nu \in \mathcal{G}_{\alpha}$, it is always possible to find a null-recurrent CMS that will see this distribution emerge for at least one point in the phase space. 

\begin{prop}
    \label{prop:Every_waiting_time_achievable_when_CMS_well_chosen}
    Let $\nu \in \mathcal{G}_{\alpha}$ and a normalizing sequence $(a_n)_{n \geq 1} \in \RV(\alpha)$ with $0 <\alpha < 1$. Then, there exists a null-recurrent CMS $(\Omega, T)$ endowed with a null-recurrent potential $\phi$ that is Walters and with $P_G(\phi) < +\infty$ such that for some $x \in \Omega$ and $B_n := [x_0^{n-1}]$ for $n\geq 1$, 
    \begin{align*}
        \gamma(\mu(B_n))\, r_{[x_0]} \circ T^{j_{x_0}}\xRightarrow[n \to +\infty]{\mu_{B_n}} W \sim \nu \,.
    \end{align*}
\end{prop}

In particular, for the point considered, Theorem \ref{thm:HTS_infinite_points_that_never_come_back_statements} implies 
\begin{align*}
    N_{B_n}^{\gamma} \xRightarrow[n \to +\infty]{\mu_{B_n}} \RPP(\widetilde{J}_{\alpha}(\nu)) \quad \text{and} \quad N_{B_n}^{\gamma} \xRightarrow[n \to +\infty]{\mathcal{L}(\mu)} \DRPP(J_{\alpha}(\nu), \widetilde{J}_{\alpha}(\nu)).
\end{align*}
In fact, the examples we construct are even null recurrent Markov chains (\textit{i.e.} the potential $\phi$ is Markovian, see more about the connection between CMS with Markovian potential and Markov chains in Section \ref{section:Connection_with_Markov_Chains}). The proof relies on a specific construction of tower-renewal based graph that allows us to build any particular example (see Figure \ref{fig:HoC-Renewal_structure} in Section \ref{section:proof_possible_limit_laws}). Propositions \ref{prop:only_possible_limits_for_waiting_times_in_G} and \ref{prop:Every_waiting_time_achievable_when_CMS_well_chosen} are proven in Section \ref{section:proof_possible_limit_laws}.

\begin{rem}[Non convergence]
    In fact, our construction can give a more general result than Proposition \ref{prop:only_possible_limits_for_waiting_times_in_G} as, keeping the same construction as in the proof, it is possible to build Markov Chains such that for a point $x$ there is no convergence of the REPP $N_{B_n}^{\gamma}$ associated to $B_n = [x_0^{n-1}]$. Furthermore, fixing $\mathcal{G}' \subset \mathcal{G}_{\alpha}$ we can construct a point $x$ such that the accumulation laws of $\gamma(\mu(B_n))r_{[0]}$ under $\mu_{B_n}$ are exactly the laws $W'$ such that $\overline{F_{W'}} \in \mathcal{G}'$.\\
    In particular, in the context of null recurrent CMS, it can be possible to find shrinking cylinders to a point such that there is no convergence of the REPP which is impossible for positive recurrent CMS.
\end{rem}

Propositions \ref{prop:only_possible_limits_for_waiting_times_in_G} and \ref{prop:Every_waiting_time_achievable_when_CMS_well_chosen} show that in all generality of null-recurrent CMS, it is impossible to obtain a universal \textit{all-point REPP} property and the best results achievable in general are Theorems \ref{thm:HTS_infinite_CFPP_periodic_points} and \ref{thm:HTS_infinite_CFPP_periodic_points}.

\section{All-point REPP property for paradigmatic examples} 
\label{section:all-point_REPP_for_examples}
Although a general \textit{all-point REPP} property cannot be formulated for arbitrary null-recurrent CMS, our framework enables the analysis of specific systems—namely, null-recurrent CMS that have been more extensively studied in the literature. We present the results for different paradigmatic examples of null-recurrent CMS, frequently appearing in the literature. The two main ones are the ones built on an House of Cards type graph and $\mathbb{Z}$-extensions of subshifts of finite type. While both satisfy the \textit{all-point REPP} property, the behaviors and the limit obtained are different.

\subsection{House of Cards type CMS} 

We start by the House of Cards type CMS. We define the CMS on the alphabet $V = \mathbb{N}$ and associated to the transition matrix
\begin{align*}
    A(v,w) =  \bigg\{\begin{array}{cc}
        1 & \text{if} \; w = 0 \; \text{or}\; w = v + 1 \\
        0 & \text{else}. 
    \end{array} 
\end{align*}

\begin{figure}[h]
\centering
\begin{tikzpicture}[
  state/.style={circle, draw, minimum size=7mm, inner sep=1pt, font=\small},
  >=Stealth,
  node distance=10mm and 12mm
]

\node[state] (H0) {0};
\node[state, right=of H0] (V1) {\(1\)};
\node[state, right=of V1] (V2) {\(2\)};
\node[right=1cm of V2] (Vdots) {\(\cdots\)};
\node[state, right=1cm of Vdots] (Vn) {\(n\)};
\node[right=1cm of Vn] (Vdots2) {\(\cdots\)};

\draw[->, bend right=40] (V1) to (H0);
\draw[->, bend right=40] (V2) to (H0);
\draw[->, bend right=40] (Vn) to (H0);
\draw[->, bend right=45] (Vdots2) to (H0);

\draw[->] (H0) edge[loop left]  (H0);

\foreach \x/\y in {H0/V1, V1/V2, V2/Vdots, Vdots/Vn, Vn/Vdots2} {
  \draw[->] (\x) -- (\y);
}


\end{tikzpicture}
\caption{The House of Cards graph.}
\label{fig:HoC_structure}
\end{figure}
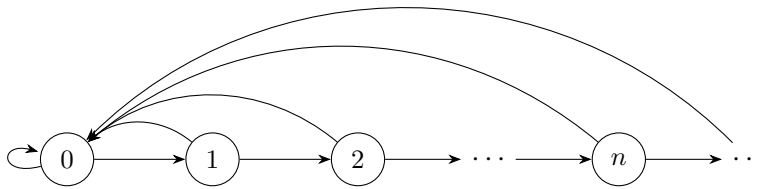

\noindent In this dynamical system, the only way to avoid coming back to $0$ is to keep climbing the tower. Hence, the set of points $\mathcal{D}$ such that their orbit encounters $0$ only a finite number of time is 
\begin{align*}
    \mathcal{D} := \mathcal{O}(x^{up}) \cup \bigcup_{k\geq 1} T^{-k}\{x^{up}\}
\end{align*}
where $x^{up} := (k)_{k\geq 0}$ is the point starting at $0$ and climbing indefinitely the tower. Then, such a system endowed with a null-recurrent potential (with some regular variation hypothesis) satisfy an \textit{all-point REPP} with a \textit{trichotomy} between periodic points, points in $\mathcal{D}$ and the remaining points.

\begin{thm}
    \label{thm:all_points_HoC_structure}
    Let $(\Omega, T)$ be a House of Cards shift and $\phi$ a potential satisfying Walters' condition and such that $P_G(\phi) < +\infty$. Assume furthermore that for some $v \in V$ there exists $\alpha \in (0,1]$ such that 
    \begin{align*}
        \frac{1}{\mu_{\phi}[v]} \sum_{k = 1}^{n} \lambda_{\phi}^{-k}Z_k(\phi, v) \in \RV(\alpha).
    \end{align*}
    \noindent Let $\nu$ be the Pareto law of index $\alpha$ and parameter $\sin(\pi \alpha)/(\pi\alpha)$. Then, for every $x \in \Omega$ and $B_n := [x_0^{n-1}]$
    \begin{enumerate}
        \item if $x$ is non periodic and $x \notin \mathcal{D}$,
        \begin{align*}
            N_{B_n}^{\gamma} \xRightarrow[n \to +\infty]{\mathcal{L}(\mu)} \fPp_{\alpha}(\Gamma(1+\alpha)) \quad \text{and} \quad N_{B_n}^{\gamma} \xRightarrow[n \to +\infty]{\mu_{B_n}} \fPp_{\alpha}(\Gamma(1+\alpha)).
        \end{align*}
        \item if $x$ is periodic of prime period $p$, 
        \begin{align*}
                N_{B_n}^{\gamma} \xRightarrow[n \to +\infty]{\mathcal{L}(\mu)} \cfPp_{\alpha}( \theta\Gamma(1+\alpha), \geo(\theta)) \quad \text{and} \quad N_{B_n}^{\gamma} \xRightarrow[n \to +\infty]{\mu_{B_n}} \RPP(W_{\alpha,\theta}(\theta\Gamma(1+\alpha))).
        \end{align*}
        with $\theta = 1 - \exp(S_p\phi(x) - pP_G(\phi))$.
        \item if $x \in \mathcal{O}(x^{up})$ (\textit{i.e.}, $x = (j + k)_{k\geq 0}$ for some $j\geq 0$),
        \begin{align*}
            N_{B_n}^{\gamma} \xRightarrow[n \to +\infty]{\mu_{B_n}} \RPP(\widetilde{J}_{\alpha}(\nu)) \quad \text{and} \quad N_{B_n}^{\gamma} \xRightarrow[n \to +\infty]{\mathcal{L}(\mu)} \DRPP(J_{\alpha}(\nu), \widetilde{J}_{\alpha}(\nu)).
        \end{align*}
        \item if $x \in T^{-k}\{x^{up}\}$ for some $k \geq 0$, we set $\nu_x$ the Pareto law of parameter $\exp(S_k\phi_*(x))\, \allowbreak \sin(\pi \alpha)/(\pi\alpha)$. Then,
        \begin{align*}
            N_{B_n}^{\gamma} \xRightarrow[n \to +\infty]{\mu_{B_n}} \RPP(\widetilde{J}_{\alpha}(\nu_x)) \quad \text{and} \quad N_{B_n}^{\gamma} \xRightarrow[n \to +\infty]{\mathcal{L}(\mu)} \DRPP(J_{\alpha}(\nu_x), \widetilde{J}_{\alpha}(\nu_x)).
        \end{align*}
    \end{enumerate}
\end{thm}

We observe that points in the positive orbit of $x$ exhibit the same asymptotic behavior as $x^{up}$ itself, whereas for its preimages, the parameter of the Pareto distribution governing the point process varies according to the specific point $x$ considered. Both exhibit a limit distribution different from the fractional Poisson process. It lies in the local behavior near $x^{up}$. Indeed, when the system is close to $x^{up}$, the entrance into the target set displays a very specific structure—trajectories tend to climb the tower for a very long time, which in turn generates a distinct limiting point process.\\

Now, if $x$ belongs to the positive orbit of $x^{up}$, maintaining this prolonged climbing of the tower requires following the sequence of symbols associated to $x$ for a sufficiently long time. In this case, the difference between starting at floor $k$ or at floor $0$ becomes negligible in the limit.\\

In contrast, if $x$ is a preimage of $x^{up}$, following the orbit of $x$ for a long time still induces a climb up the tower, but this can occur without returning to a neighborhood of $x$, since the system may enter the tower through another preimage. Consequently, the contribution of $x$ becomes diluted among its preimages, which explains the emergence of the additional parameter $\exp(S_m\phi_*(x)) < 1$ in this case.

\subsection{$\mathbb{Z}$-extensions} 

A paradigmatic family of infinite measure preserving systems are $\mathbb{Z}$-extensions over a probability-preserving dynamical system. This notion extends beyond our scope of symbolic dynamics and have various applications in different directions in ergodic theory \cite{Pene_livre_infini}. Let us start by defining a $\mathbb{Z}$-extension in the most general way.

\begin{defn}[$\mathbb{Z}$-extension]
    \label{defn:Z_extension}
    Let $(\Omega, \sigma, \nu)$ be a probabilistic dynamical system and $h : \Omega \to \mathbb{Z}$ be an observable ($h$ is called the jump function). We define the dynamical system on $X = \Omega \times \mathbb{Z}$ associated with the dynamics
    \begin{align*}
        T : (\omega, z) \mapsto (\sigma \omega, z + h(\omega)).
    \end{align*}
    Note that the (infinite) measure $\mu := \nu \otimes \mathfrak{m}$, where $\mathfrak{m}$ is the counting measure on $\mathbb{Z}$, is preserved by the map $T$.
\end{defn}

For us, we will consider $\mathbb{Z}$-extensions over TMS. Note first that by definition, a $\mathbb{Z}$-extension over a TMS is not directly constructed as a real TMS. However, it is topologically conjugated to one and thus we will still be able to study it as one (see Appendix \ref{section:proofs_all-point_REPP_examples}). \\

Here, we will restrict ourselves to $\mathbb{Z}$-extension over strong positive recurrent TMS with a Gibbs measure and nice jump functions, which is mandatory to fall in the scope of our study, \textit{i.e.} to have PDE CEMPT with a regularly varying normalizing sequence. Thus, we consider a topologically mixing TMS with the Big Image and Preimages (BIP) property, \textit{i.e.} there exists $N > 0$ and a finite set of letters $(v_i)_{1\leq i \leq N} \in V^N$ such that for all $v \in V$, there exist $1\leq i,j \leq N$ such that $A(v_i,v) = A(v,v_j) = 1$. In particular, this property is trivially satisfied for SFTs. We will also assume that our jump function $h$ is non-arithmetic, \textit{i.e.} $h$ is not cohomologous in $L^2(\nu)$ to a sub-lattice valued function. This ensures that the CMS build for the $\mathbb{Z}$-extension is topologically mixing and avoid some technicalities.

\begin{thm}
    \label{thm:all_points_Z_extension}
    Let $(X, T)$ be a $\mathbb{Z}$-extension over a TMS $(\Omega, \sigma)$. Assume that $(\Omega, \sigma)$ has the BIP property and is associated to a weakly Hölder potential $\phi$ with $\var_1(\phi) < +\infty$ and $P_G(\phi) < +\infty$ (note $\nu$ the associated invariant probability). Assume that the jump function $h \in L^1(\nu)$, $\mathbb{E}_{\nu}[h] = 0$, non-arithmetic, locally Hölder and is in one of the following cases:
    \begin{itemize}
        \item $h \in L^2(\nu)$. In this case, set $\alpha = 1/2$.
        \item $t \mapsto \nu(|h| > t) \in \RV(-\beta)$ with $\beta \in (1, 2]$. In this case, set $\alpha = 1 - 1/\beta$.
    \end{itemize}
    Let $x = (\omega, z) \in X$ and consider $B_n := [\omega_0^{n-1}] \times \{z\}$ for $n\geq 1$. Then,
    \begin{enumerate}
        \item If $x$ is non periodic, 
        \begin{align*}
            N_{B_n}^{\gamma} \xRightarrow[n \to +\infty]{\mathcal{L}(\mu)} \fPp_{\alpha}(\Gamma(1+\alpha)) \quad \text{and} \quad N_{B_n}^{\gamma} \xRightarrow[n \to +\infty]{\mu_{B_n}} \fPp_{\alpha}(\Gamma(1+\alpha)).
        \end{align*}
        \item If $x$ is periodic of prime period $p$,
        \begin{align*}
                N_{B_n}^{\gamma} \xRightarrow[n \to +\infty]{\mathcal{L}(\mu)} \cfPp_{\alpha}( \theta\Gamma(1+\alpha), \geo(\theta)) \quad \text{and} \quad N_{B_n}^{\gamma} \xRightarrow[n \to +\infty]{\mu_{B_n}} \RPP(W_{\alpha,\theta}(\theta\Gamma(1+\alpha))).
        \end{align*}
        where $\theta := 1 - \exp(S_p\phi(\omega)- pP_G(\phi))$.
    \end{enumerate}
\end{thm}

Note that Theorem~\ref{thm:all_points_Z_extension} provides a substantial generalization of~\cite[Theorem~3.10]{PS23}, which was restricted to points in a set of full measure and concerned only $\mathbb{Z}$-extensions over subshifts of finite type.

\begin{rem}
    Note that $\gamma$ is not made explicit in the statement of the Theorem but is computed from the underlying normalizing sequence coming from the PDE property that holds for the $\mathbb{Z}$-extension. Explicit formulas of the normalizing sequences can be find in \cite[Table 3.1]{Thomine_These}.
\end{rem}

\begin{rem}
    The periodic points are the periodic points for the $\mathbb{Z}$-extension $(X, T)$ and not the periodic point for the TMS we are building the extension from. Indeed, if $x = (\omega, z)$ is a periodic point for $(X, T)$, then $\omega$ is a periodic point for $(\Omega, \sigma)$ but the converse is not true. However, the extremal index $\theta$ is still computed only from $\omega \in \Omega$.
\end{rem}

We observe that, in the case of $\mathbb{Z}$-extensions, we obtain an \textit{all-point REPP} exhibiting a \textit{dichotomy}, analogous to that observed in the finite setting (Theorem \ref{thm:all-point_REPP_positive_recurrent_CMS}) and distinct from the behavior arising in the previous House of Cards example (Theorem \ref{thm:all_points_HoC_structure}). Although it is straightforward to construct points that fail to satisfy the combinatorial condition of Theorem \ref{thm:HTS_infinite_FPP_infinitely_recurrent}, and one might therefore expect new limiting behavior for such points, this is in fact not the case.\\

The reason is that there exist ``too many'' of these points, in the sense that a large variety of admissible paths visit each symbol only finitely many times. Consequently, the relative contribution of any individual point becomes negligible in the limit. As a result, the long-term behavior induced by visits to these exceptional targets disappears asymptotically, and the system exhibits once again the standard fractional Poisson process limit observed for the other points. A more straightforward application of this phenomenon can be observed in the case of the tree House of Cards in Section \ref{section:other_examples}. In the $\mathbb{Z}$-extension case, our method is more subtle. We build upon Proposition \ref{prop:only_possible_limits_for_waiting_times_in_G} to make a proof by contradiction which provides a new sufficient condition for convergence towards fractional Poisson process in the case where every state is visited a finite number of times.

\begin{prop}
    \label{prop:sufficient_condition_for_0_delay_time_finitely_recurrent}
    Let $x \in X$ be such that $|\mathcal{O}(x) \cap [v]| <+\infty$ for all $v \in V$ and $B_n := [x_0^{n-1}]$ for all $n\geq 1$. Assume that 
    \begin{align}
        \label{eq:suff_cond_for_0_delay_time}
        S_n \phi_*(x) \xrightarrow[n \to +\infty]{} - \infty.
    \end{align}
    Then,
    \begin{align*}
                N_{B_n}^{\gamma} \xRightarrow[n \to +\infty]{\mathcal{L}(\mu)} \fPp_{\alpha}(\Gamma(1+\alpha)) \quad \text{and} \quad N_{B_n}^{\gamma} \xRightarrow[n \to +\infty]{\mu_{B_n}} \fPp_{\alpha}(\Gamma(1+\alpha)).
    \end{align*}
\end{prop}

\begin{rem}
    By definition, we have $\phi_* \leq 0$ and thus $S_{n}\phi_* \leq 0$ is non increasing. This is due to the identity $1 = L_{\phi_*}1 = \sum_{Ty = x}e^{\phi_*(y)}$.
\end{rem}

\begin{rem}
\label{rem:another_sufficient_condition_delay_time_0_imeges}
Note that hypothesis \eqref{eq:suff_cond_for_0_delay_time} can also be rewritten as
\begin{align}
        \label{eq:rem:another_sufficient_condition_delay_time_0_imeges}
        \lim_{m \to +\infty} \lim_{n \to +\infty} \frac{\mu[x_0^{n-1}]}{\mu[x_m^{n-1}]} = 0.
\end{align}
\end{rem}

In the $\mathbb{Z}$-extension case, we are able to show that indeed this condition is satisfied for all points such that $|\mathcal{O}(x) \cap [v]| <+\infty$ for all $v \in V$. However, the condition is not necessary and the next case of null-recurrent CMS based on a graph with multiple towers exhibits some examples.

\subsection{Other examples}
\label{section:other_examples}

\paragraph{Multiple towers.}
\label{subsubsection:multiple_towers}
Our analysis also extends to generalizations of the House of Cards structure in which multiple (possibly countably many) towers may originate from the same symbol. This construction also reflects the broader architecture of Markov-AFN maps, which may feature several indifferent fixed points \cite{Zwe98}.\\

In the House of Cards structure, the only possibility is to climb over one tower. However, we can generalize the model to including multiple (possibly a countable number) of towers starting from $0$. So we define the CMS $(\Omega, T)$ on the countable alphabet $V := \{0\} \cup \{i.j = (i,j),\; i\geq 1, \, j\geq 1\}$ associated with the transition matrix for
\begin{align*}
    A(v,w) =  \left\{\begin{array}{cc}
        1 & \text{if} \; w = 0\\
        1 & \text{if} \; v = 0 \; \text{and} \; w = (i, 1)\\
        1 & \text{if} \; v = (i,j)\; \text{and} \; w = (i, j+1)\\
        0 & \text{else}. 
    \end{array} \right.
\end{align*} \index{CMS!multiple towers}
The CMS thus defined can be visualized on Figure \ref{fig:Multiple_towers_structure}. We endow it with a null recurrent potential $\phi$ (in particular $\phi$ is Walters and $P_G(\phi) < +\infty$) and we assume that
\begin{align*}
    \frac{1}{\mu[0]} \sum_{k = 1}^n e^{-kP_G(\phi)}Z_k(\phi, 0) \in \RV(\alpha).
\end{align*}
In particular, it implies that 
\begin{align*}
    t \mapsto \mu_{[0]}(r_{[0]} > n) \in \RV(-\alpha). 
\end{align*}
To be able to state a result for every point $x \in \Omega$, we will need a bit more properties on each tower. For all $i\geq 1$, let $\ell_i := \mathbf{1}_{T^{-1}[i.1]}\,r_{[0]}$ be the return time to $0$ through the tower $i$. We assume that there exists a sequence $(p_i)_{i\geq 1}$ such that 
\begin{align}
    \label{eq:hypothesis_regular_variation_on_each_tower}
    \frac{\mu_{[0]}(\ell_i > n)}{\mu_{[0]}(r_{[0]} > n)} \xrightarrow[n \to +\infty]{} p_i.
\end{align}
Note that we necessarily have $\sum_{i\geq 1} p_i = 1$ and it is possible to have $p_i = 0$. If $p_i > 0$, it implies that $\mu_{[0]}(\ell_i > n) \in \RV(-\alpha)$.

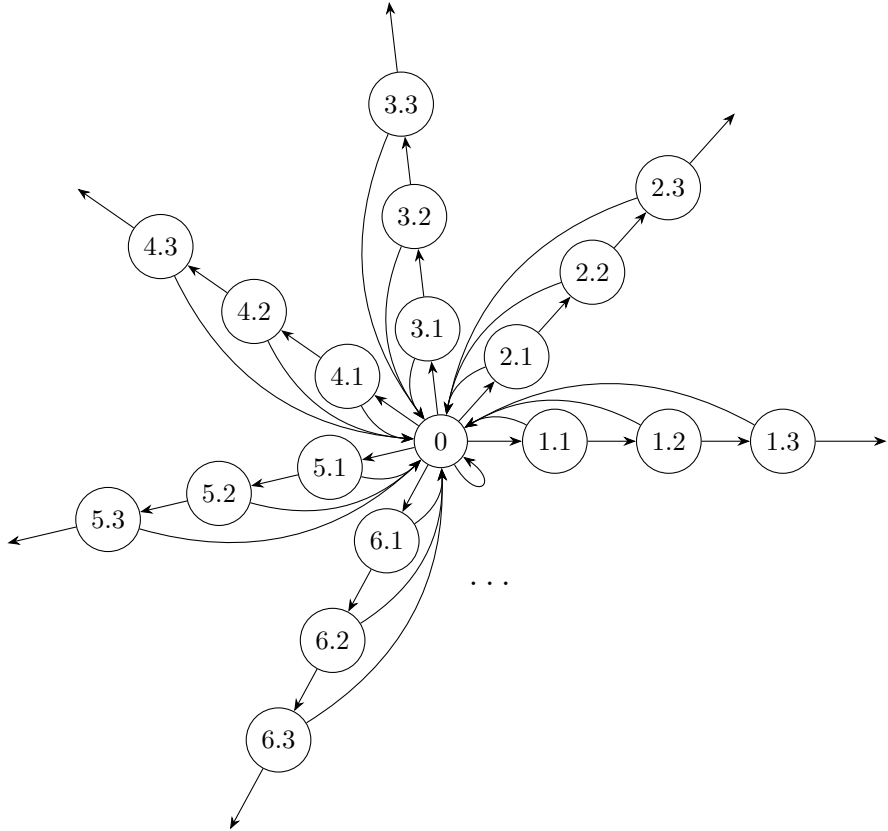
\begin{figure}[h]
\centering
\begin{tikzpicture}[
  state/.style={circle, draw, minimum size=7mm, font=\small},
  >=Stealth
]

\node[state] (O) at (0,0) {0};
\draw[->] (O) to[out=300, in=325, looseness=10] (O);

\def\nTowers{6} 
\def\towerHeight{3} 
\def\radiusStep{1.5} 
\def\angleStep{290/\nTowers}

\foreach \k in {1,...,\nTowers} {
  \pgfmathsetmacro{\angle}{(\k-1)*\angleStep}
  \foreach \l in {1,...,\towerHeight} {
    \node[state] (T\k\l) at ($(O) + (\angle:\l*\radiusStep)$) {\(\k.\l\)};
    \draw[->, bend right=30] (T\k\l) to (O);
  }
  \node[] (T\k fin) at ($(O) + (\angle:4*\radiusStep)$) { };
    \draw[->] (T\k\towerHeight) -- (T\k fin);
  \draw[->] (O) -- (T\k1);
  \foreach \l in {1,...,\numexpr\towerHeight-1} {
    \pgfmathtruncatemacro\nextl{\l+1}
    \draw[->] (T\k\l) -- (T\k\nextl);
  }
}


\node at ($(O) + (-70:2)$) {\Large{\ldots}};


\end{tikzpicture}
\caption{A structure with multiple towers.}
\label{fig:Multiple_towers_structure}
\end{figure}

\noindent For all $i\geq 1$, let $x^{(i), up}= (0(i.1)(i.2)\cdots(i.n)\cdots)$ be the point climbing the $i$-th tower. The set $\mathcal{D}$ of points that are not infinitely recurrent for $[0]$ is 
\begin{align*}
    \mathcal{D} = \bigcup_{i \geq 1} \bigg(\mathcal{O}(x^{(i),up}) \cup \bigcup_{m\geq 0} T^{-m}\{x^{(i),up}\}\bigg). 
\end{align*}

\noindent Thus, we obtain the next result for the asymptotic behavior of the REPP.

\begin{prop}
    \label{prop:all_points_multiple_tower_structure}
    Let $\nu$ be the Pareto law of index $\alpha$ and parameter $\sin(\pi \alpha)/(\pi\alpha)$. Then, for every $x \in \Omega$ and $B_n := [x_0^{n-1}]$,
    \begin{itemize}
        \item if $x$ is non periodic and $x \notin \mathcal{D}$,
        \begin{align*}
            N_{B_n}^{\gamma} \xRightarrow[n \to +\infty]{\mu_{\phi}} \fPp_{\alpha}(\Gamma(1+\alpha)) \quad \text{and} \quad N_{B_n}^{\gamma} \xRightarrow[n \to +\infty]{\mu_{B_n}} \fPp_{\alpha}(\Gamma(1+\alpha)).
        \end{align*}
        \item if $x$ is periodic of prime period $p$,  
        \begin{align*}
                N_{B_n}^{\gamma} \xRightarrow[n \to +\infty]{\mathcal{L}(\mu)} \cfPp_{\alpha}( \theta\Gamma(1+\alpha), \geo(\theta)) \quad \text{and} \quad N_{B_n}^{\gamma} \xRightarrow[n \to +\infty]{\mu_{B_n}} \RPP(W_{\alpha,\theta}(\theta\Gamma(1+\alpha))).
        \end{align*}
        where $\theta = 1 - \exp(S_p\phi(x) - pP_G(\phi))$.
    \end{itemize}
    For points in $\mathcal{D}$, we split according to the values of $p_i$.
    \begin{itemize}
        \item If $p_i = 0$ and $x\in \mathcal{O}(x^{(i),up}) \cup \bigcup_{m\geq 0}T^{-m}\{x^{(i), up}\}$, then
        \begin{align*}
            N_{B_n}^{\gamma} \xRightarrow[n \to +\infty]{\mu_{\phi}} \fPp_{\alpha}(\Gamma(1+\alpha)) \quad \text{and} \quad N_{B_n}^{\gamma} \xRightarrow[n \to +\infty]{\mu_{B_n}} \fPp_{\alpha}(\Gamma(1+\alpha)).
        \end{align*}
        \item If $p_i > 0$ and $x \in \mathcal{O}(x^{(i),up})$ for some $i \geq 0$, denote $\nu_{p_i}$ be the Pareto law of index $\alpha$ and parameter $p_i \sin(\pi\alpha)/(\pi\alpha)$ the law of $p_i^{1/\alpha}W$. Then,
        \begin{align*}
            N_{B_n}^{\gamma} \xRightarrow[n \to +\infty]{\mu_{B_n}} \RPP(\widetilde{J}_{\alpha}(\nu_{p_i})) \quad \text{and} \quad N_{B_n}^{\gamma} \xRightarrow[n \to +\infty]{\mathcal{L}(\mu)} \DRPP(J_{\alpha}(\nu_{p_i}), \widetilde{J}_{\alpha}(\nu_{p_i})).
        \end{align*}
        \item If $p_i > 0$ and $x \in T^{-k}\{x^{(i), up}\}$ for some $k \geq 0$, denote $\nu_{p_i, x}$ the Pareto law of index $\alpha$ and parameter $\exp(S_k\phi_*(x))\,p_i\, \sin(\pi\alpha)/(\pi\alpha)$. Then,
        \begin{align*}
            N_{B_n}^{\gamma} \xRightarrow[n \to +\infty]{\mu_{B_n}} \RPP(\widetilde{J}_{\alpha}((\nu_{p_i, x})) \quad \text{and} \quad N_{B_n}^{\gamma} \xRightarrow[n \to +\infty]{\mathcal{L}(\mu)} \DRPP(J_{\alpha}((\nu_{p_i, x}), \widetilde{J}_{\alpha}((\nu_{p_i, x})).
        \end{align*}
    \end{itemize}
\end{prop}

\begin{rem}
    A consequence of Proposition \ref{prop:all_points_multiple_tower_structure}, is that we can have points that are not recurrent for any symbol but still exhibit a standard fractional Poisson process limit for their associated REPP (the case $p_i = 0$). Indeed, if $p_i = 0$, it implies that the tail of $\ell_i$ is negligible with respect to the tail of $r_{[0]}$. Then, following for $n$ step the point $x^{(i), up}$ will give us an excursion through the tower $i$ but this excursion will be much smaller than excursion through other towers. Thus, in the limit, it is like excursions in tower $i$ do not count and the point $x^{(i),up}$ behaves like every other infinitely recurrent point. 
\end{rem}

\begin{proof}[Proof (of Proposition \ref{prop:all_points_multiple_tower_structure})]
    The proof is similar to the House of Cards case and Theorem \ref{thm:all_points_HoC_structure}. The only change is to study the returns from the points $x^{(i), up}$. However, in this case, we have $B_n = [0]\cap \{\ell_i \geq n\}$ and \eqref{eq:hypothesis_regular_variation_on_each_tower} give the additional scaling factor $p_i$. When $p_i = 0$, the process is the fractional Poisson process (of parameter $\Gamma(1+\alpha)$). 
\end{proof}

\paragraph{The tree House of Cards structure.}

In both the House of Cards structure and its multiple-tower generalizations, a characteristic \textit{trichotomy} emerges, with a distinct limit law governing points that ``climb'' the towers. In this section, we present a simple example of a system in which many points exhibit such climbing behavior, yet the system still satisfies the \textit{all-point REPP} property with only a \textit{dichotomy} between periodic and non-periodic points. This provides an elementary example of a setting where numerous points of comparable significance escape toward infinity, while the system nonetheless displays a standard limiting behavior and gives insights for the more delicate $\mathbb{Z}$-extension case presented earlier.\\

We define the (topologically mixing) CMS $(\Omega, T)$ on the countable alphabet $V := \{ 0\mathbf{a} \; |\; \mathbf{a} \in \{0, 1\}^{k}, \, k \geq 0\}$ and with the transition matrix \index{CMS!tree house of cards}
\begin{align*}
    A(0\mathbf{a},0\mathbf{b}) =  \bigg\{\begin{array}{cc}
        1 & \text{if} \; \mathbf{b} = \emptyset,\; \mathbf{b} = \mathbf{a}0 \; \text{or} \; \mathbf{b} = \mathbf{a}1 \\
        0 & \text{else}. 
    \end{array} 
\end{align*}

For simplicity of the exposition of this example, we present it in the Markov chain context. Thus we consider the probability kernel $P(0\mathbf{a}, 0\mathbf{a}i) = (1 - \alpha/n)/2$ and $P( 0 \mathbf{a},0) = \alpha/n$ for $i \in \{0,1\}$ and $|\mathbf{a}| = n-1$ (see Figure \ref{fig:HoC_tree_structure}). This probability kernel is exactly chosen such that the Markov chain is null-recurrent and the invariant Markov measure $\mu$ satisfies $\mu( [0] \cap \{r_{[0]}\geq n\}) \in \RV(-\alpha)$. Then, for such a Markov chain, we obtain the \textit{all-point REPP} with a \textit{dichotomy} between periodic and non periodic points.

\begin{prop}
    \label{prop:dichotomy_tree_HoC}
    For every $x \in \Omega$ and $B_n := [x_0^{n-1}]$
    \begin{itemize}
        \item If $x$ is non-periodic,
         \begin{align*}
            N_{B_n}^{\gamma} \xRightarrow[n \to +\infty]{\mathcal{L}(\mu)} \fPp_{\alpha}(\Gamma(1+\alpha)) \quad \text{and} \quad N_{B_n}^{\gamma} \xRightarrow[n \to +\infty]{\mu_{B_n}} \fPp_{\alpha}(\Gamma(1+\alpha)).
        \end{align*}
        \item If $x$ is periodic with prime period $q$,
        \begin{align*}
                N_{B_n}^{\gamma} \xRightarrow[n \to +\infty]{\mathcal{L}(\mu)} \cfPp_{\alpha}( \theta\Gamma(1+\alpha), \geo(\theta)) \quad \text{and} \quad N_{B_n}^{\gamma} \xRightarrow[n \to +\infty]{\mu_{B_n}} \RPP(W_{\alpha,\theta}(\theta\Gamma(1+\alpha))).
        \end{align*}
        where $\theta = 1 - \prod_{i=0}^{q-1} P(x_i, x_{i+1})$.
    \end{itemize}
\end{prop}

\begin{proof}[Proof (of Proposition \ref{prop:dichotomy_tree_HoC})]
    Consider the set of points 
    \begin{align*}
        \mathcal{D} := \{ x \in [0] \;|\; x_i \neq 0 \; \forall i\geq 1\}
    \end{align*}
    \textit{i.e.}, $\mathcal{D}$ is the set of points that keep climbing. For every point $x\in \Omega$ that is not in $\bigcup_{k\geq 1} T^k \mathcal{D} \cup \bigcup_{j \geq 0} T^{-j}\mathcal{D}$, the result is a direct consequence of Theorems \ref{thm:HTS_infinite_FPP_infinitely_recurrent} and \ref{thm:HTS_infinite_CFPP_periodic_points}. Furthermore, the extremal index follows from the formula $\theta = 1 - \exp(S_q\phi(x) - qP_G(\phi))$ and the definition of the potential $\phi$ for Markov chains (see Section \ref{section:Connection_with_Markov_Chains}). \\
    For the remaining points, we first show that every point in $\mathcal{D}$ has a $0$ delay limit. Fix some $x \in \mathcal{D}$. Thank to the symmetry of the transition kernel, for all $m\geq n$, we have
    \begin{align*}
        \mu([0] \cap \{r_{[0]}\geq m\}) = 2^n \mu([x_0^{n-1}]\cap \{r_{[0]}\geq m\}).
    \end{align*}
    Thus, for all $t > 0$, we have
    \begin{align*}
        \mu(B_n\cap \{r_{[0]}\geq t/\gamma(\mu(B_n))\}) &= 2^{-n} \mu([0] \cap \{r_{[0]}\geq t/\gamma(\mu(B_n))\}) \\
        & \isEquivTo{n \to +\infty} 2^{-n} t^{-\alpha}\mu(B_n) \quad \text{by Lemma \ref{lem:equivalence_queue_renormalisation_gamma}}\\
        & = o(\mu(B_n))
    \end{align*}
    Thus, the rescaled delay vanishes in the limit and thus point in $\mathcal{D}$ have the fractional Poisson process as a limit by Theorem \ref{thm:HTS_infinite_points_that_never_come_back}. We conclude the proof for the remaining points by Propositions \ref{prop:convergence_REPP_preimages} and \ref{prop:convergence_REPP_images}.
\end{proof}

\begin{figure}[h]
\centering
\begin{tikzpicture}[
  state/.style={circle, draw, minimum size=7mm, inner sep=1pt, font=\small},
  >=Stealth,
  sibling distance=10mm,
  level distance=10mm
]

\node[state] (0) at (0,0) {0};

\node[state] (00) at (-1.5,2) {00};
\node[state] (01) at (1.5,2) {01};

\node[state] (000) at (-2.2,4) {000};
\node[state] (001) at (-0.8,4) {001};
\node[state] (010) at (0.8,4) {010};
\node[state] (011) at (2.2,4) {011};

\node[] (Bdots111) at ($(000) + (0.6,1)$) {};
\node[] (Bdots110) at ($(000) + (-0.6,1)$) {};
\node[] (Bdots101) at ($(001) + (0.6,1)$) {};
\node[] (Bdots100) at ($(001) + (-0.6,1)$) {};
\node[] (Bdots011) at ($(010) + (0.6,1)$) {};
\node[] (Bdots010) at ($(010) + (-0.6,1)$) {};
\node[] (Bdots001) at ($(011) + (0.6,1)$) {};
\node[] (Bdots000) at ($(011) + (-0.6,1)$) {};

\draw[->] (0) -- (00);
\draw[->] (0) -- (01);

\draw[->] (00) -- (000);
\draw[->] (00) -- (001);
\draw[->] (01) -- (010);
\draw[->] (01) -- (011);

\draw[->] (000) -- (Bdots111);
\draw[->] (000) -- (Bdots110);
\draw[->] (001) -- (Bdots101);
\draw[->] (001) -- (Bdots100);
\draw[->] (010) -- (Bdots011);
\draw[->] (010) -- (Bdots010);
\draw[->] (011) -- (Bdots001);
\draw[->] (011) -- (Bdots000);

\draw[->] (0) edge[loop below] (0);
\draw[->, bend right = 45] (000) to (0);
\draw[->, bend left = 10] (001) to (0);
\draw[->, bend right = 10] (010) to (0);
\draw[->, bend left = 45] (011) to (0);
\draw[->, bend right = 30] (00) to (0);
\draw[->, bend left = 30] (01) to (0);



\end{tikzpicture}
\caption{The tree House of Cards structure.}
\label{fig:HoC_tree_structure}
\end{figure}
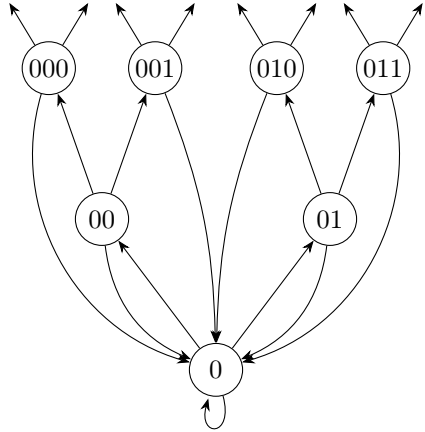

\paragraph{Renewal structure.} We end our presentation with the renewal shift. It is build as the hous of cards shift but with the arrows of the underlying graph inverted. Hence, we consider the countable alphabet alphabet $V = \mathbb{N}$ and the transition matrix 
\begin{align*}
    A(v,w) =  \bigg\{\begin{array}{cc}
        1 & \text{if} \; v = 0 \; \text{or}\; w = v -1 \\
        0 & \text{else}. 
    \end{array} 
\end{align*}

\begin{figure}[h]
\centering
\begin{tikzpicture}[
  state/.style={circle, draw, minimum size=7mm, inner sep=1pt, font=\small},
  >=Stealth,
  node distance=10mm and 12mm
]

\node[state] (H0) {0};
\node[state, right=of H0] (V1) {\(1\)};
\node[state, right=of V1] (V2) {\(2\)};
\node[right=1cm of V2] (Vdots) {\(\cdots\)};
\node[state, right=1cm of Vdots] (Vn) {\(n\)};
\node[right=1cm of Vn] (Vdots2) {\(\cdots\)};

\draw[<-, bend right=40] (V1) to (H0);
\draw[<-, bend right=40] (V2) to (H0);
\draw[<-, bend right=40] (Vn) to (H0);
\draw[->, bend left=45] (H0) to (Vdots2);

\draw[->] (H0) edge[loop left]  (H0);

\foreach \x/\y in {H0/V1, V1/V2, V2/Vdots, Vdots/Vn, Vn/Vdots2} {
  \draw[<-] (\x) -- (\y);
}


\end{tikzpicture}
\caption{The renewal graph.}
\label{fig:Renewal_structure}
\end{figure}
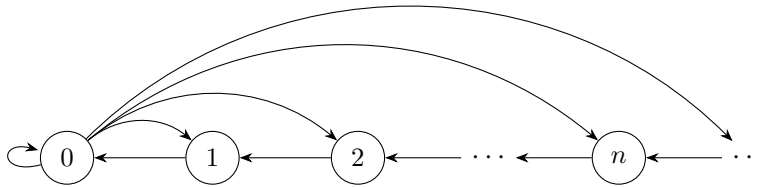

\noindent The renewal structure is special as for all $x \in \Omega$, we have $|\mathcal{O}(x) \cap [0]| = +\infty$, contrary to the House of Card shift where the points $x^{up}$ keeps climbing. So, for every point $x \in \Omega$, we are either in the regime of Theorem \ref{thm:HTS_infinite_FPP_infinitely_recurrent} for non periodic points and Theorem \ref{thm:HTS_infinite_CFPP_periodic_points} for periodic points. Hence, we obtain the \textit{all-point REPP} property with a \textit{dichotomy} between periodic and non-periodic points. \index{CMS!renewal}

\begin{prop} 
    \label{prop:all_points_renewal}
    Let $(\Omega, T)$ be a renewal shift and $\phi$ a potential satisfying Walters' condition and such that $P_G(\phi) < +\infty$. Assume furthermore that for some $v \in V$ there exists $\alpha \in (0,1]$ such that 
    \begin{align*}
        \frac{1}{\mu_{\phi}[v]} \sum_{k = 1}^{n} \lambda_{\phi}^{-k}Z_k(\phi, v) \in RV(\alpha).
    \end{align*}
    Then, for every $x \in \Omega$ and $B_n := [x_0^{n-1}]$ we have
    \begin{itemize}
        \item if $x$ is non periodic, we have
        \begin{align*}
            N_{B_n}^{\gamma} \xRightarrow[n \to +\infty]{\mu_{\phi}} \fPp_{\alpha}(\Gamma(1+\alpha)) \quad \text{and} \quad N_{B_n}^{\gamma} \xRightarrow[n \to +\infty]{\mu_{B_n}} \fPp_{\alpha}(\Gamma(1+\alpha)).
        \end{align*}
        \item if $x$ is periodic of prime period $p$, we have   
        \begin{align*}
                N_{B_n}^{\gamma} \xRightarrow[n \to +\infty]{\mathcal{L}(\mu)} \cfPp_{\alpha}( \theta\Gamma(1+\alpha), \geo(\theta)) \quad \text{and} \quad N_{B_n}^{\gamma} \xRightarrow[n \to +\infty]{\mu_{B_n}} \RPP(W_{\alpha,\theta}(\theta\Gamma(1+\alpha))).
        \end{align*}
        where $\theta = 1 - \exp(S_p\phi(x) - p P_G(\phi))$.
    \end{itemize}
\end{prop}

\begin{proof}[Proof (of Proposition \ref{prop:all_points_renewal})]
    In the renewal structure, for all points $x$, we have $|\mathcal{O}(x) \cap [0]| = +\infty$. The result is a direct application of Theorems \ref{thm:HTS_infinite_FPP_infinitely_recurrent} and \ref{thm:HTS_infinite_CFPP_periodic_points}.
\end{proof}

\begin{rem}
    At first glance, it may seem surprising to obtain an \textit{all-point REPP} property featuring a \textit{dichotomy} and a \textit{trichotomy} (respectively for the renewal and House of Cards structures) given that both are typically expected to exhibit similar probabilistic behaviors. However, the third regime in the renewal structure is not absent but rather hidden when considering certain natural sequences of rare events. Specifically, convergence toward a Pareto-driven point process can be observed for the renewal structure by focusing on the sequence of asymptotically rare events defined by $B_n = [0 (\geq n)] := \bigcup_{k = n}^{+\infty} [0k]$.
\end{rem}

\section{Convergence towards an embedded Subshift of Finite Type}
\label{section:convergence_embedded_SFT}
We conclude this article with another application of our theory to another natural class of rare events. We focus on visits near an embedded subshift of finite type (so with a finite alphabet) within a CMS. This time, the sequence of asymptotically rare events is build as follows: at step $n$, we consider points that remain in the embedded subshift of finite type for its first iterations. These rare events are not cylinders of depth $n$ but instead a union of such cylinders. However, we are still able to find results in this setting without much effort with our theory, thus generalizing the result of \cite{CCC09} which also considered embedded subshifts of finite type but assumed that the general phase space was also a subshift of finite type.\\

More precisely, let $(\Omega, T)$ be a topologically mixing TMS associated with a Walters null-recurrent potential with normalizing sequence $(a_n)_{n\geq 1} \in \RV(\alpha)$ for $\alpha \in(0,1]$. Let $\Delta$ be a finite subset of $V$ and $\Omega_{\Delta} := \Omega \cap \Delta^{\mathbb{N}}$. By construction, $\Omega_{\Delta}$ is shift invariant. We assume furthermore that $\Omega_{\Delta}$ is a topologically mixing subshift of finite type (that is to say there exists some $N > 0$ such that $(A|_{\Delta})^n > 0$ for all $n\geq N$). We consider $\phi_{\Delta} = \phi|_{\Omega_{\Delta}}$ the restricted potential to the subshift of finite type. The potential $\phi_{\Delta}$ has a finite Gurevich pressure $P_{\Delta}$. Indeed, let $v\in \Delta$. Then, for all $n \geq 1$, we have 
\begin{align*}
    Z^{\Omega_{\Delta}}_n(\phi_{\Delta}, v) = \sum_{T^n x = x} e^{S_n\phi(x)} \mathbf{1}_{[v]}(x) \mathbf{1}_{\Omega_{\Delta}}(x) \leq \sum_{T^nx = x} e^{S_n\phi(x)}\mathbf{1}_{[v]}(x) = Z_n(\phi, v).
\end{align*}
In particular, $P_{\Delta} := P_G(\phi_{\Delta})$ (on $\Omega_{\Delta}$) is smaller than $P_G(\phi)$. The next theorem shows that the REPP associated to visits close to $\Omega_{\Delta}$ converge towards a compound fractional Poisson process.

\begin{thm} 
    \label{thm:REPP_subshift_CFPP}
    Let $B_n = [\Delta^n] := \{ x \in \Omega\;|\; (x_0^{n-1}) \in \Delta^{n}\}$. Then, 
    \begin{align*}
        N_{B_n}^{\gamma} \xRightarrow[n\to +\infty]{\mathcal{L}(\mu)} \cfPp_{\alpha}(e^{P_*}\Gamma(1+\alpha), \geo(e^{P_*}))
    \end{align*}
    and 
    \begin{align*}
        N_{B_n}^{\gamma} \xRightarrow[n\to +\infty]{\mu_{B_n}} \RPP(W_{\alpha,e^{P_*}}(e^{P_*}\Gamma(1+\alpha))),
    \end{align*}
    where $P_* = P_{\Delta} - P_G(\phi) < 0$ is the relative pressure of the subshift.
\end{thm}

Theorem \ref{thm:REPP_subshift_CFPP} is proven in Section \ref{section:proof_embedded_SFT}. In particuler, note that the subshift $\Omega_{\Delta}$ acts like a periodic point in the infinite system. Indeed, as a periodic point, we have clusters of returns due to the non-zero of remaining inside the subshift after one step if the point was already inside it. We notice that the computation of the extremal index is different and is this time given by $e^{P_*}$.

\bibliographystyle{abbrv}	
	
\bibliography{biblio.bib}

\begin{thebibliography}{10}

\bibitem{Aar97}
J.~Aaronson.
\newblock {\em An introduction to infinite ergodic theory}, volume~50 of {\em Mathematical Surveys and Monographs}.
\newblock American Mathematical Society, Providence, RI, 1997.

\bibitem{AS11}
M.~Abadi and B.~Saussol.
\newblock Hitting and returning to rare events for all alpha-mixing processes.
\newblock {\em Stochastic Processes and their Applications}, 121(2):314--323, 2011.

\bibitem{AS16}
M.~Abadi and B.~Saussol.
\newblock Almost sure convergence of the clustering factor in {$\alpha$}-mixing processes.
\newblock {\em Stoch. Dyn.}, 16(3):1660016, 11, 2016.

\bibitem{atnip2023compound}
J.~Atnip, G.~Froyland, C.~Gonz\'{a}lez-Tokman, and S.~Vaienti.
\newblock Compound {P}oisson statistics for dynamical systems via spectral perturbation.
\newblock {\em Nonlinearity}, 38(8):Paper No. 085023, 47, 2025.

\bibitem{AFV15}
H.~Ayta\c{c}, J.~M. Freitas, and S.~Vaienti.
\newblock Laws of rare events for deterministic and random dynamical systems.
\newblock {\em Trans. Amer. Math. Soc.}, 367(11):8229--8278, 2015.

\bibitem{BT24_FractionalPoisson}
D.~Bansard-Tresse.
\newblock The fractional poisson process and other limit point processes for rare events in infinite ergodic theory.
\newblock 2024.

\bibitem{BT25_PhD}
D.~Bansard-Tresse.
\newblock {\em Return-Time Point Processes in Finite and Infinite Measure-Preserving Dynamical Systems}.
\newblock PhD thesis, Ph. D. Thesis, Institut Polytechnique de Paris, Ecole Polytechnique, 2025.

\bibitem{BF23}
D.~Bansard-Tresse and J.~M. Freitas.
\newblock Inducing techniques for quantitative recurrence and applications to {M}isiurewicz maps and doubly intermittent maps.
\newblock {\em Ann. Inst. Henri Poincar\'{e} Probab. Stat.}, 61(2):1458--1486, 2025.

\bibitem{Bingham89_RegularVariation}
N.~H. Bingham, C.~M. Goldie, and J.~L. Teugels.
\newblock {\em Regular variation}, volume~27.
\newblock Cambridge university press, 1989.

\bibitem{Bowen75}
R.~Bowen.
\newblock {\em Equilibrium states and the ergodic theory of {A}nosov diffeomorphisms}, volume 470 of {\em Lecture Notes in Mathematics}.
\newblock Springer-Verlag, Berlin, revised edition, 2008.
\newblock With a preface by David Ruelle, Edited by Jean-Ren\'{e} Chazottes.

\bibitem{BZ01}
X.~Bressaud and R.~Zweim\"{u}ller.
\newblock Non exponential law of entrance times in asymptotically rare events for intermittent maps with infinite invariant measure.
\newblock {\em Ann. Henri Poincar\'{e}}, 2(3):501--512, 2001.

\bibitem{BCS22}
J.~Buzzi, S.~Crovisier, and O.~Sarig.
\newblock Measures of maximal entropy for surface diffeomorphisms.
\newblock {\em Ann. of Math. (2)}, 195(2):421--508, 2022.

\bibitem{CampaninoIsola95}
M.~Campanino and S.~Isola.
\newblock Statistical properties of long return times in type {I} intermittency.
\newblock {\em Forum Math.}, 7(3):331--348, 1995.

\bibitem{CCC09}
J.-R. Chazottes, Z.~Coelho, and P.~Collet.
\newblock Poisson processes for subsystems of finite type in symbolic dynamics.
\newblock {\em Stoch. Dyn.}, 9(3):393--422, 2009.

\bibitem{ChazottesCollet13}
J.-R. Chazottes and P.~Collet.
\newblock Poisson approximation for the number of visits to balls in non-uniformly hyperbolic dynamical systems.
\newblock {\em Ergodic Theory Dynam. Systems}, 33(1):49--80, 2013.

\bibitem{Coelho00}
Z.~Coelho.
\newblock Asymptotic laws for symbolic dynamical systems.
\newblock In {\em Topics in symbolic dynamics and applications ({T}emuco, 1997)}, volume 279 of {\em London Math. Soc. Lecture Note Ser.}, pages 123--165. Cambridge Univ. Press, Cambridge, 2000.

\bibitem{Col02}
P.~Collet.
\newblock Statistics of closest return for some non-uniformly hyperbolic systems.
\newblock {\em Ergodic Theory Dynam. Systems}, 21(2):401--420, 2001.

\bibitem{DaleyVereJones03_VolumeI}
D.~J. Daley and D.~Vere-Jones.
\newblock {\em An introduction to the theory of point processes. {V}ol. {I}}.
\newblock Probability and its Applications (New York). Springer-Verlag, New York, second edition, 2003.
\newblock Elementary theory and methods.

\bibitem{DaleyVereJones08_VolumeII}
D.~J. Daley and D.~Vere-Jones.
\newblock {\em An introduction to the theory of point processes. {V}ol. {II}}.
\newblock Probability and its Applications (New York). Springer, New York, second edition, 2008.
\newblock General theory and structure.

\bibitem{FFT10}
A.~C.~M. Freitas, J.~M. Freitas, and M.~Todd.
\newblock Hitting time statistics and extreme value theory.
\newblock {\em Probab. Theory Related Fields}, 147(3-4):675--710, 2010.

\bibitem{FFT13}
A.~C.~M. Freitas, J.~M. Freitas, and M.~Todd.
\newblock The compound {P}oisson limit ruling periodic extreme behaviour of non-uniformly hyperbolic dynamics.
\newblock {\em Comm. Math. Phys.}, 321(2):483--527, 2013.

\bibitem{FFTV16}
A.~C.~M. Freitas, J.~M. Freitas, M.~Todd, and S.~Vaienti.
\newblock Rare events for the {M}anneville-{P}omeau map.
\newblock {\em Stochastic Process. Appl.}, 126(11):3463--3479, 2016.

\bibitem{HLV05}
N.~Haydn, Y.~Lacroix, and S.~Vaienti.
\newblock Hitting and return times in ergodic dynamical systems.
\newblock {\em Ann. Probab.}, 33(5):2043--2050, 2005.

\bibitem{Haydn13_EntryAndReturnTimesDistribution}
N.~T.~A. Haydn.
\newblock Entry and return times distribution.
\newblock {\em Dyn. Syst.}, 28(3):333--353, 2013.

\bibitem{HWZ14}
N.~T.~A. Haydn, N.~Winterberg, and R.~Zweim\"{u}ller.
\newblock Return-time statistics, hitting-time statistics and inducing.
\newblock In {\em Ergodic theory, open dynamics, and coherent structures}, volume~70 of {\em Springer Proc. Math. Stat.}, pages 217--227. Springer, New York, 2014.

\bibitem{Hirata93}
M.~Hirata.
\newblock Poisson law for {A}xiom {A} diffeomorphisms.
\newblock {\em Ergodic Theory Dynam. Systems}, 13(3):533--556, 1993.

\bibitem{Kel12}
G.~Keller.
\newblock Rare events, exponential hitting times and extremal indices via spectral perturbation.
\newblock {\em Dyn. Syst.}, 27(1):11--27, 2012.

\bibitem{KL09}
G.~Keller and C.~Liverani.
\newblock Rare events, escape rates and quasistationarity: some exact formulae.
\newblock {\em J. Stat. Phys.}, 135(3):519--534, 2009.

\bibitem{Las03}
N.~Laskin.
\newblock Fractional {P}oisson process.
\newblock volume~8, pages 201--213. 2003.
\newblock Chaotic transport and complexity in classical and quantum dynamics.

\bibitem{LastPenrose18_LecturesOnPoissonProcess}
G.~Last and M.~Penrose.
\newblock {\em Lectures on the {P}oisson process}, volume~7 of {\em Institute of Mathematical Statistics Textbooks}.
\newblock Cambridge University Press, Cambridge, 2018.

\bibitem{LFFF16}
V.~Lucarini, D.~Faranda, A.~C. Freitas, J.~M.~M. Freitas, M.~Holland, T.~Kuna, M.~Nicol, M.~Todd, and S.~Vaienti.
\newblock {\em Extremes and recurrence in dynamical systems}.
\newblock Pure and Applied Mathematics (Hoboken). John Wiley \& Sons, Inc., Hoboken, NJ, 2016.

\bibitem{Mar17}
J.~Marklof.
\newblock Entry and return times for semi-flows.
\newblock {\em Nonlinearity}, 30(2):810--824, 2017.

\bibitem{MNV11_FFPAndTheInverseStableSubordinator}
M.~M. Meerschaert, E.~Nane, and P.~Vellaisamy.
\newblock The fractional {P}oisson process and the inverse stable subordinator.
\newblock {\em Electron. J. Probab.}, 16:no. 59, 1600--1620, 2011.

\bibitem{MS19_Stochastic_models_for_fractional_calculus}
M.~M. Meerschaert and A.~Sikorskii.
\newblock {\em Stochastic models for fractional calculus}, volume~43 of {\em De Gruyter Studies in Mathematics}.
\newblock De Gruyter, Berlin, second edition, 2019.

\bibitem{Book_FractionalCalculus}
K.~S. Miller and B.~Ross.
\newblock {\em An introduction to the fractional calculus and fractional differential equations}.
\newblock A Wiley-Interscience Publication. John Wiley \& Sons, Inc., New York, 1993.

\bibitem{OwadaSamorodnitsky15}
T.~Owada and G.~Samorodnitsky.
\newblock Functional central limit theorem for heavy tailed stationary infinitely divisible processes generated by conservative flows.
\newblock {\em Ann. Probab.}, 43(1):240--285, 2015.

\bibitem{Pene_livre_infini}
F.~P{\`e}ne.
\newblock {\em Chaos en mesure infinie : récurrence, ergodicité, mélange, théorèmes limites}.
\newblock in preparation.

\bibitem{PeneSaussol10_BackToBallsInBilliards}
F.~P\`ene and B.~Saussol.
\newblock Back to balls in billiards.
\newblock {\em Comm. Math. Phys.}, 293(3):837--866, 2010.

\bibitem{PS23}
F.~P{\`e}ne and B.~Saussol.
\newblock Quantitative recurrence for {$T$}, {$T^{-1}$} transformations.
\newblock {\em Probability Theory and Related Fields}, pages 1--38, 2024.

\bibitem{PSZ13}
F.~P\`ene, B.~Saussol, and R.~Zweim\"{u}ller.
\newblock Return- and hitting-time limits for rare events of null-recurrent {M}arkov maps.
\newblock {\em Ergodic Theory Dynam. Systems}, 37(1):244--276, 2017.

\bibitem{RZ20}
S.~Rechberger and R.~Zweim\"{u}ller.
\newblock Return- and hitting-time distributions of small sets in infinite measure preserving systems.
\newblock {\em Ergodic Theory Dynam. Systems}, 40(8):2239--2273, 2020.

\bibitem{Sarig03_Gibbs}
O.~Sarig.
\newblock Existence of {G}ibbs measures for countable {M}arkov shifts.
\newblock {\em Proc. Amer. Math. Soc.}, 131(6):1751--1758, 2003.

\bibitem{Sar01_NullRecurrentPotentials}
O.~M. Sarig.
\newblock Thermodynamic formalism for null recurrent potentials.
\newblock {\em Israel J. Math.}, 121:285--311, 2001.

\bibitem{tdfsurvey}
O.~M. Sarig.
\newblock Thermodynamic formalism for countable {M}arkov shifts.
\newblock In {\em Hyperbolic dynamics, fluctuations and large deviations}, volume~89 of {\em Proc. Sympos. Pure Math.}, pages 81--117. Amer. Math. Soc., Providence, RI, 2015.

\bibitem{Sau09_Survey_AnIntroductionToQuantitativeRecurrenceInDynamicalSystems}
B.~Saussol.
\newblock An introduction to quantitative {P}oincar\'{e} recurrence in dynamical systems.
\newblock {\em Rev. Math. Phys.}, 21(8):949--979, 2009.

\bibitem{Thomine_These}
D.~Thomine.
\newblock {\em Th{\'e}or{\`e}mes limites pour les sommes de Birkhoff de fonctions d'int{\'e}grale nulle en th{\'e}orie ergodique en mesure infinie}.
\newblock PhD thesis, Rennes 1, 2013.

\bibitem{Yas18}
N.~Yassine.
\newblock Quantitative recurrence of some dynamical systems preserving an infinite measure in dimension one.
\newblock {\em Discrete and Continuous Dynamical Systems}, 38:343--361, 2018.

\bibitem{Yas24_quantitativerecurrencezextensionthreedimensional}
N.~Yassine.
\newblock Quantitative recurrence for $\mathbb{Z}$-extension of three-dimensional axiom a flows.
\newblock preprint, 2024.

\bibitem{Zha21}
X.~Zhang.
\newblock A {P}oisson limit theorem for {G}ibbs-{M}arkov maps.
\newblock {\em Dyn. Syst.}, 36(1):88--103, 2021.

\bibitem{Zwe98}
R.~Zweim\"{u}ller.
\newblock Ergodic properties of infinite measure-preserving interval maps with indifferent fixed points.
\newblock {\em Ergodic Theory Dynam. Systems}, 20(5):1519--1549, 2000.

\bibitem{Zwe07_InfiniteMeasurePreservingTransformationsWithCompactFirstRegeneration}
R.~Zweim\"{u}ller.
\newblock Infinite measure preserving transformations with compact first regeneration.
\newblock {\em J. Anal. Math.}, 103:93--131, 2007.

\bibitem{Zwei07}
R.~Zweim\"{u}ller.
\newblock Mixing limit theorems for ergodic transformations.
\newblock {\em J. Theoret. Probab.}, 20(4):1059--1071, 2007.

\bibitem{Zwe16}
R.~Zweim\"{u}ller.
\newblock The general asymptotic return-time process.
\newblock {\em Israel J. Math.}, 212(1):1--36, 2016.

\bibitem{Zwe22}
R.~Zweim\"{u}ller.
\newblock Hitting times and positions in rare events.
\newblock {\em Ann. H. Lebesgue}, 5:1361--1415, 2022.

\end{thebibliography}

\begin{appendices}

\section{More on recurrent TMS}

\subsection{Some bounded distortion estimates}

\begin{lem} \index{bounded distortion}
    \label{lem:bounded_distortion_CMS_living_in_compact}
    Let $\phi$ be recurrent potential satisfying Walters' condition and such that $P_G(\phi) <+\infty$. Let $v \in V$. Then, the set 
    \begin{align*}
        \{\widehat{T}_{\mu_{\phi}}^n (\mathbf{1}_{[a_0^n]}/\mu[a_0^{n}])\;|\; n\geq 0, \; (a_0^n) \; \text{admissible}, \; a_n = v\}
    \end{align*}
    is compact in $L^{1}(\mu)$. More precisely, set 
    \begin{align*}
        \mathcal{U}_{K, M}(v) := \big\{ u : \Omega \to \mathbb{R}\;|\; u \in [M^{-1}, M] \;\text{on} \;\mathrm{Supp}(u) = [v], \; \int u\,\dd\mu = 1,\; u \; \text{is} \; K\text{-continuous}\;\text{on}\; [v]\}. 
    \end{align*}
    where $K$-continuous for a sequence $K$ converging to $0$ means that $\var_j(f) \leq K(j) \quad \forall j \geq 2$.
    Then,
    \[
    \{\widehat{T}_{\mu_{\phi}}^n (\mathbf{1}_{[a_0^n]}/\mu[a_0^{n}]), \; (a_0^n) \; \text{admissible},\; a_n =v\} \subset \mathcal{U}_{K_v, M_v}(v) \quad \text{for some} \quad M_v > 0\quad\text{and}\quad K_v.
    \]
\end{lem}

\begin{proof}[Proof (of Lemma \ref{lem:bounded_distortion_CMS_living_in_compact})]
    \noindent Let $(a_0^{n})$ be admissible with $a_n = v$. By invariance of the measure and the definition of the transfer operator, we directly have
    \begin{align*}
        \int \frac{1}{\mu[a_0^n]} \widehat{T}_{\mu_{\phi}}^n \mathbf{1}_{[a_0^n]}\,\dd\mu_{\phi} = \mu_{\phi}[a_0^n]/\mu_{\phi}[a_0^n] = 1. 
    \end{align*}
    
    \noindent For all $x \in [v]$, we have
    \begin{align*}
        \widehat{T}_{\mu_{\phi}}^n\mathbf{1}_{[a_0^{n}]} &= L_{\phi_*}^n\mathbf{1}_{[a_0^n]}(x) = \sum_{T^ny = x}e^{S_n\phi_*(x)}\mathbf{1}_{[a_0^{n}]}(y) \quad \text{by Remark \ref{rem:definition_phi_*}}\,,\\
        &= e^{S_n\phi_*(a_0^{n-1}x)} \\
        & = e^{\pm \var_{n+1}(S_n\phi_*)} e^{S_n\phi_*(a_0^{n-1}z)} \quad \forall z\in [v]\,,\\
        & = e^{\pm \var_{n+1}(S_n\phi_*)}L^{n}_{\phi_*}\mathbf{1}_{[a_0^{n-1}]}(z) \quad \forall z\in [v]\,\\
        & = e^{\pm W_1(\phi_*)} \int_{[v]} L^{n}_{\phi_*}\mathbf{1}_{[a_0^{n-1}]}(z)\,\dd\mu_{v}(z) \\
        & = \frac{e^{\pm W_1(\phi_*)}}{\mu[v]} \int \widehat{T}^n \mathbf{1}_{[a_0^{n-1}]}\cdot \mathbf{1}_{[v]}\, \dd\mu_{\phi} \\
        & = \frac{e^{\pm W_1(\phi_*)}}{\mu[v]} \int \mathbf{1}_{[a_0^{n-1}]} \cdot \mathbf{1}_{[v]} \circ T^n\,\dd\mu_{\phi} \\
        & = \frac{e^{\pm W_1(\phi_*)}}{\mu[v]} \mu_{\phi}[a_0^n].
    \end{align*}
    Therefore
    \begin{align}
        \label{eq:in_proof_bounded_distortion_uniforly_bounded}
        \frac{1}{\mu[a_0^{n-1}]}\widehat{T}_{\mu_{\phi}} \mathbf{1}_{[a_0^{n-1}]} \in [M_{v}^{-1}, M_{v}] \quad \text{on $[v]$}\,,
    \end{align}
    where $M_{v} := e^{W_1(\phi_*)}/\mu[v]$. We deduce that $\{\widehat{T}_{\mu_{\phi}}^n (\mathbf{1}_{[a_0^n]}/\mu[a_0^{n}]), \; (a_0^n) \; \text{admissible}\}$ is uniformly bounded from above and below on $[v]$. \\
    
    \noindent Furthermore, for all $j \geq 1$ and $x, y \in [v]$ such that $d(x, y) \leq \eta^j$ (\textit{i.e.}, $x_0^{j-1} = y_0^{j-1}$), we have 
    \begin{align*}
        \widehat{T}_{\mu_{\phi}}^n\mathbf{1}_{[a_0^n]}(x) &= L_{\phi_*}^{n}\mathbf{1}_{[a_0^n]}(x) = e^{S_n\phi_*(a_0^{n-1}x)} \\
        & = e^{\pm \var_{n+j} (S_n\phi_*)}e^{S_n\phi_*(a_0^{n-1}y)} \\
        & = e^{\pm W_{j}(\phi_*)} \widehat{T}_{\mu_{\phi}}^n\mathbf{1}_{[a_0^{n-1}]}(y).
    \end{align*}
    It yields
    \begin{align*}
        \bigg|\frac{1}{\mu[a_0^n]}\widehat{T}_{\mu_{\phi}}^n \mathbf{1}_{[a_0^n]}(x) - \frac{1}{\mu[a_0^{n}]}\widehat{T}_{\mu_{\phi}}^n \mathbf{1}_{[a_0^n]}(y) \bigg| & \leq (e^{W_j(\phi_*)} - 1)\frac{1}{\mu[a_0^n]} \widehat{T}_{\mu_{\phi}}^n\mathbf{1}_{[a_0^{n-1}]}(y)\\
        &\leq (e^{W_j(\phi_*)} - 1) M_{v} \quad \text{ by \eqref{eq:in_proof_bounded_distortion_uniforly_bounded}}.
    \end{align*}
    We deduce that $\{\widehat{T}_{\mu_{\phi}}^n (\mathbf{1}_{[a_0^n]}/\mu[a_0^{n}]), \; (a_0^n) \; \text{admissible}\} \subset \mathcal{U}_{K_v, M_v}(v)$ with $K_v(j) = (e^{W_j(\phi_*)} - 1)M_v$ which goes to $0$ since $W_j(\phi_*) \xrightarrow[j \to +\infty]{} 0$ because $\phi_*$ is Walters. \\

    We now show that $\mathcal{U}_{K,M}(v)$ is a compact set in $L^1(\mu)$ as we cannot directly apply Arzela-Ascoli since $X$ is not necessarily compact. Since we only seek compactness in $L^1(\mu)$, this will be sufficient. Precompactness is sufficient since the underlying space is complete. Thus, consider $\varepsilon > 0$. Let $N \geq 0$ and $(m_i)_{1\leq i \leq N}$ such that $[M^{-1} ,M] \subset \bigcup_{1\leq i \leq N} B(m_i, \varepsilon/(4\mu[v]))$. Consider $j$ such that $K_j \leq \varepsilon/(4\mu[v])$ and let $S$ be a finite subset of $V$ such that $\mu([v] \cap \bigcup_{w \in S^c}\bigcup_{k = 1}^j T^{-k}[w]) \leq \varepsilon/(4M)$. Consider the finite set $F$ of maps $m : S^j \to \{m_i\}_{1\leq i \leq N}$ and the associated finite set of function $\widetilde{f}_m$ such that for all $s_1,\dots, s_j \in S^j$, $\widetilde{f}_m |_{[v s_1 \cdots s_j]} := m(s_1, \dots, s_j)$ and $\widetilde{f} = M$ elsewhere. Then, $\mathcal{U}_{K,M}(v) \subset \bigcup_{m\in F} B_{L^1(\mu)}(\widetilde{f}_m, \varepsilon)$. Indeed, take $f \in \mathcal{U}_{K,M}(v)$ and $s_1, \dots, s_j \in S$. For all $x, y \in [vs_1\cdots s_j]$, $|f(x) - f(y)| \leq K_{j+1} \leq \varepsilon/(4\mu[v])$ while $f(x)\in [M^{-1}, M]$ and thus there exists $1\leq i\leq N$ such that $f(x) \in B(m_i, \varepsilon/(2\mu[v]))$ and we set $m(s_1,\dots, s_j) = m_i$. Then, such a map $m$ belongs to $F$ and we have 
    \begin{align*}
        \| f - \widetilde{f}_m\|_{L^1(\mu)} & \leq \int_{[v] \cap \bigcup_{w \in S^c}\bigcup_{k = 1}^j T^{-k}[w]} |f - \widetilde{f}_m|\,\dd\mu + \sum_{s_1,\dots, s_j \in S^j} \int_{[vs_1\cdots s_j]} |f - \widetilde{f}_m|\,\dd\mu \\
        & \leq M\mu([v] \cap \bigcup_{w \in S^c}\bigcup_{k = 1}^j T^{-k}[w]) + \sum_{s_1,\dots, s_j \in S^j} \int_{[vs_1\cdots s_j]} \varepsilon \,\dd\mu \\
        & \leq \varepsilon /2  + \varepsilon/2 = \varepsilon.
    \end{align*}

\end{proof}

\noindent Another consequence of the bounded distortion property is that the measure of a cylinder is controlled by the product of the measure of sub-cylinders. This property is summarized in the following lemma.

\begin{lem}
    \label{lem:bounded_distortion_CMS}
    There exist a constant $C > 1$ such that for all $(a_0^{n-1}) \in V^{n}$ and $(b_0^{j-1})\in V^j$ with $(a_0^{n-1}b_0^{j-1})$ admissible, we have 
    \begin{align*}
        C^{-1}\frac{\mu[a_0^{n-1}b_0]\mu[b_0^{j-1}]}{\mu[b_0]}\leq \mu[a_0^{n-1}b_0^{j-1}] \leq  C\frac{\mu[a_0^{n-1}b_0]\mu[b_0^{j-1}]}{\mu[b_0]}.
    \end{align*}
    In particular, we can take $C = e^{W_1(\phi_*)}$.
\end{lem}

\begin{proof}[Proof (of Lemma \ref{lem:bounded_distortion_CMS})]
    We have 
    \begin{align*}
        \mu[a_0^{n-1}b_0^{j-1}] & = \mu\big([a_0^{n-1} \cap T^{n}[b_0^{j-1}]\big) = \int \mathbf{1}_{[b_0^{j-1}]} \widehat{T}^n\mathbf{1}_{[a_0^{n-1}]}\,\dd\mu = \int \mathbf{1}_{[b_0^{j-1}]} L_{\phi_*}^n\mathbf{1}_{[a_0^{n-1}]}\,\dd\mu \\
        & = \int \mathbf{1}_{[b_0^{j-1}]}(y) e^{S_n\phi_*(a_0^{n-1}y)}\,\dd\mu.
    \end{align*}
    However, since $\phi_*$ satisfies Walters' condition, we have $W_1(\phi_*) = \sup_{n\geq 1} \var_{n+1} S_n\phi_* <+\infty$. For all $x \in [b_0]$, it implies
    \begin{align*}
        e^{S_n\phi_*(a_0^{n-1}y)} & = e^{\pm W_1(\phi_*)} e^{S_n\phi_*(a_0^{n-1}x)} = e^{\pm C}\frac{1}{\mu[b_0]} \int\mathbf{1}_{[b_0]}(x)e^{S_n\phi_*(a_0^{n-1}x)}\,\dd\mu \\
        & = e^{\pm W_1(\phi_*)} \frac{\mu[a_0^{n-1}b_0]}{\mu[b_0]}.
    \end{align*}
    Hence,
    \begin{align*}
        \mu[a_0^{n-1}b_0^{j-1}] & = \int \mathbf{1}_{[b_0^{j-1}]}e^{\pm W_1(\phi_*)} \frac{\mu[a_0^{n-1}b_0]}{\mu[b_0]}\,\dd\mu \\
        & = e^{\pm W_1(\phi_*)}\frac{\mu[a_0^{n-1}b_0]\mu[b_0^{j-1}]}{\mu[b_0]}.
    \end{align*}
\end{proof}

\subsection{Extremal index estimates}

\noindent In this section, we prove some properties of fast recurrence around periodic and non periodic point. We introduce the topological recurrence $r(A)$ of a set $A$ by the formula
\begin{align*}
    r(A) := \inf_{x\in A} r_A(x).
\end{align*}
To ease the notations, we will drop the brackets when $A$ is a cylinder and write $r(x_0^{n-1})$ instead of $r([x_0^{n-1}])$.

\begin{lem}
    \label{lem:periodicity_equivalent_to_finite_topological_return}
    For all $x\in \Omega$, we have the following equivalence
    \begin{align*}
        x \; \text{is periodic} \quad \text{if and only if} \quad \limsup_{n \to +\infty} r(x_0^{n-1}) < +\infty.
    \end{align*}
    In this case, if $p$ is the prime period of $x$ and $r(x_0^{n-1}) = p$ for $n$ large enough. 
\end{lem}

Although this lemma is classical in symbolic dynamics, we provide a proof for the sake of completeness, as it is both brief and elementary.

\begin{proof}[Proof (of Lemma \ref{lem:periodicity_equivalent_to_finite_topological_return})]
    If $x$ is $p$ periodic, then $[x_0^{n-1}] \cap T^{-p}[x_0^{n-1}] \neq \emptyset$ for all $n \geq 1$ and thus the limit has to be finite. So $r(x_0^{n-1}) \leq p$. In particular, $r(x_0^{n-1})$ is smaller than the prime period of $x$.\\

    \noindent Conversely, note that $(r(x_0^{n-1}))_{n\geq 1}$ is non decreasing. Indeed, for all $k\geq 0$ and $n \geq 1$,  $[x_0^{n}] \cap T^{-k}[x_0^n] \subset [x_0^{n-1}] \cap T^{-k}[x_0^{n-1}]$. Thus, $r(x_0^{n-1})$ converges to some $p$ and there exists $N \geq 0$ such that for all $n\geq N$, $r(x_0^{n-1}) = p$. We are going to show that $T^p(x) = x$. For $n - 1 \geq N$, let $n - 1= qp +r$ its euclidean division by $p$. We have $[x_0^{n-1}] \cap T^{-p}[x_0^{n-1}] = [x_0^{p-1}x_0^{n-1}] \neq \emptyset$ meaning that $x_0^{p-1} = x_p^{2p - 1}$, by propagation $x_{kp}^{(k+1)p -1} = x_0^{p-1}$ for $0\leq k < q$ and $x_{qp}^r = x_0^r$. Taking $n \to +\infty$ ensures that $x_{kp + r} = x_r$ for all $k\geq 0$ and $0 \leq r < p$ and hence $x$ is $p$-periodic.\\
    \noindent With both cases, we also see that $r(x_0^{n-1})$ is the prime period $p$ of $x$ as soon as $n\geq p-1$.
\end{proof}

Now, we are able to prove the formulas for the extremal index (also called potential well in the literature).
\begin{lem}
    \label{lem:calcul_EI}
    Let $x \in \Omega$. We have
    \begin{itemize}
        \item If $x$ is non periodic,
        \[\mu_{[x_0^{n-1}]}\left(r_{[x_0^{n-1}]} > r(x_0^{n-1})\right) \xrightarrow[n\to +\infty]{} 1.\]
        \item If $x$ is periodic of prime period $p$,
        \[\mu_{[x_0^{n-1}]}\left(r_{[x_0^{n-1}]} > r(x_0^{n-1})\right) \xrightarrow[n\to +\infty]{} 1 - e^{S_p\phi(x) - pP_G(\phi)}.\]        
    \end{itemize}
\end{lem}

\begin{proof}[Proof (of Lemma \ref{lem:calcul_EI})]
    \noindent We start by computing the probability of the return time to $[x_0^{n-1}]$ to be equal to $r(x_0^{n-1})$ starting from $[x_0^{n-1}]$. We have
    \begin{align}
        \label{eq:prop_having_return_=_topological_return}
        \mu_{[x_0^{n-1}]}\left(r_{[x_0^{n-1}]} = r(x_0^{n-1})\right) & = \mu_{[x_0^{n-1}]}(T^{-r(x_0^{n-1})}[x_0^{n-1}]) \nonumber\\
        &= \mu[x_0^{n-1}]^{-1} \int \mathbf{1}_{[x_0^{n-1}]} \mathbf{1}_{[x_0^{n-1}]} \circ T^{r(x_0^{n-1})}\,\dd\mu \nonumber\\
        & = \mu[x_0^{n-1}]^{-1} \int \widehat{T}_{\mu_{\phi}}^{r(x_0^{n-1})} \mathbf{1}_{[x_0^{n-1}]}\cdot \mathbf{1}_{[x_0^{n-1}]} \,\dd\mu \nonumber\\
        & = \mu[x_0^{n-1}]^{-1} \int L^{r(x_0^{n-1})}_{\phi_*}\mathbf{1}_{[x_0^{n-1}]} \cdot \mathbf{1}_{[x_0^{n-1}]}\,\dd\mu.
    \end{align}
    \noindent We now split the proof between the periodic and non periodic case.
    \begin{itemize}
        \item \textbf{ $x$ periodic.} If $x$ is periodic of prime period $p$, we have $r(x_0^{n-1}) = p$ for $n$ large enough (by Lemma \ref{lem:periodicity_equivalent_to_finite_topological_return}). Following the proof of Lemma \ref{lem:bounded_distortion_CMS_living_in_compact}, for all $y\in [x_0^{n-1}]$, we obtain
    \begin{align*}
        L_{\phi_*}^{r(x_0^{n-1})}\mathbf{1}_{[x_0^{n-1}]}(y) = L_{\phi_*}^p\mathbf{1}_{[x_0^{n-1}]}(y) & = e^{\pm W_{n-p}(\phi_*)}L_{\phi_*}^p\mathbf{1}_{[x_0^{n-1}]}(x)
    \end{align*}
    and thus
    \begin{align*}
        \mu_{[x_0^{n-1}]}(r_{[x_0^{n-1}]} = r(x_0^{n-1})) &= \mu[x_0^{n-1}]^{-1} \int e^{\pm W_{n-p}(\phi_*)} L_{\phi_*}^{p}\mathbf{1}_{[x_0^{n-1}]}(x)\mathbf{1}_{[x_0^{n-1}]}(y)\,\dd\mu(y)\\
        & = e^{\pm W_{n-p}(\phi_*)}L_{\phi_*}^{p}\mathbf{1}_{[x_0^{n-1}]}(x).
    \end{align*}
    Since $W_{n-p}(\phi_*) \xrightarrow[n \to +\infty]{} 0$, the limit is $L_{\phi_*}^p\mathbf{1}_{[x_0^{n-1}]}(x)$. Furthermore, we have 
    \begin{align*}
        L_{\phi_*}^p\mathbf{1}_{[x_0^{n-1}]}(x) &= e^{S_p\phi_*(x)} = e^{S_p\phi(x) - pP_G(\phi)}e^{\log h(x) - \log h \circ T^p(x)} \quad \text{(by definition of $\phi_*$)}\\
        & = e^{S_p\phi - pP_G(\phi)}.
    \end{align*}
    
    \item \textbf{$x$ non-periodic.} Assume now that $x$ is not periodic. We write $r_n = r(x_0^{n-1})$ to ease the notations. For every $y, z\in [x_0^{n-1}]$, we have 
    \begin{align*}
        L_{\phi_*}^{r_n}\mathbf{1}_{[x_0^{n-1}]}(y) = \sum_{T^{r_n}y' = y} e^{S_{r_n}\phi_*(y')}\mathbf{1}_{[x_0^{n-1}]}(y').
    \end{align*}
    By the Markov property of the TMS, to each $y'$ such that $T^{r_n}y' = y$, there exists a unique $z'$ such that $T^{r_n}z' = z$ and $(y')_0^{r_n} = (z')_0^{r_n}$ (and conversely to each $z'$ we can associate a unique $y'$). Furthermore, since $y_0^{n-1} = z_0^{n-1} = x_0^{n-1}$, we have $(y')_0^{n+r_n - 1} = (z')_0^{n+r_n -1}$. In particular, it gives
    \begin{align*}
        L_{\phi_*}^{r_n}\mathbf{1}_{[x_0^{n-1}]}(y) & = \sum_{T^{r_n}z' = z} e^{S_{r_n}\phi_*(z')}e^{\pm \var_{r_n + n}S_n\phi_*}\mathbf{1}_{[x_0^{n-1}]}(z')\\
        & = e^{\pm W_{r_n}(\phi_*)}L^{r_n}_{\phi_*}\mathbf{1}_{[x_0^{n-1}]}(z).
    \end{align*}
    Hence, 
    \begin{align*}
        L_{\phi_*}^{r_n}\mathbf{1}_{[x_0^{n-1}]}(y) &= e^{\pm W_{r_n}(\phi_*)} \mu[x_0^{n-1}]^{-1} \int L_{\phi_*}^{r_n} \mathbf{1}_{[x_0^{n-1}]}\cdot \mathbf{1}_{[x_0^{n-1}]}\,\dd\mu\\
        & = e^{\pm W_{r_n}(\phi_*)}\mu[x_0^{n-1}]^{-1} \mu([x_0^{n-1}] \cap T^{-r_n}[x_0^{n-1}]).
    \end{align*}
    Going back to \eqref{eq:prop_having_return_=_topological_return}, it yields
    \begin{align*}
        \mu_{[x_0^{n-1}]}(r_{[x_0^{n-1}]} =r_n) &= \mu[x_0^{n-1}]^{-2} \int e^{\pm W_{r_n}(\phi_*)} \mu([x_0^{n-1}] \cap T^{-r_n}[x_0^{n-1}])\cdot\mathbf{1}_{[x_0^{n-1}]}\,\dd\mu\\
        & \leq e^{W_{r_n}(\phi_*)} \mu[x_0^{n-1}]^{-1}\mu([x_0^{n-1}] \cap T^{-r_n}[x_0^{n-1}]).
    \end{align*}
    If $r_n \leq n$, then 
    \begin{align*}
        \mu([x_0^{n-1}] \cap T^{-r_n}[x_0^{n-1}]) &= \mu[x_0^{r_n - 1}x_0^{n-1}]\\
        & \leq e^{W_1(\phi_*)}\mu[x_0]^{-1}\mu[x_0^{r_n}]\mu[x_0^{n-1}] \quad \text{(by Lemma \ref{lem:bounded_distortion_CMS})}.
    \end{align*}
    \noindent If $r_n > n$, we have 
    \begin{align*}
        \mu([x_0^{n-1}] \cap T^{-r_n}[x_0^{n-1}]) &= \sum_{\underset{(x_{n-1}a_0^{r_n - n - 1}x_0)\; \text{admissible}}{(a_0^{r_n - n -1}) \in V^{r_n - n}}} \mu[x_0^{n-1}a_0^{r_n - n - 1}x_0^{n-1}] \\
        & \leq \sum_{\underset{(x_{n-1}a_0^{r_n - n - 1}x_0)\; \text{admissible}}{(a_0^{r_n - n -1}) \in V^{r_n - n}}} e^{W_1(\phi_*)}\mu[x_0]^{-1} \mu[x_0^{n-1}a_0^{r_n - n - 1}x_0]\mu[x_0^{n-1}]\\
        & \leq e^{W_1(\phi_*)}\mu[x_0]^{-1}\mu[x_0^{n-1}]^2\,,
    \end{align*}
    where we used again Lemma \ref{lem:bounded_distortion_CMS}. Thus, we obtain
    \begin{align*}
        \mu_{[x_0^{n-1}]}(r_{[x_0^{n-1}]} =r_n) & \leq e^{W_{r_n}(\phi_*)}e^{W_1(\phi_*)}\mu[x_0]^{-1} \mu[x_0^{n\wedge r_n - 1}].
    \end{align*}
    Since $r_n \to +\infty$ when $n \to +\infty$ (by Lemma \ref{lem:periodicity_equivalent_to_finite_topological_return}), it converges to $0$ as $\mu$ is a non-atomic measure.
    \end{itemize} 
\end{proof}

\subsection{Connection with Markov Chains}
\label{section:Connection_with_Markov_Chains}

We show that a Markov chain on a countable (or finite) state space can be represented as a special topological Markov shift associated with a Markovian potential.\\ \index{Markov!chain}

\noindent Let $(X, T)$ be an irreducible TMS endowed with a recurrent Markovian measure $\mu$. Let $\pi$ be the stationary measure on $V$ and $P$ the transition kernel such that for all $[a_0^k] \in \mathcal{C}_{k+1}(X)$
\begin{align*}
    \mu[a_0^k] = \pi(a_0)\prod_{i=0}^{k-1} P(a_i, a_{i+1}).
\end{align*}

\begin{prop}
    \label{prop:Markov_chain_to_CMS}
    The Markov system $(X, T, \mu)$ is such that  $\mu$ is a fixed point of the Ruelle-Perron-Frobenius operator associated to the Markovian potential
    \begin{align*}
        \phi : x \mapsto \log \bigg(\frac{\pi(x_0)P(x_0, x_1)}{\pi(x_1)}\bigg).
    \end{align*}
\end{prop}

\begin{proof}[Proof (of Proposition \ref{prop:Markov_chain_to_CMS})]
    We just need to check that the transfer operator $\widehat{T}$ associated to $\mu$ is equal $L_{\phi}$. We show that $L_{\phi}$ satisfies the defining equation of the transfer operator. Indeed, for all $p, k \geq 0$,  $[b_0^p] \in \mathcal{C}_{p}(X)$ and $[a_0^k] \in \mathcal{C}_k(X)$ (we can assume that $k \geq p - 1$) we have 
    \begin{align*}
        \int L_{\phi}\mathbf{1}_{[b_0^p]} \cdot \mathbf{1}_{[a_0^k]}\,\dd\mu &= \int \sum_{Ty = x} \frac{\pi(y_0)P(y_0, x_0)}{\pi(x_0)} \mathbf{1}_{[b_0^p]}(y)\mathbf{1}_{[a_0^k]}(x)\,\dd\mu(x)\\
        & = \frac{\pi(b_0)P(b_0, a_0)}{\pi(a_0)} \mathbf{1}_{\{b_1^p = a_0^{p-1}\}}\int \mathbf{1}_{[b_0^p]}(b_0x)\cdot \mathbf{1}_{[a_0^k]}(x)\,\dd\mu(x)\\
        & = \frac{\pi(b_0)P(b_0, a_0)}{\pi(a_0)} \mathbf{1}_{\{b_1^p = a_0^{p-1}\}} \mu[a_0^k]\\
        & = \pi(b_0)P(b_0, a_0)\prod_{i = 0}^{k-1} P(a_i, a_{i+1})\\
        & = \mu([b_0^{p}] \cap T^{-1}[a_0^k])\\
        & = \int \mathbf{1}_{[b_0^p]} \cdot \mathbf{1}_{[a_0^k]}\circ T\,\dd\mu\,.
    \end{align*}
    Of course, since $L_{\phi}1 = 1$, $P_G(\phi) < +\infty$, the generalized Ruelle's Perron-Frobenius Theorem \ref{thm:GRPF} can be applied for $\phi$.
\end{proof}

This ensures that a Markov chain can be viewed as a topological Markov shift (TMS), allowing specific results for recurrent Markov chains to be derived as corollaries of general results for recurrent TMS. Although we will not use this property in what follows, there exists a converse to Proposition \ref{prop:Markov_chain_to_CMS}. Indeed, given a Markovian potential $\phi$ satisfying $L_{\phi}1 = 1$, the associated TMS can be interpreted as a Markov chain in which the direction of the graph's arrows is reversed (see \cite[Section 4.3]{tdfsurvey}).

\subsection{Inducing for CMS}
\label{section:induced_systems_CMS}

Inducing is a powerful tool for deriving limit laws in the context of quantitative recurrence for dynamical systems. For TMS, we aim to extend this approach by inducing on suitably chosen sets. If the inducing set $Y$ is of the form $Y := [V'] := \bigcup_{v'\in V}[v']$ for some $V' \subset V$, we even have a representation of the induced system as a full shift. Indeed, we set 
\begin{align*}
    V_Y := \left\{ \bigsqcup_{q\in V'} [pwq]\;|\; p \in V', \; w \in (V\backslash V')^{n}, \, n \geq 0\right\}
\end{align*}
(with the convention that $(V\backslash V')^0 = \emptyset$) and consider the full shift on the alphabet $V_Y$, that is to say the dynamical system $(V_Y^{\mathbb{N}}, \sigma)$ with $\sigma$ the shift on $V_Y^{\mathbb{N}}$. Then, the exists a (measurable) isomorphism between $(V_Y^{\mathbb{N}}, \sigma)$ and $(Y, T_Y)$ and we can identify the two systems. For the induced system, we will denote $\mathcal{C}_{Y}(n)$ the set of $n$-cylinders on the induced system, \textit{i.e.},
\begin{align*}
    \mathcal{C}_Y(n) = \bigvee_{k = 0}^{n-1} T_Y^{-k}\{[\xi], \; \xi \in V_Y\}.
\end{align*}
For example, if $\xi = \bigsqcup_{q \in V'} [pw_0^{n-1}q] \in V_Y$, then $\xi \in \mathcal{C}_Y(1)$ but also $\xi$ is $\mathcal{C}(n+2)$-measurable (and included in an element of $\mathcal{C}(n+1)$). Of course it also works for deeper cylinders. For example, $\xi = \bigsqcup_{q'\in V'}[pv_0^{n-1}qw_0^{m-1}q'] \in \mathcal{C}_Y(2)$ is $\mathcal{C}(n + m + 3) = \mathcal{C}(1 + r_{Y}^{(2)}|_{[\xi]})$-measurable (and included in an element of $\mathcal{C}(n+m+2)$). \\

\noindent In the particular case, $Y = [v]$ for some $v \in V$, the expression of $V_Y$ is simpler as we have $V_Y = \{[vwv], \, w \in (V\backslash \{v\})^n, \, n \geq 0\}$. Furthermore, $\xi = [vw_0^{n-1}v] \in \mathcal{C}_{[v]}(1)$ is not only $\mathcal{C}(n+2)$-measurable but is in fact an element for $\mathcal{C}(n+2)$.\\ 

\noindent We can also induce potentials on $Y$. For $\phi : \Omega \to \mathbb{R}$, we denote $\phi_Y := \sum_{k = 0}^{r_{Y} - 1} \phi \circ T_k$ the induced potential on $Y$. Fortunately, some regularity properties are preserved through inducing when $Y = [V']$ for some $V' \subset V$.

\begin{prop}
    \label{prop:regularity_implies_induced_regularity}
    If $\phi$ is weakly Hölder (resp. Walters) on $\Omega$, then $\phi_Y$ is weakly Hölder (resp. Walters) on $(V_Y)^{\mathbb{N}}$. 
\end{prop}

\begin{rem}
    Here, we prefer to write $(V_Y)^{\mathbb{N}}$ rather than $Y$ because to talk about regularity with the induced topology, we need to make sure that $T_Y^n$ is well defined for every $n$, which is the case when we work on $(V_Y^{\mathbb{N}}, \sigma)$. So we need to drop the points for which the induced map is not well defined for some $n \geq 1$. For example consider a point $x \in \Omega$ such that $x_0 \in V'$ but $x_i \notin V'$ for all $i \geq 1$. The induced map is not well defined and thus we cannot talk about the induced regularity. However, for the original system, the point $x$ can be still be considered and compared with other points.
\end{rem}

\begin{rem}
    Proposition \ref{prop:regularity_implies_induced_regularity} is not true in general for summable variation, \textit{i.e.}, if $\phi$ has summable variations, then $\phi_Y$ can have non-summable variations. 
\end{rem}

\begin{proof}[Proof (of Proposition \ref{prop:regularity_implies_induced_regularity})]
    \noindent We prove it only for Walters' condition, the proof in the Hölder case being similar. For all $n,k\geq 1$ and $x, y$ in a same element of $\mathcal{C}_Y(n+ k)$ (so in a same element of $\mathcal{C}(r_Y^{(n+k)}(x))$, we have
    \begin{align*}
        S_n \phi_Y(x) - S_n\phi_Y(y) &= \sum_{k = 0}^{r_Y^{(n)}(x) - 1} \phi \circ T^k(x) - \sum_{k = 0}^{r_Y^{(n)}(x) - 1} \phi \circ T^k(y) \\
        & = S_{r_Y^{(n)}(x)} \phi(x) - S_{r_Y^{(n)}(x)} \phi(y)\\
        & \leq \var_{r_Y^{(n)}(x) + r^{(k)}_Y(T_Y^n(x))} S_{r_Y^{(n)}(x)}\phi \\
        & \leq \var_{r_Y^{(n)}(x) + k} S_{r_Y^{(n)}(x)}\phi\\
        & \leq \sup_{n\geq 1} \var_{n+k} S_n\phi = W_k(\phi)\,.
    \end{align*}
    Hence, $\sup_{n\geq 1}\var_{n+k} S_n \phi_Y \leq W_k(\phi)$.
\end{proof}

One of the advantages of inducing is that we can recover a crucial property in the context of CMS, the BIP property (see Definition \ref{defn:BIP_property}).

\section{Proofs of results from Section \ref{section:positive_recurrent_CMS}}

\subsection{Abstract conditions for convergence towards the (compound) Poisson point process}
In the finite measure preserving setting, we recall the assumptions and the abstract theorem from \cite{Zwe22} for proving convergence towards (compound) Poisson point processes. 

\begin{assumption}
    \label{assum:Roland_Compound_measure_finie}
    Let $(B_n)_{n\geq 1}$ be a sequence of asymptotically rare events. Assume the following
    \begin{enumerate}[label = \arabic*.]
        \item \label{cond:CPP_extremal index} There exist two subsets $U(B_n)$ and $Q(B_n)$ of $B_n$ such that $B_n = U(B_n) \sqcup Q(B_n)$ and $\mu(Q(B_n))/\mu(B_n) \to \theta \in (0,1]$ when $n \to +\infty$ (we allow $U(B_n) = \emptyset$).
        \item \label{cond:CPP_good_compact_set} There exists a sequence of measurable functions $\tau_n : B_n \to \mathbb{N}$ and a compact subset $\mathcal{U}$ of $L^1(\mu)$ such that 
        \begin{align*}
            \widehat{T}^{\tau_n}(\mathbf{1}_{Q(B_n)}/\mu(Q(B_n)) := \sum_{k\geq 0} \widehat{T}^{k}(\mathbf{1}_{Q(B_n) \cap \{\tau_n = k\}}/\mu(Q(B_n))) \in \mathcal{U} \quad \forall n \in \mathbb{N}.
        \end{align*}
        \item \label{cond:CPP_tau_n_small_enough} The sequence $(\tau_n)_{n \geq 1}$ is such that 
        \begin{align*}
            \mu(B_n)\,\tau_n \xRightarrow[n \to +\infty]{\mu_{B_n}} 0
        \end{align*}
        \item \label{cond:CPP_no_cluster_from_Q} For $Q(B_n)$ we have $\mu_{Q(B_n)}(r_{B_n}\leq \tau_n) \xrightarrow[n \to +\infty]{} 0$. \\ \\
    If $U(B_n) \neq \emptyset$, assume furthermore
        \item \label{cond:CPP_cluster_from_U} For $U(B_n)$, we have $\mu_{U(B_n)}(r_{B_n} > \tau_n) \xRightarrow[n \to +\infty]{} 0$.
        \item \label{cond:CPP_compatibility_geometric_law} Finally, we have
        \begin{align*}
            \Big\| \widehat{T}_{B_n}(\mathbf{1}_{U(B_n)}/\mu(U(B_n)) - \mathbf{1}_{B_n}/\mu(B_n) \Big\|_{L^{\infty}(\mu)} \xrightarrow[n \to +\infty]{} 0.
        \end{align*}
    \end{enumerate}
\end{assumption}

\begin{thm}[\textup{\cite[Theorems 3.6 and 3.8]{Zwe22}}]
    \label{thm:Roland_CPP_convergence_when_conditions_are_satisfied}
    Let $(X, \mathscr{B}, \mu, T)$ be an ergodic probability preserving dynamical systems. Let $(B_n)_{n\geq 1}$ be a sequence of asymptotically rare events satisfying Assumption \ref{assum:Roland_Compound_measure_finie}. Then,
    \begin{align*}
        N_{B_n}^{\id} \xRightarrow[n \to +\infty]{\mu} \CPPP(\theta, \geo(\theta)) \quad \text{and} \quad N_{B_n}^{\id} \xRightarrow[n \to +\infty]{\mu_{B_n}} \RPP(W_{1, \theta}).
    \end{align*}
    In particular, if $\theta = 1$ or $U(B_n) = \emptyset$, 
    \begin{align*}
        N_{B_n}^{\id} \xRightarrow[n \to +\infty]{\mu} \PPP \quad \text{and} \quad N_{B_n}^{\id} \xRightarrow[n \to +\infty]{\mu_{B_n}} \PPP.
    \end{align*}
\end{thm}

In the positive recurrent case, we exploit the inducing equivalence \cite[Theorem 3]{FFTV16} (see also \cite[Theorem 11.1]{Zwe22}) to reduce the problem to verifying the sufficient conditions and Theorem \ref{thm:Roland_CPP_convergence_when_conditions_are_satisfied} directly for the induced system. \\

Unfortunately, this strategy breaks down in the infinite measure setting, which will be addressed in Section \ref{section:proofs_for_null-recurrent_TMS}. There, the absence of suitable inducing method necessitates a direct proof and other sufficient conditions.

\subsection{Proof of Theorem \ref{thm:all-point_REPP_positive_recurrent_CMS}}
\label{section:proof_of_thm_all_point_REPP_positive_recurrent}

We prove that a topologically mixing CMS associated to a positive recurrent potential have the \textit{all-point REPP} property for cylinders with a \textit{dichotomy} between periodic and non-periodic points. \\

The plan is to use the strategy explained in the previous section and Theorem \ref{thm:Roland_CPP_convergence_when_conditions_are_satisfied}. We start with the proof for non periodic points, \textit{i.e} Theorem \ref{thm:all-point_REPP_positive_recurrent_CMS}-1.

\begin{proof}[Proof (of Theorem \ref{thm:all-point_REPP_positive_recurrent_CMS}-1.)] To use Theorem \ref{thm:Roland_CPP_convergence_when_conditions_are_satisfied} and get the convergence as a consequence, we need to check the Assumption \ref{assum:Roland_Compound_measure_finie}. Furthermore, since we consider non periodic points here, we choose $Q(B_n) = B_n$ and $U(B_n) = \emptyset$ and the extremal index $\theta$ is trivially equal to $1$. Thus we only need to check assumptions from Assumption \ref{assum:Roland_Compound_measure_finie}-\ref{cond:CPP_good_compact_set} to \ref{assum:Roland_Compound_measure_finie}-\ref{cond:CPP_compatibility_geometric_law}. \\

\noindent We split the proof in two distinct parts depending on the (topological) recurrence properties of the point $x \in \Omega$ considered. \\

\noindent \textbf{Case 1 (Infinitely recurrent).} Assume $x \in \Omega$ is such that $r_{x_0}^{(k)}(x) < +\infty$ for all $k\geq 1$, (\textit{i.e.}, $x$ returns an infinite number of time into $[x_0]$, in particular, $\mu$-almost every point verify this hypothesis). Set $\tau_n := \min\{k \geq 1 \;\big|\; r^{(k)}_{x_0}(x) \geq n-1\}$. \cite[Theorem 3]{FFTV16} (or \cite[Theorem 11.1]{Zwe22}) ensures that the convergence of the REPP is equivalent to the convergence of the REPP on the induced dynamical system $([x_0], T_{x_0}, \mu_{x_0})$. Note that $B_n := [x_0^{n-1}]$ is $\mathcal{C}_{x_0}(\tau_n)$ measurable and is included in a $(\tau_n - 1)$-cylinder for the induced system. Because $([x_0],T_{x_0}, \mu_{x_0})$ is $\psi$-mixing, we have exponential decay of the measure of cylinders and thus 
\begin{align*}
    \tau_n\mu_{x_0}(B_n) = \tau_n\mu[x_0]^{-1}\mu(B_n) \leq C\mu[x_0]^{-1}\tau_ne^{-c(\tau_n - 1)} \xrightarrow[n\to +\infty]{} 0
\end{align*}
\noindent because $B_n$ is included in an element of $\mathcal{C}_{x_0}(\tau_n - 1)$. It proves \ref{assum:Roland_Compound_measure_finie}-\ref{cond:CPP_tau_n_small_enough}, \textit{i.e.}, $\mu_{[x_0]}(B_n) \tau_n \to 0$.\\

\noindent We turn now to \ref{assum:Roland_Compound_measure_finie}-\ref{cond:CPP_good_compact_set} That is to say, we need to show that there exists a compact subset $\mathcal{U}$ of $L^1(\mu)$ such that
\begin{align*}
    \widehat{T}^{\tau_n}\bigg( \frac{\mathbf{1}_{B_n}}{\mu(B_n)}\bigg) \in \mathcal{U}, \quad \forall n\geq 1.
\end{align*}

As stated earlier, $B_n$ is measurable with respect to $\mathcal{C}_{[x_0]}(\tau_n)$ and thus we can write $B_n =: \bigsqcup_{k\in I_n} B_n^{(k)}$ where $B_n^{(k)} \in \mathcal{C}_{x_0}(\tau_n)$ for every $k$ and $I_n \subset \mathbb{N}$ (eventually infinite). Each $B_n^{(k)} =: [y_0^{p_{n,k} - 1}]$ is in  $\mathcal{C}(p_{n,k})$ with $p_{n,k}\geq n$ and where $(y_0^{p_{n,k} - 1})$ verifies $y_0^{n-1} = x_0^{n-1}$ and $y_{p_{n,k} - 1} = x_0$. Thus, by the bounded distortion Lemma \ref{lem:bounded_distortion_CMS_living_in_compact}, there exist $K_{x_0}, M_{x_0} > 1$ such that for all $n\geq 1$ and $k \in I_n$
\begin{align*}
    \frac{1}{\mu(B_n^{(k)})}\widehat{T}_{x_0}^{\tau_n}\mathbf{1}_{B_n^{(k)}} = \mu[y_0^{p_{n,k}-1}]^{-1}L_{\phi_*}^{p_{n,k}-1}\mathbf{1}_{[y_0^{p_{n,k}-1}]} \in \mathcal{U}_{K_{x_0}, M_{x_0}}(x_0)
\end{align*}
and thus, since $\mathcal{U}_{K_{x_0}, M_{x_0}}(x_0)$ is stable by (finite or countable) convex combination, we obtain
\begin{align*}
    \frac{1}{\mu(B_n)}\widehat{T}_{x_0}^{\tau_n}\mathbf{1}_{B_n}  = \sum_{k \in I_n} \frac{\mu(B_n^{(k)})}{\mu(B_n)}\frac{1}{\mu(B_n^{(k)})}\widehat{T}_{x_0}^{\tau_n}\mathbf{1}_{B_n^{(k)}} \in \mathcal{U}_{K_{x_0}, M_{x_0}}(x_0)
\end{align*}
and \ref{assum:Roland_Compound_measure_finie}-\ref{cond:CPP_good_compact_set} follows.\\

\noindent It remains to check \ref{assum:Roland_Compound_measure_finie}-\ref{cond:CPP_no_cluster_from_Q}, \textit{i.e.}, $\mu_{B_n}(r_{B_n}^{x_0} \leq \tau_n) \to 0$ as $n \to +\infty$. We have
\begin{align*}
    \mu_{B_n}(r^{x_0}_{B_n}\leq \tau_n) &\leq  \mu_{B_n}(r_{B_n} < n) + \mu(B_n)^{-1}\sum_{k \in I_n} \mu(B_n^{(k)} \cap T_{x_0}^{-\tau_n}(B_n)).
\end{align*}

\noindent We can use the bounded distortion Lemma \ref{lem:bounded_distortion_CMS} estimate to get 
\begin{align*}
    \mu(B_n)^{-1}\sum_{k \in I_n} \mu(B_n^{(k)} \cap T_{x_0}^{-\tau_n}(B_n)) \leq e^{W_1(\phi_*)}\mu[x_0]^{-1} \mu(B_n)^{-1} \sum_{k \in I_n} \mu(B_n^{(k)}) \mu(B_n) \xrightarrow[n \to +\infty]{} 0\,.
\end{align*}
\noindent On the other side, we need to control $\mu_{B_n}(r_{B_n} < n)$. If $r(x_0^{n-1}) \geq n$, the probability is $0$ so there is nothing to prove. So we can consider the integers such that $r(x_0^{n-1}) < n$. Then, we have
\begin{align*}
    \mu_{B_n}(r_{B_n} < n) = \mu_{B_n}(r_{B_n} = r(x_0^{n-1})) + \mu_{B_n}(r(x_0^{n-1}) < r_{B_n} < n).
\end{align*}

\noindent By Lemma \ref{lem:calcul_EI}, we know that the first term vanishes. Let $n =: q_nr(x_0^{n-1}) +r_n$ be the euclidean division of $n$ by $r(x_0^{n-1})$ (by hypothesis $q_n\geq 1$). We have 
\begin{align*}
    \{r(x_0^{n-1}) < r_{[x_0^{n-1}]} < n\} = \{q_n r(x_0^{n-1}) < r_{[x_0^{n-1}]} < n\}. 
\end{align*}
Indeed, by definition of $r(x_0^{n-1})$ before $q_n$, returns before $n$ can only occur at time $i r(x_0^{n-1})$ for $1 \leq i \leq q_n$. However, if a return occurs at such time for a point $y$, we have $x_0^{n-1} = y_0^{n-1} = y_{ir(x_0^{n-1})}^{ir(x_0^{n-1}) + n-1}$ and in particular $y_{r(x_0^{n-1})}^{r(x_0^{n-1}) + n-1} = x_0^{n-1}$ and thus $r_{[x_0^{n-1}]}(y) = r(x_0^{n-1})$.\\  
Thus, we have
\begin{align*}
    \mu_{[x_0^{n - 1}]}(r(x_0^{n-1})< r_{[x_0^{n - 1}]} < n) &\leq \sum_{k = q_n r(x_0^{n-1})}^{n-1} \mu_{[x_0^{n-1}]}(T^{-k}[x_0^{n-1}]) \\
    &\leq \sum_{\overset{q_nr(x_0^{n-1}) \leq k \leq n-1}{x_k = x_0}} \mu[x_0^{n-1}]^{-1}\mu[x_0^{k-1}x_0^{n-1}] \\
    & \leq \sum_{\overset{q_nr(x_0^{n-1}) \leq k \leq n-1}{x_k = x_0}} e^{W_1(\phi_*)}\mu[x_0]^{-1} \mu[x_0^{k-1}x_0] \quad \text{by Lemma \ref{lem:bounded_distortion_CMS}} \\
    & \leq e^{W_1(\phi_*)}\mu[x_0]^{-1} \sum_{k = j_n}^{+\infty}  Ce^{-ck} \xrightarrow[n\to +\infty]{} 0,
\end{align*}
where we used bounded distortion estimates and the exponential of the measure of the cylinders and $j_n := \max \{k\geq 0\;|\; r^{(k)}_{x_0}(x) \leq q_n r(x_0^{n-1})\} \to +\infty$ when $n\to +\infty$. Hence \ref{assum:Roland_Compound_measure_finie}-\ref{cond:CPP_no_cluster_from_Q} is satisfied and the limit behavior of the REPP is a direct consequence of Theorem \ref{thm:Roland_CPP_convergence_when_conditions_are_satisfied}.\\

\noindent \textbf{Case 2 (Finitely recurrent).}
Assume now  that $p := \inf \{k\geq 1 \;|\; r_{x_0}^{(k)}(x) = +\infty\}$ is finite. Consider $n \geq r^{(p-1)}_{x_0}(x)$. Since the TMS is topologically mixing, up to zero measure, we have again a decomposition of $B_n$ of the following form
\begin{equation*}
    B_n =: \bigsqcup_{k\in I_n} B_n^{(k)}, 
\end{equation*}
where $B_n^{(k)} \in \mathcal{C}_{x_0}(p)$ (contrary to the previous case, the depth $p$ for the induced system is constant this time). Setting $\tau_n := p$, we obtain 
\begin{align}
    \label{eq:compact_finitely_recurrent}
    \widehat{T}_{x_0}^{\tau_n} \left(\frac{\mathbf{1}_{B_n}}{\mu(B_n)}\right) &= \sum_{k\in I_n} \frac{1}{\mu(B_n)} T_{x_0}^{\tau_n} \mathbf{1}_{B_n^{(k)}} = \sum_{k \in I_n} \frac{1}{\mu(B_n)} L_{\phi^*}^{p_{n,k} - 1} \mathbf{1}_{B_n^{(k)}}. 
\end{align}
where $p_{n,k}$ is such that $B_n^{(k)} \in \mathcal{C}(p_{n,k})$.
As in the first case, this form a compact sequence ensuring Assumption \ref{assum:Roland_Compound_measure_finie}-\ref{cond:CPP_good_compact_set}\\

\noindent Furthermore, for the induced system, Assumption \ref{assum:Roland_Compound_measure_finie}-\ref{cond:CPP_tau_n_small_enough} is even easier because $\tau_n\mu_{x_0}(B_n) = p\mu_{x_0}(B_n) \xrightarrow[n\to +\infty]{} 0$. \\

\noindent Finally, we prove that \ref{assum:Roland_Compound_measure_finie}-\ref{cond:CPP_no_cluster_from_Q} holds. Assume $n > 2r_{x_0}^{(p-1)}(x)$. Each $B_n^{(k)} \in \mathcal{C}_{x_0}(p)$ and their $p$-th symbol is different from the $p-1$ preceding (because of the choice of $n$). Hence,
\begin{align*}
    \mu_{B_n}(r^{x_0}_{B_n} \leq \tau_n) & = \frac{\mu_{x_0}(B_n\cap \{r^{x_0}_{B_n} \leq p\})}{\mu_{x_0}(B_n)} = \frac{\mu_{x_0}(B_n \cap T_{x_0}^{-p}B_n)}{\mu_{x_0}(B_n)} \\
    & \leq C \frac{\mu_{x_0}(B_n)^2}{\mu_{x_0}(B_n)} = C \mu_{x_0}(B_n) \xrightarrow[n\to +\infty]{} 0
\end{align*}
by bounded distortion (Lemma \ref{lem:bounded_distortion_CMS}).
\end{proof}

\noindent We turn now to the study of periodic points. This time, due to the presence of cluster, the limit process is a Compound Poisson Process. However, the study follows the same path as in the finite case. The main difference is that we need to look at deeper cylinders (to control the exit of the cluster) but the additional depth is the period and it vanishes in the limit. On the other side, periodic points also have nicer recurrence properties, for example they are made by only a finite number of symbols, and it can give a valuable help in the proofs.

\begin{proof}[Proof (of Theorem \ref{thm:all-point_REPP_positive_recurrent_CMS}-2.)]
    Let $x \in \Omega$ be a periodic point of prime period $q$. Consider $Q(B_n) = B_n \cap T^{-q}B_n^c$ and $U(B_n) = B_n \cap T^{-q}B_n$. By Lemma \ref{lem:calcul_EI}, Assumption \ref{assum:Roland_Compound_measure_finie}-\ref{cond:CPP_extremal index} is satisfied with $\theta = 1 - e^{S_q\phi(x) - qP_{G}(\phi)}$. Again, it is enough to work on the induced system $([x_0], T_{x_0}, \mu_{x_0})$. Let $\tau_n := \min\{k\geq 1\;|\; r_{x_0}^{(k)}(x) \geq n + q -1\}$. Then, $Q(B_n)$ is measurable with respect to $\mathcal{C}_{x_0}(\tau_n + q)$ (at worst) and is included in an element of $\mathcal{C}_{x_0}(\tau_n - 1)$. As in the proof of Theorem \ref{thm:all-point_REPP_positive_recurrent_CMS}-1., Assumptions \ref{assum:Roland_Compound_measure_finie}-\ref{cond:CPP_good_compact_set}, \ref{assum:Roland_Compound_measure_finie}-\ref{cond:CPP_tau_n_small_enough} and \ref{assum:Roland_Compound_measure_finie}-\ref{cond:CPP_no_cluster_from_Q} are satisfied as the additional $q$ disappears in the limit. In the periodic case, we need to show the two other Assumptions \ref{assum:Roland_Compound_measure_finie}-\ref{cond:CPP_cluster_from_U} and \ref{assum:Roland_Compound_measure_finie}-\ref{cond:CPP_compatibility_geometric_law}. We have $U(B_n) = B_n \cap T^{-q}B_n = [x_0^{n + q - 1}]$. Since $r_{B_n} = q$ on $U(B_n)$ and $\tau_n \xrightarrow[n\to +\infty]{} +\infty$, assumption \ref{assum:Roland_Compound_measure_finie}-\ref{cond:CPP_cluster_from_U} is easy. Finally, for all $z, z' \in [x_0^{n-1}]$, we have
    \begin{align*}
        & \widehat{T_{B_n}} \frac{\mathbf{1}_{U(B_n)}}{\mu(U(B_n))}(z)  = \widehat{T}^q \frac{\mathbf{1}_{[x_0^{n+q - 1}]}}{\mu[x_0^{n+q - 1}]}(z) = \frac{1}{\mu[x_0^{n+q - 1}]}L_{\phi_*}^q \mathbf{1}_{[x_0^{n+q - 1}]}(z) \\ 
        & = \frac{1}{\mu[x_0^{n+q - 1}]} \sum_{T^qy = z}e^{S_q\phi_*(y)} \mathbf{1}_{[x_0^{n+q - 1}]}(y)  = \frac{1}{\mu[x_0^{n+q - 1}]} \sum_{T^qy' = z'}e^{S_q\phi_*(y')}e^{\pm W_{n}(\phi_*)} \mathbf{1}_{[x_0^{n+q - 1}]}(y')\\
        & = e^{\pm W_n(\phi_*)} \widehat{T_{B_n}} \frac{\mathbf{1}_{U(B_n)}}{\mu(U(B_n))}(z').
    \end{align*}
    Using that $\widehat{T_{B_n}} \frac{\mathbf{1}_{U(B_n)}}{\mu(U(B_n))}$ is probability density with support on $B_n$ and $W_n(\phi_*) \xrightarrow[n\to +\infty]{} 0$, it ensures \ref{assum:Roland_Compound_measure_finie}-\ref{cond:CPP_compatibility_geometric_law} and Theorem \ref{thm:Roland_CPP_convergence_when_conditions_are_satisfied} can be applied.
\end{proof}

\section{Proofs of results from Section \ref{section:quatitative_recurrence_null-recurrent_CMS}}
\label{section:proofs_for_null-recurrent_TMS}

\noindent Recall that in this section, we assume that we have a topologically mixing CMS $(\Omega,T)$ endowed with a null recurrent potential $\phi : \Omega \to \mathbb{R}$. By the Generalized Ruelle-Perron-Frobenius Theorem (Theorem \ref{thm:GRPF}), we know that there exists an invariant measure $\mu$ such that the dynamical system $(\Omega, T, \mu)$ is a PDE CEMPT (Theorem \ref{thm:GRPF_null_recurrent}). As in the positive recurrent case, we denote $\phi_*$ the potential such that $\widehat{T} = L_{\phi*}$ where $\widehat{T}$ is the transfer operator associated to the dynamical system $(\Omega, T, \mu)$. Assume furthermore that the normalizing sequence $(a_n) \in RV(\alpha)$ for some $0 < \alpha \leq 1$.  \\

Recall that for an infinite measure preserving dynamical systems and a subset $A$ of finite measure, the wandering rate of $A$ is defined by 
\begin{align*}
    w_n(A) := \mu\bigg( \bigcup_{k = 0}^{n-1} T^{-k}A\bigg) = \sum_{k = 0}^{n-1} \mu(A \cap \{r_A > k\})
\end{align*}

\noindent Let us present a first consequence of regular variation on the tail of returns towards a non-rare event. 

\begin{lem}
    \label{lem:equivalence_queue_renormalisation_gamma}
    Let $(u_n)_{n\geq 1}$ be non-increasing sequence converging to $0$. Then, for all $k \geq 1$ and  $\mathbf{v} \in V^k$ admissible,
    \begin{align*}
        \mu \big([\mathbf{v}] \cap \{\gamma(u_n)r_{[\mathbf{v}]} \geq t\}) \sim \frac{1}{\Gamma(1+\alpha)\Gamma(1 - \alpha)}t^{-\alpha}u_n\,.
    \end{align*}
\end{lem}

\begin{proof}[Proof (of Lemma \ref{lem:equivalence_queue_renormalisation_gamma})] By \cite[Proposition 3.8.7 p.137]{Aar97} and since $[\mathbf{v}]$ is a uniform set for all admissible $\mathbf{v} \in V^k$, we have 
\begin{align*}
    w_{n}([\mathbf{v}]) \sim \frac{1}{\Gamma(1+\alpha)\Gamma(2 - \alpha)} \frac{n}{a_n}.
\end{align*}

\noindent Since the sequence $(\mu([\mathbf{v}] \cap \{r_{[\mathbf{v}]} > n\}))_{n\geq 0}$ is monotonous, by the monotone density theorem \cite[Theorem 1.7.2]{Bingham89_RegularVariation}, we get
\begin{align*}
    \mu([\mathbf{v}] \cap \{r_{[\mathbf{v}]} > n\}) & \isEquivTo{n\to +\infty} (1- \alpha) \frac{w_n([\mathbf{v}])}{n} \isEquivTo{n \to +\infty}  \frac{1 - \alpha}{\Gamma(1+\alpha)\Gamma(2 - \alpha)} \frac{1}{a_n}\\
    &\isEquivTo{n \to +\infty} \frac{1}{\Gamma(1+\alpha)\Gamma(1 - \alpha)}a_n^{-1}
    = \frac{\sin(\pi \alpha)}{\pi\alpha}a_n^{-1}.
\end{align*}

We deduce that $(\mu([\mathbf{v}] \cap \{r_{[\mathbf{v}]} > n\}))_{n\geq 0} \in \RV(-\alpha)$. In particular, since $\mu([\mathbf{v}] \cap \{r_{[\mathbf{v}]} > n+1\}) \isEquivTo{n\to +\infty} \mu([\mathbf{v}] \cap \{r_{[\mathbf{v}]} > n\})$, we have
\begin{align*}
    \mu([\mathbf{v}] \cap \{\gamma(u_n)\,r_{[\mathbf{v}]} \geq t\}) &\isEquivTo{n\to +\infty} \frac{\sin(\pi \alpha)}{\pi\alpha} \big(a(ta^{\leftarrow}(u_n^{-1}))\big)^{-1} \quad \text{by definition of $\gamma$ in \eqref{eq:defn_gamma}}\\
    & \isEquivTo{n\to +\infty} \frac{\sin(\pi \alpha)}{\pi\alpha} t^{-\alpha}u_n.
\end{align*}
\end{proof}

\noindent Finally, for a measurable set $B \in \mathscr{B}$, we also define its entrance time $e_B$ by $e_B(x) := \inf \{n \geq 0\;|\; T^nx\in B\}$. By induction, we can also define the successive entrance times $e_B^{(k)}$. This quantity is linked to the return time as we have $e_B = r_B \,\mathbf{1}_{B^c}$ but it will sometimes be more suitable to work with it, especially when we deal with the delay times.

\subsection{Proof of Theorems \ref{thm:HTS_infinite_FPP_infinitely_recurrent} and \ref{thm:HTS_infinite_CFPP_periodic_points}}
\label{section:FPP_CFPP_infinite_recurrent_points}
\label{section:proof_of_thm_FPP_infinite_recurrent_and_periodic}

\noindent The general strategy parallels that used in the case of a positive recurrent potential. However, a crucial difference arises: in the null-recurrent (or infinite measure) setting, we can no longer rely on the standard inducing argument to transfer results from the induced system back to the original system. While the inducing procedure itself remains structurally similar in both the finite and infinite measure settings, its implications differ significantly. However, our approach continues to make essential use of the favorable properties of the induced system, in particular the bounded distortion estimates given in Lemmas ~\ref{lem:bounded_distortion_CMS_living_in_compact} and~\ref{lem:bounded_distortion_CMS} .\\

\paragraph{Non-periodic infinitely recurrent points.}
\label{subsubsection:FPP_CMS_generic_points}
In this paragraph, we establish Theorem \ref{thm:HTS_infinite_FPP_infinitely_recurrent}. An abstract convergence result towards the fractional Poisson process has already been developed in \cite{BT24_FractionalPoisson}, and we recall here the relevant assumptions and statement of that theorem. We remark that this abstract result also encompasses the case where clusters appear, which will be addressed in the following paragraph. Nevertheless, while the theorem provides a general framework and identifies the conditions required for convergence, the principal challenge lies in verifying these conditions for the specific sequence of asymptotically rare events under consideration. \\

\noindent  \textbf{Assumption $(A)_{\alpha}$.} A sequence $(B_n) \in \mathscr{B}^{\mathbb{N}}$ of asymptotically rare events 
satisfies $(A)_{\alpha}$ if there exists a uniform set $Y$ such that $B_n \subset Y$ for all $n\geq 1$ and
\begin{enumerate}[label = $(A\arabic*)_{\alpha}$, leftmargin=*]
\item \label{cond_CFPP:extremal_index} For every $n \geq 1$, we can write 
$B_n = U(B_n) \sqcup Q(B_n)$, and $\lim_{n\to +\infty} \mu(Q(B_n))/\mu(B_n) = \theta 
\in (0,1]$ (the extremal index).
\item \label{cond_CFPP:good_density_after_tau_n} There exists a sequence of 
measurable functions $\tau_n : B_n \to \mathbb{N}$ and a compact subset $\mathcal{U}$ 
of $L^1(\mu)$ such that
\[
\widehat{T^{\tau_n}}\left(\frac{\mathbf{1}_{Q(B_n)}}{\mu(Q(B_n))}\right)\in 
\mathcal{U} \;, \forall n\geq 1.
\]
\item \label{cond_CFPP:tau_n_small_enough} The sequence $(\tau_n)_{n\geq 0}$
satisfies $\gamma(\mu(B_n))\,\tau_n \xRightarrow[n\to +\infty]{\mu_{B_n}} 0,$ where 
$\gamma$ is defined in \eqref{eq:defn_gamma}.
\item \label{cond_CFPP:cluster_compatible_tau_n_no_cluster_from_Q} The sequence 
$(Q(B_n))_{n\geq 0}$ is such that $\mu_{Q(B_n)}(\lr_{B_n} < \tau_n) \xrightarrow[n\to 
+\infty]{} 0.$
\end{enumerate}
Furthermore, if $U(B_n) \neq \emptyset$, we have 
\begin{enumerate}[label = $(A\arabic*)_{\alpha}$, leftmargin=*]
\setcounter{enumi}{4}
\item \label{cond_CFPP:cluster_compatible_tau_n_cluster_from_U}
The sequence $(U(B_n))_{n\geq 0}$ is such that $\mu_{U(B_n)}(\lr_{B_n} > 
\tau_n)\xrightarrow[n\to +\infty]{} 0$
\item \label{cond_CFPP:Compatibility_Geometric_law} We have the following limit
\begin{align*}
\left\|\; \widehat{T_{B_n}}\left(\frac{\mathbf{1}_{U(B_n)}}{\mu(U(B_n))}\right) - 
\frac{1}{\mu(B_n)}\mathbf{1}_{B_n}\right\|_{L^{\infty}(\mu_{B_n})} \xrightarrow[n\to 
+\infty]{} 0.
\end{align*}
\end{enumerate}

\begin{thm}[\textup{\cite[Theorem 2.3]{BT24_FractionalPoisson}}]
\label{thm:sufficient_conditions_convergence_compound_FPP}
Let $(X,\mathscr{B},\mu, T)$ be a PDE CEMPT with $\mu(X) = +\infty$ and normalizing sequence $(a_n)_{n\geq 1} \in \RV(\alpha)$ for some $0 < \alpha \leq 1$. Let $(B_n)_{n\geq 0}$ be a sequence of asymptotically rare events satisfying $(A)_{\alpha}$. Then,
    \begin{align*}
        N_{B_n}^{\gamma} \xRightarrow[n\to +\infty]{\mathcal{L}(\mu)} \cfPp_{\alpha}(\theta\Gamma(1+\alpha), \geo(\theta))
    \end{align*}
    and 
    \begin{align*}
        N_{B_n}^{\gamma} \xRightarrow[n\to +\infty]{\mu_{B_n}} \RPP(W_{\alpha,\theta}(\theta\Gamma(1+\alpha)))\,.
    \end{align*}
    In particular, if $\theta = 1$, we have 
    \begin{align*}
        N_{B_n}^{\gamma} \xRightarrow[n\to +\infty]{\mathcal{L}(\mu)} \fPp_{\alpha}(\Gamma(1+\alpha)) \quad \text{and} \quad  N_{B_n}^{\gamma} \xRightarrow[n\to +\infty]{\mu_{B_n}} \fPp_{\alpha}(\Gamma(1+\alpha))\,.
    \end{align*}
\end{thm}

\noindent The main technical challenge lies in verifying condition~\ref{cond_CFPP:tau_n_small_enough}, which in the finite measure case was greatly simplified by the use of inducing. In the infinite measure context, more care is required to control the necessary tail estimates directly in the original system.\\

\noindent We are now ready to prove Theorem \ref{thm:HTS_infinite_FPP_infinitely_recurrent}.

\begin{proof}[Proof (of Theorem \ref{thm:HTS_infinite_FPP_infinitely_recurrent})]
    Let $x\in \Omega$ be a non periodic point and $v \in V$ be such that $x\in \bigcap_{k\geq 0} \bigcup_{n\geq k} T^{-n}[v]$. We denote $(j_k)_{k\geq 0}$ the sequence of hits of $[v]$ (\textit{i.e.}, $j_k := e^{(k)}_{[v]}(x)$). Furthermore, for all $n\geq j_0$, we set $\ell(n) := \max\{ j_k\;|\; j_k \leq n - 2\}$ (\textit{i.e.}, $\ell(n)$ is the last visit of $[v]$ before $n-2$).\\
    
    \noindent As in the positive recurrent case, we build upon Theorem \ref{thm:sufficient_conditions_convergence_compound_FPP} to prove limiting results for the hitting and return REPP. We check that the assumptions $(A)_{\alpha}$ are satisfied. By Theorem \ref{thm:GRPF}, the cylinder $[x_0]$ is a Darling-Kac set and by hypothesis $(a_n) \in RV(\alpha)$. Since $B_n \subset [x_0]$ for all $n\geq 1$ and $[x_0]$ is a uniform set, the first constraint is verified. We are considering the non-periodic case so we take $U(B_n) = \emptyset$ for all $n\geq 0$ meaning that we only need to check \ref{cond_CFPP:good_density_after_tau_n}-\ref{cond_CFPP:cluster_compatible_tau_n_no_cluster_from_Q}.\\ 
    
    \noindent We set $\tau_n := n-1 + e_{[v]}\circ T^{n-1}$. Checking \ref{cond_CFPP:good_density_after_tau_n} is similar to the finite case. We show that for all $n\geq 1$, 
    \begin{align*}
        \widehat{T}^{\tau_n}\Big(\frac{\mathbf{1}_{B_n}}{\mu(B_n)}\Big) \in \mathcal{U}_{K_v, M_v}(v).
    \end{align*}
    \noindent Indeed, by definition of $\tau_n$, we necessarily have $T^{\tau_n}(B_n) \subset [v]$ (remark that $\tau_n$ depends on the point $x$ considered) ensuring $\mathrm{Supp} \,(T^{\tau_n}\mathbf{1}_{B_n}) \subset [v]$. Since the system is recurrent, $e_{[v]}^{(k)}$ is finite almost everywhere for all $k\geq 1$. Thus, for all $n\geq 1$, there exists a collection of admissible words $\mathcal{I}_n$, which can be interpreted as a partition of $[x_0^{n-1}]$ up to the time $\tau_n$, giving the following decomposition of $[x_0^{n-1}]$
    \begin{align*}
        [x_0^{n-1}] =: \bigsqcup_{(y_0^{p}) \in \mathcal{I}_n} [x_0^{n-1}y_0^{p}] 
    \end{align*}
    More precisely, the collection $\mathcal{I}_n$ can be computed explicitly as we have
    \begin{align*}
        \mathcal{I}_n := \{(y_0^{p}) \; \text{admissible} \;|\; p\geq 0, \; x_{n-1} \rightarrow y_0,\; y_p = v,\; \forall 0\leq i < p, \;y_i \neq v\}. 
    \end{align*}
    Note that if $x_{n-1} = v$, $\mathcal{I}_n = \emptyset$ and the decomposition is trivial. For $(y_0^{p}) \in \mathcal{I}_n$, we have $\tau_n = n+p$ on $[x_0^{n-1}y_0^p]$. Hence, we get
    \begin{align*}
        \frac{1}{\mu(B_n)}\widehat{T}^{\tau_n}\mathbf{1}_{B_n} &= \mu(B_n)^{-1} \sum_{y_0^p \in \mathcal{I}_n} \widehat{T}^{n+p}\mathbf{1}_{[x_0^{n-1}y_0^p]} \\
        & = \mu(B_n)^{-1} \sum_{y_0^{p}\in \mathcal{I}_n} \mu[x_0^{n-1}y_0^{p}] \widehat{T}^{n+p}(\mathbf{1}_{[x_0^{n-1}y_0^{p}]}/\mu[x_0^{n-1}y_0^{p}]).
    \end{align*}
    By Lemma \ref{lem:bounded_distortion_CMS_living_in_compact}, this function belongs to $\mathcal{U}_{K_v, M_v}(v)$ and hence \ref{cond_CFPP:good_density_after_tau_n} is satisfied. \\

    \noindent We turn now to \ref{cond_CFPP:tau_n_small_enough}. This is the condition for which we need to make a deeper study than in the finite measure case. This is where we are going to take advantage of the infinite recurrence of the symbol $v$ for $x$. The regular variation hypothesis is also crucial to get the result. Assume first that $x \in [v]$.  By definition of $\ell(n)$, we have 
\begin{align*}
    \tau_n & = e_{[v]} \circ T^{n-1} + n-1 = r_{[v]} \circ T^{\ell(n)} + \ell(n)
\end{align*}
Thus, with the bounded distortion estimate Lemma \ref{lem:bounded_distortion_CMS}, there exists a constant $C > 0$ such that for all $n \geq j_0 +2$, we have
\begin{align*}
    \mu_{B_n}(\gamma(\mu(B_n))\tau_n \geq t) & = \mu_{B_n} (\gamma(\mu(B_n)) (r_{[v]} \circ T^{\ell(n)} + \ell(n)) \geq t)\\
    & =  \mu_{B_n}(T^{-\ell(n)} \{\gamma(\mu(B_n))\,r_{[v]} \geq t - \ell(n)\gamma(\mu(B_n))\})\\
    & = \mu[x_0^{n-1}]^{-1} \mu\big([x_0^{\ell(n)}] \cap T^{-\ell(n)}([x_{\ell(n)}^{n-1}] \cap \{r_{[v]} \geq t\gamma(\mu(B_n))^{-1} - \ell(n)\})\big)
\end{align*}
Now, $[x_{\ell(n)}^{n-1}] \cap \{r_{[v]} \geq t\gamma(\mu(B_n))^{-1} - \ell(n)\}$ is measurable with respect to $\mathcal{C}( t\gamma(\mu(B_n))^{-1}\vee n - \ell(n))$ and we note $\varsigma_n$ its associated partition. It yields

\begin{align*}
    \mu_{B_n}(\gamma(\mu(B_n))\tau_n \geq t) &= \mu[x_0^{n-1}]^{-1} \sum_{[y_0^p] \in \varsigma_n} \mu[x_0^{\ell(n)}y_1^p] \\
    & \leq \mu[x_0^{n-1}]^{-1}e^{W_1(\phi_*)}\mu[v]^{-1} \sum_{[y_0^p] \in \varsigma_n}  \mu[x_0^{\ell(n)}]\mu[y_0^p]\\
    & \leq \mu[x_{\ell(n)}^{n-1}]^{-1}e^{2W_1(\phi_*)}\mu[v]^{-2} \mu([x_{\ell(n)}^{n-1}] \cap \{r_{[v]} \geq t\gamma(\mu(B_n))^{-1} - \ell(n)\})\\
    &\leq  e^{2W_1(\phi_*)}\mu[v]^{-2} \mu_{[x_{\ell(n)}^{n-1}]}(r_{[v]} \geq t\gamma(\mu(B_n))^{-1} - \ell(n))
\end{align*}
where we used twice Lemma \ref{lem:bounded_distortion_CMS}. Hence, it gives the following implication (in fact this is an equivalence as the lower bound can also be deduced from Lemma \ref{lem:bounded_distortion_CMS}).
\begin{align}
    \label{eq:reformulation_proving_delay_time_converges_to_0}
    \gamma(\mu(B_n))\tau_n \xRightarrow[n\to +\infty]{\mu_{B_n}} 0 \quad \text{if and only if} \quad \gamma(\mu(B_n))r_{[v]} \xRightarrow[n \to +\infty]{\mu_{[x_{\ell(n)}^{n-1}]}} 0 \;\text{and}\; \ell(n)\gamma(\mu(B_n)) \xrightarrow[n \to +\infty]{} 0.
\end{align}

\noindent We start be showing that the time $\ell(n)$ is negligible with respect to the renormalization. Recall that we assumed $x_0$ to be infinitely recurrent for $x$.

\begin{lem}
    \label{lem:negligeable_up_to_time_n_x_infinitely_recurrent}
    For all $x \in X$ such that $x_0$ is infinitely recurrent for $x$, we have
    \begin{align*}
        n\gamma(\mu[x_0^{n-1}]) \xrightarrow[n\to +\infty]{} 0.
    \end{align*}
\end{lem}

\begin{proof}[Proof (of Lemma \ref{lem:negligeable_up_to_time_n_x_infinitely_recurrent})]
    We split the proof into two cases depending on the recurrence properties of $x$. 
    \begin{enumerate}[label = (\alph*)]
        \item Assume first that 
        \begin{align}
            \label{eq:recurrence_more_than_ln}
            S_n \mathbf{1}_{[x_0]}/ \log(n) \xrightarrow[n\to +\infty]{} +\infty.
        \end{align}
    Then, the information we have on the induced map is sufficient to get the result. Indeed, since the induced map is $\psi$-mixing, the exponential decay of the measure of cylinders implies
    \begin{align*}
        n\mu[x_0^{n-1}] \xrightarrow[n\to +\infty]{} 0.
    \end{align*}
    This is even stronger than $n \gamma(\mu[x_0^{n-1}]) \xrightarrow[n\to +\infty]{} 0$.
        \item Assume now that 
        \begin{align}
            \label{eq:bad_points_for_returns}
            S_n \mathbf{1}_{[x_0]}/ \log(n) \leq K.
        \end{align}
        for all $n \geq 0$. 
    Set $l_1(n)$ and $l_2(n)$ the two longest waiting times before coming back to $x_0$ up to time $n-1$, that is to say
    \begin{align*}
        l_1(n) &:= \max_{j \;|\; r_{[x_0]}^{(j)}(x) \leq n-1} n \wedge r_{[x_0]}^{(j+1)}(x) - r_{[x_0]}^{(j)}(x).
    \end{align*}
    Define $j_n = j_{n}(x)$ the argmax of $l_1(n)$, $k_n := r_{[x_0]}^{(j_n)}(x) \leq n-1$. Then,
    \begin{align*}
        l_2(n) := \max_{j \;|\; r_{[x_0]}^{(j)}(x) \leq n-1, \; j\neq j_n} n \wedge r_{[x_0]}^{(j+1)}(x) - r_{[x_0]}^{(j)}(x)
    \end{align*}
    and similarly define $j_n^{(2)} = j_n^{(2)}(x)$ the argmax of the previous quantity and $k_n^{(2)} = r_{[x_0]}^{(j^{(2)}_n)}(x)$. Of course, $j_n^{(2)}$ and $k_n^{(2)}$ are well defined for $n > r^{(2)}_{[x_0]}(x)$. Because $x_0$ is infinitely recurrent for $x$, we have $j_n, \; j_n^{(2)}, \; k_n, \; k_n^{(2)} \xrightarrow[n\to +\infty]{} +\infty$. Since we assumed \eqref{eq:bad_points_for_returns}, we have
    \begin{align*}
        l_1(n) \geq \frac{n}{K\log(n)}.
    \end{align*}
    We again split into two different cases. First, if $l_1(n) \geq n/2$. By bounded distortion Lemma \ref{lem:bounded_distortion_CMS}, we get
    \begin{align*}
        \mu[x_0^{n-1}] &\leq C \mu[x_0^{k_n}]\mu[x_{k_n}^{n-1}] \\
        & \leq C \mu[x_0^{k_n}]\mu\Big[x_{k_n}^{(n-1)\wedge r_{[x_0]}^{(j_n + 1)}(x) -1} \Big]\\
        & \leq C\mu[x_0^{k_n}]\mu([x_0] \cap \{r_{[x_0]} \geq l_1(n)\}) \\
        & \leq C\mu[x_0^{k_n}]\mu([x_0] \cap \{r_{[x_0]} \geq n/2\}) \\
        & \leq C'\mu[x_0^{k_n}]a\big(n/2\big)^{-1} \quad \text{by Lemma \ref{lem:equivalence_queue_renormalisation_gamma}}.
    \end{align*}
    On the other side, if $l_1(n) < n/2$, we have 
    \begin{align*}
        l_2(n) \geq \frac{n}{4K\log(n)}.
    \end{align*}
    Hence, assuming without loss of generality that $j_n \leq j_n^{(2)}$ and using again the bounded distortion Lemma \ref{lem:bounded_distortion_CMS} (twice this time), we get
    \begin{align*}
        \mu[x_0^{n-1}] &\leq C^2\mu[x_0^{k_n}] \mu[x_{k_n}^{k_n^{(2)}}]\mu[x_{k^{(2)}_n}^{n-1}] \\
        & \leq C^2 \mu[x_0^{k_n}] \mu([x_0] \cap \{r_{[x_0]} \geq l_1(n)\})\, \mu([x_0] \cap \{r_{[x_0]} \geq l_2(n)\} \\
        & \leq C^2 \mu[x_0^{k_n}] \,\mu([x_0] \cap  \{r_{[x_0]} \geq n/(K\log(n))\})\, \mu([x_0] \cap \{r_{[x_0]} \geq n/(4K\log(n))\}) \\
        & \leq C' a\big(n/(K\log(n))\big)^{-1}a\big( n/(4K\log(n))\big)^{-1}
    \end{align*}
    \noindent Using the regular variation hypothesis on $a$ and the definition $\gamma$ with respect to $a$, for all $\varepsilon > 0$ and $n$ large enough so that $\mu[x_0^{k_n}] \leq \varepsilon$, we obtain
    \begin{align*}
        n \gamma(\mu[x_0^{n-1}]) &\leq n\gamma(C'\varepsilon a(n/2)^{-1}) \\
        & \leq C\varepsilon^{1/\alpha} n  \big(a^{\leftarrow}(a(n/2))\big)^{-1} \\
        &\lesssim C\varepsilon^{1/\alpha}
    \end{align*}
    On the other side, we also have 
    \begin{align*}
        & \frac{a(n)}{a\big(n/(K\log(n))\big)a\big( n/(4K\log(n))\big)}\\
        & = \frac{n^{\alpha}L(n)}{(n/(K\log(n)))^{\alpha}(n/(4K\log(n)))^{\alpha} L(n/(K\log(n)))L(n/(4K\log(n)))}\\
        & \isEquivTo{n \to +\infty} \frac{4K^2n^{\alpha}\log(n)^{2\alpha}}{n^{2\alpha}} \frac{L(n)}{L(n/\log(n))^2}\\
        & \xrightarrow[n\to +\infty]{} 0
    \end{align*}
    Thus, for $n$ large enough, we have 
    \begin{align*}
        n \gamma(\mu[x_0^{n-1}]) \leq n \gamma(\varepsilon a(n)) \lesssim \varepsilon^{1/\alpha}.
    \end{align*}
    Hence $n\gamma(\mu[x_0^{n-1}]) \xrightarrow[n\to +\infty]{} 0$.  
    \end{enumerate}
\end{proof}

\begin{rem}
    In the proof of Lemma~\ref{lem:negligeable_up_to_time_n_x_infinitely_recurrent}, one could consider splitting at a finer scale than $\log(n)$ to better exploit the rescaling effect of $\gamma$. However, this does not alter the fundamental need to adopt different strategies for controlling the scaling: one for “typical” points that return sufficiently quickly, and another for “exceptional” points that take longer to return but for which the measures of the associated cylinders decay faster.
\end{rem}

\noindent Lemma \ref{lem:negligeable_up_to_time_n_x_infinitely_recurrent} proves the second part of \eqref{eq:reformulation_proving_delay_time_converges_to_0} as we have $\ell(n) \leq n$ by definition. This first part is taken care of by the following Lemma.

\begin{lem}
    \label{lem:second_part_reformulation_proving_delay_time_converges_to_0}
    We have 
    \begin{align*}
        \gamma(\mu(B_n))r_{[v]} \xRightarrow[n \to +\infty]{\mu_{[x_{\ell(n)}^{n-1}]}} 0.
    \end{align*}
\end{lem}

\begin{proof}[Proof (of Lemma \ref{lem:second_part_reformulation_proving_delay_time_converges_to_0})]
    Let $t > 0$. We have 
    \begin{align*}
        \mu([x_{\ell(n)}^{n-1}] \cap \{r_{[v]} > t/\gamma(\mu(B_n))\}) &\leq \mu([v] \cap \{r_{[v]} > t/\gamma(\mu(B_n))\}) \quad \text{because $x_{\ell(n)} = v$}\\
        & \lesssim \frac{\sin(\pi \alpha)}{\pi\alpha} t^{-\alpha} \mu(B_n) \quad \text{by Lemma \ref{lem:equivalence_queue_renormalisation_gamma}}.
    \end{align*}
    On the other side, we have 
    \begin{align*}
        \mu(B_n)/ \mu[x_{\ell(n)}^{n-1}] \leq e^{W_1(\phi_*)}\mu[v]^{-1}\mu[x_0^{\ell(n)}] \quad \text{by Lemma \ref{lem:bounded_distortion_CMS}},
    \end{align*}
    and $\mu[x_0^{\ell(n)}] \xrightarrow[n \to +\infty]{} 0$ as $\ell(n) \xrightarrow[n\to +\infty]{} +\infty$ because $x_0$ is infinitely recurrent for $x$.
\end{proof}

\noindent Hence, Lemma \ref{lem:second_part_reformulation_proving_delay_time_converges_to_0} concludes the proof of the right part of \eqref{eq:reformulation_proving_delay_time_converges_to_0} and it shows that 
\ref{cond_CFPP:tau_n_small_enough} is satisfied for this choice of $\tau_n$. \\

Now if $x \notin [v]$ but is infinitely recurrent for the symbol $v$, the same idea can be applied. We set $j := j_0 = r_{[v]}(x)$. The difference between the case $x \in [v]$ is that we need to introduce a further delay $j$ before seeing $v$. Fortunately, this delay is constant and negligible in the limit. Indeed, assume that $n > r_{[v]}(x) = j_0 =: j$ and let $t > 0$. We have 
    \begin{align*}
        & \mu_{[x_0^{n-1}]}(\gamma(\mu(B_n))\tau_n \geq t) = \mu[x_0^{n-1}]^{-1} \mu([x_0^{n-1}] \cap \{\gamma(\mu(B_n))(n-1 + e_{[v]}\circ T^{n-1}) \geq t\}) \\
        & = \mu[x_0^{n-1}]^{-1}\mu([x_0^{n-1}] \cap T^{-j} \{\gamma(\mu(B_n))(n-1 + j + e_{[v]} \circ T^{n-1-j}) \geq t\}) \\
        & = e^{\pm C} \mu[x_j^{n-1}]^{-1}\mu([x_j^{n-1}] \cap \{\gamma(\mu(B_n))(n-1 + e_{[v]}\circ T^{n-1-j}) \geq t + j\gamma(\mu(B_n))\}) \\
        & = e^{\pm C} \mu_{[x_j^{n-1}]}(n-1 + e_{[v]}\circ T^{n-1 - j}\geq t/\gamma(\mu(B_n)) - j),
    \end{align*}
    where we used bounded distortion estimates from Lemma \ref{lem:bounded_distortion_CMS} the same way we used it for the delay $\ell(n)$. Now, the point $x' = T^jx \in [v]$ and thus the right end side goes to 0 (by considering any $0 < t' < t$ because $j\gamma(\mu(B_n)) \to 0$ when $n \to +\infty$), proving the result for $x$.\\

\noindent Finally, we need to check \ref{cond_CFPP:cluster_compatible_tau_n_no_cluster_from_Q}. We split with respect to the return to $[x_0^{n-1}]$. For this property, the proof is similar to the finite case thanks to our definition of $\tau_n$. We have
\begin{align*}
    \mu_{B_n}(r_{B_n} \leq \tau_n) & = \mu_{B_n}(r_{B_n} < n) + \mu_{B_n}(n \leq r_{B_n} \leq \tau_n) \\
    & \leq \mu_{B_n}(r_{B_n} < n) + \mu_{B_n}( T^{-(n-1)}(r_{B_n} \leq r_{[v]})).
\end{align*}

\noindent We deal with the second term. Since $x_j = v$, we have, for $n > j$ (recall that $j = j_0 = e_{[v]}(x)$), $\{r_{[x_0^{n-1}]} \leq r_{[v]}\} \subset \{r_{[x_0^{n-1}]} = r_{[x_0^j]}\}$. Hence, 
\begin{align*}
    \mu_{B_n}(T^{-(n-1)}(r_{B_n} \leq r_{[v]})) \leq \mu_{B_n}(T^{-(n-1)}\{r_{[x_0^{n-1}]} = r_{[x_0^j]}\})
\end{align*}
Taking a partition with respect to the first return $[x_0^j]$ and since $x_j = v$, we can obtain, using distortion estimates of Lemma \ref{lem:bounded_distortion_CMS}, we obtain 
\begin{align*}
    \mu_{B_n}(T^{-(n-1)} \{r_{[x_0^{n-1}]} = r_{[x_0^j]}\}) \leq C \mu[x_j^{n-1}] \xrightarrow[n\to +\infty]{} 0.
\end{align*}

\noindent On the other side, we study $\mu_{B_n}(r_{B_n} < n)$. If $r(x_0^{n-1}) \geq n$, then there is nothing to prove. Else, assume that $r(x_0^{n-1}) < n$. By Lemma \ref{lem:calcul_EI}, since $x$ is non periodic, we know that 
\begin{align*}
    \mu_{B_n}(r_{B_n} = r(B_n)) \xrightarrow[n\to +\infty]{} 0. 
\end{align*}
Furthermore, because $x$ is non periodic, we have $r(x_0^{n-1}) \xrightarrow[n\to +\infty]{} +\infty$. Having $r(x_0^{n-1}) \leq n$ for an infinite number of $n$ implies that $x$ is infinitely recurrent for $x_0$ and hence, we can set $v = x_0$. Then, for $n$ such that $r(x_0^{n - 1}) \leq n$, we have with $q_nr(x_0^{n-1}) + r_n = n$ the euclidean division of $n$ par $r(x_0^{n-1})$, 
\begin{align*}
    \mu_{[x_0^{n - 1}]}(r(x_0^{n-1})< r_{[x_0^{n - 1}]} < n) &\leq \sum_{k = q_nr(x_0^{n-1})}^{n-1} \mu_{[x_0^{n-1}]}(T^{-k}[x_0^{n-1}]) \\
    &\leq \sum_{\overset{q_nr(x_0^{n-1}) \leq k \leq n-1}{x_k = x_0}} \mu[x_0^{n-1}]^{-1}\mu[x_0^{k-1}x_0^{n-1}] \\
    & \leq \sum_{\overset{q_nr(x_0^{n-1}) \leq k \leq n-1}{x_k = x_0}} C \mu[x_0^{k-1}x_0] \quad \text{by Lemma \ref{lem:bounded_distortion_CMS}}\\
    & \leq \sum_{k = m_n}^{+\infty} C e^{-ck} \xrightarrow[n\to +\infty]{} 0,
\end{align*}
where $m_n = |\{1\leq k \leq q_n r(x_0^{n-1})\;|\; x_k = x_0\}|$ and with the exponential decay of the measure of cylinders for the induced map.\\

\noindent This completes the proof of \ref{cond_CFPP:cluster_compatible_tau_n_no_cluster_from_Q} and we can apply Theorem \ref{thm:sufficient_conditions_convergence_compound_FPP} to get Theorem \ref{thm:HTS_infinite_FPP_infinitely_recurrent}.
\end{proof}

\paragraph{Periodic points.}
\label{subsubsection:CFPP_CMS_periodic_points}

\noindent We turn now to the proof of Theorem \ref{thm:HTS_infinite_CFPP_periodic_points} and the study of periodic points. As always, the periodicity prevents a simple point process as a limit. However, its effect is purely local and the clusters are the same both in the finite and infinite setting. It leads to the following theorem giving the limit of the hitting  and return REPPs when the targets are cylinders shrinking to a periodic point. Apart from the difficulty caused by the presence of clusters, periodic points are in fact easier to study. Indeed, the main change with respect to the finite setting is the proof of \ref{cond_CFPP:tau_n_small_enough}. However, it is easy to check for points that exhibit good recurrence properties and due to their periodicity, periodic points have even better recurrence properties than most points making \ref{cond_CFPP:tau_n_small_enough} easier to deal with.

\begin{proof}[Proof (of Theorem \ref{thm:HTS_infinite_CFPP_periodic_points})]
    Since $x$ is $q$-periodic, we set $V_x := \{v\in V\;|\; \exists n\geq 0, \,x_n = v\}$. Then $|V_x| < +\infty$. The definition of $U(B_n)$ and $Q(B_n)$ are the same as in the finite setting and we set $Q(B_n) = B_n \cap T^{-q}B_n^c$ and $U(B_n) = B_n \cap T^{-q}B_n  = [x_0^{n+q-1}]$. By construction, $U(B_n)$ and $Q(B_n)$ are measurable with respect to $\mathcal{C}(n+q)$. By Lemma \ref{lem:calcul_EI}, we have \ref{cond_CFPP:extremal_index} and $\theta$ is equal to $1 - \exp(S_q\phi(x) - qP_G(\phi))$. Note that we have 
    \begin{align*}
        Q(B_n) = [x_0^{n-1}] \cap T^{-n}[x_n]^c \sqcup \cdots \sqcup [x_0^{n +q-2}] \cap T^{-(n+q-1)}[x_{n+q-1}]^c \\
        =: Q_0(B_n) \sqcup \cdots \sqcup Q_{q-1}(B_n).
    \end{align*}
    We set $\tau_n := n+k-1$ on $Q_k(B_n)$. 
    Assuming without loss of generality that $\mu(Q_k(B_n)) > 0$ for $0\leq k \leq q-1$ (else, we simply do not count it), we have
    \begin{align*}
        \mu(Q(B_n))^{-1}\,\widehat{T}^{\tau_n} \mathbf{1}_{Q(B_n)} = \sum_{k = 0}^{q-1} \frac{\mu(Q_k(B_n))}{\mu(Q(B_n))} \mu(Q_k(B_n))^{-1} \widehat{T}^{n+k-1}\mathbf{1}_{Q_k(B_n)}
    \end{align*}
    Yet $\mu(Q_k(B_n))^{-1} \widehat{T}^{n+k-1}\mathbf{1}_{Q_k(B_n)}$ belongs in the compact set $\mathcal{U}_{K,M}([x_{n+k-1}] \cap T^{-1}[x_{n+k}]^c)$ by an direct adaptation of the proof of Lemma \ref{lem:bounded_distortion_CMS_living_in_compact}, where $\mathcal{U}_{K,M}([x_{n+k-1}] \cap T^{-1}[x_{n+k}]^c)$ is constructed as $\mathcal{U}_{K,M}(v)$ but assuming that the support is on $[x_{n+k-1}] \cap T^{-1}[x_{n+k}]^c$ and same for the continuity hypothesis (note that here we find $M = e^{W_1(\phi_*)}/(\mu[x_{n+k-1}] - \mu[x_{n+k-1}x_{n+k}])$ which is finite because we assume that $\mu(Q_k(B_n)) > 0$). Note that 
    \begin{align*}
        \bigcup_{v, v'\in V_x} \mathcal{U}_{K_{v,v'}, M_{v,v'}}(v \cap T^{-1}[v']^c),
    \end{align*}
    is precompact in $L^1(\mu)$ because $V_x$ is finite and thus so is the closure of its convex hull. We obtain that $\mu(Q(B_n))^{-1}\,\widehat{T}^{\tau_n} \mathbf{1}_{Q(B_n)}$ remains in the same compact set for all $n\geq 1$, hence proving \ref{cond_CFPP:good_density_after_tau_n}.
    
    \noindent For \ref{cond_CFPP:tau_n_small_enough}, we can use the recurrence properties of the periodic point. Indeed, it satisfies $S_n\mathbf{1}_{[x_0]}(x) \geq n/q$ for all $n \geq 0$ and thus by exponential decay of the measure of cylinders for the induced system $([x_0], T_{x_0}, \mu_{x_0})$, we have 
    \begin{align*}
        \gamma(\mu(B_n))\,\tau_n 
        & = \gamma(\mu[x_0^{n+q-1}])(n-1+q) \leq (n-1+q)\mu[x_0^{n+q-1}]\\
        & \leq C(n+q-1) \, e^{-cS_{n+q}\mathbf{1}_{[x_0]}(x)} \xrightarrow[n \to +\infty]{} 0.
    \end{align*}

    \noindent For \ref{cond_CFPP:cluster_compatible_tau_n_no_cluster_from_Q}, we have $r(x_0^{n-1}) = q$ for $n \geq q$ by Lemma \ref{lem:periodicity_equivalent_to_finite_topological_return}. Let $n =: m_n q + r$ the Euclidean division of $n$ by $q$. We have 
    \begin{align*}
        \mu_{Q(B_n)}(r_{B_n} \leq n + q - 1) &\leq \mu(Q(B_n))^{-1} \sum_{m_nq \leq k \leq n + q - 1} \mu(Q(B_n) \cap T^{-k}B_n) \\
        & \lesssim \theta^{-1}\mu(B_n)^{-1}  e^{W_1(\phi_*)} \mu[x_0]^{-1} 2q \mu[x_0^{m_n q}]\mu(B_n) \\
        & \lesssim C \mu[x_0^{m_n q}] \xrightarrow[n\to +\infty]{} 0. 
    \end{align*}
    Finally, \ref{cond_CFPP:cluster_compatible_tau_n_cluster_from_U} and \ref{cond_CFPP:Compatibility_Geometric_law} can be proven the exact same way as in the finite case. Thus the conclusion of Theorem \ref{thm:HTS_infinite_CFPP_periodic_points} is a consequence of Theorem \ref{thm:sufficient_conditions_convergence_compound_FPP}.
\end{proof}

\subsection{Proof of Theorem \ref{thm:sufficient_conditions_convergence_other_point_processes}}
\label{section:proof_CMS_allowing_for_non_zero_times}

To prove Theorem \ref{thm:sufficient_conditions_convergence_other_point_processes}, we will study an equivalent stochastic process that keep track of the sequence of returns, that it to say, for a sequence of asymptotically rare events $(B_n)_{n\geq 1}$, we define the associated Rare Event Stochastic Process (RESP) with scaling $\gamma$ as the sequence of functions $\gamma(\mu(B_n))\, \Phi_{B_n}$ taking values in $\overline{\mathbb{R}}_+^{\mathbb{N}}$ by
\begin{align*}
    \gamma(\mu(B_n))\, \Phi_{B_n} := \big(\gamma(\mu(B_n))\,r_{B_n}^{(1)}, \gamma(\mu(B_n))\, r_{B_n}^{(2)}, \dots\big)
\end{align*}

Of course, similarly to the REPP, we can define the notion of hitting RESP and return RESP if we consider it as a sequence of random variable on the probability spaces $(X, \nu)$ or $(B_n, \mu_{B_n})$ respectively. In fact, if the limit RESP belongs to the set $\{(t_i)_{i\geq 1} \in (\mathbb{R}_+)^{\mathbb{N}} \;|\; t_i \leq t_{i+1} \; \forall i\geq 0 \; \text{and} \; \lim_{i\to +\infty}t_i = +\infty\}$ almost surely, then the convergence of the (hitting or return) RESP is equivalent to the convergence of the  (hitting or return) REPP $(N_{B_n}^{\gamma})_{n\geq 1}$ and there is a one-to-one connection between the limit processes. For more details on this connection, we refer to \cite[Appendix A.2]{BT25_PhD}. \\

The RESP is sometimes more suitable to study the convergence because the convergence of such a process is equivalent to the convergence of its finite dimensional marginals.

Theorem \ref{thm:sufficient_conditions_convergence_other_point_processes} provides a generalization of Theorem \cite[Theorem 2.3]{BT24_FractionalPoisson} (in the absence of cluster) is crucial in the proof of Theorem \ref{thm:HTS_infinite_points_that_never_come_back_statements} and to establish the \textit{all-point REPP} property for the examples in Section \ref{section:all-point_REPP_for_examples}. As explained, it is easier to work with the RESP $(\gamma(\mu(B_n))\,\Phi_{B_n})_{n \geq 1}$ taking values in $(\mathbb{R}_+)^{\mathbb{N}}$. We recall the Theorem giving the relationship between limiting behavior of the hitting RESP and the return RESP for PDE systems (with regular variation).

\begin{thm}[\textup{\cite[Theorem 2.1]{BT24_FractionalPoisson}}]
\label{thm:HTS-REPP_vs_RTS-REPP_infinite_measure_renormalized_measure}
Let $(X,\mathscr{B},\mu, T)$ be a PDE CEMPT with $\mu(X)=+\infty$ with normalizing sequence $(a_n)_{n\geq 1} \in \RV(\alpha)$ for some $0 < \alpha \leq 1$. Let $Y$ be a uniform set and $(B_n)_{n\geq 0}$ be a sequence of asymptotically rare events satisfying included in $Y$ for all $n\geq 1$.\\
Let $\Phi, \widetilde{\Phi}$ be stochastic processes in $\{(t_i)_{i\geq 1} \in (\overline{\mathbb{R}}_+)^{\mathbb{N}}\;|\; t_i\leq t_{i+1} \; \text{for all $i\geq 1$}\}$. Then,
\begin{align*}
\gamma(\mu(B_n))\, \Phi_{B_n} \xRightarrow[n\to +\infty]{\mu_{B_n}} \widetilde{\Phi} \quad \text{if and only if} \quad 
\gamma(\mu(B_n))\, \Phi_{B_n} \xRightarrow[n\to +\infty]{\mathcal{L}(\mu)} \Phi. 
\end{align*}
Moreover, the distributions of $\Phi$ and $\widetilde{\Phi}$ uniquely determine each other in the following way. For all $d \geq 1$, denoting $F^{[d]}$ 
(respectively $\widetilde{F}^{[d]}$) the distribution function of the $d$ first coordinates of $\Phi$ (respectively 
$\widetilde{\Phi}$), we have, for all $0 \leq t_1 \leq \dots \leq t_d$,
\begin{align}
\label{eq:relation_hitting_return_infinite_measure}
F^{[d]}(t_1,\dots, t_d)
&= \alpha \int_0^{t_1} \biggl( \widetilde{F}^{[d-1]}\left(t_2 -t_1 + x, \dots, t_d - t_1 + x\right) \nonumber\\
& \quad\qquad \qquad - \widetilde{F}^{[d]}\left(x, t_2 - t_1 + x,\dots, t_d - t_1 +x\right)\biggr) 
(t_1 - x)^{\alpha - 1}\,\dd x\,,
\end{align}
with the convention $\widetilde{F}^{[0]} = 1$.
\end{thm}

We also recall the following lemma giving a uniform control of the convergence to $0$ for the average of the iterations by the transfer operator for functions in a compact subset of $L^1(\mu)$ when the system is a CEMPT.

\begin{lem} \textup{\cite[Theorem 3.1]{Zwe07_InfiniteMeasurePreservingTransformationsWithCompactFirstRegeneration}}
    \label{lem:convergence_Birkhoff_transfer_operator_compact}
    Let $(X, \mathscr{B}, \mu, T)$ be a CEMPT and $\mathcal{U}$ a compact subset of $L_1(\mu)$ such that $\int u\,\dd\mu = 1$ for all $u\in \mathcal{U}$. Then, uniformly in $u, u^* \in \mathcal{U}$, we have 
    \begin{align*}
        \left\| \frac{1}{M}\sum_{j = 0}^{M-1} \widehat{T}^ju - \frac{1}{M}\sum_{j = 0}^{M-1} \widehat{T}^ju^*\right\|_{L^1(\mu)} \xrightarrow[M\to +\infty]{} 0.
    \end{align*}
\end{lem} 
We are now ready to prove Theorem \ref{thm:sufficient_conditions_convergence_other_point_processes}.

\begin{proof}[Proof (of Theorem \ref{thm:sufficient_conditions_convergence_other_point_processes})]
  We consider the RESP $(\gamma(\mu(B_n))\Phi_{B_n})_{n\geq 1}$ on $\overline{\mathbb{R}}_+$ and the sequence of probability spaces $(B_n,\mu_{B_n})$, \textit{i.e.} we look at the successive return times. As $(\overline{\mathbb{R}}_+)^{\mathbb{N}}$ is compact, this sequence is tight so up to a subsequence and the convergence of the point process is characterized by the convergence of its finite-dimensional marginals, we can assume that the family of distribution functions $(\widetilde{F}_{B_n}^{[d]})_{d\geq 1}$ converges towards the family of functions $(\widetilde{F}^{[d]})_{d\geq 0}$ (meaning that we have pointwise convergence for each one of them at the continuity points of $\widetilde{F}^{[d]}$). We are going to show that the only possible limits are the distribution functions of the successive return times of the stochastic process $\Phi_{\RPP(\widetilde{J}_{\alpha}(\nu))}$, which is enough to get the convergence towards this process. By Theorem \ref{thm:HTS-REPP_vs_RTS-REPP_infinite_measure_renormalized_measure}, for any given 
density \( u \), let \((F_{B_n, v}^{[d]})_{d \geq 1}\) denote the family of distribution functions 
corresponding to the finite-dimensional marginals of the stochastic process
\((\gamma(\mu(B_n))\Phi_{B_n})\) on the probability space \((X, \mu_{v})\), where \(\mu_v\) is a 
probability measure absolutely continuous with respect to \(\mu\), having density \( v \). Then, 
the family of renormalized distribution functions \((\widetilde{F}_{B_n,v}^{[d]})_{d \geq 1}\) 
converges to the family of functions \((F^{[d]})_{d \geq 1}\). 
Both \((\widetilde{F}_{B_n,v}^{[d]})_{d \geq 1}\) and \((F^{[d]})_{d \geq 1}\) satisfy the 
relationship given in \eqref{eq:relation_hitting_return_infinite_measure}.

      For all $d\geq 1$ and $0\leq t_1\leq \dots \leq t_d$ such that $(t_1,\dots, t_d)$ is a continuity point of $\widetilde{F}^{[d]}$, we have 
    \begin{align}
        \widetilde{F}_{B_n}^{[d]}(t_1,\dots, t_d) =\mu_{B_n}\big(\gamma(\mu_{B_n})\,r_{B_n} \leq t_1,\dots, \gamma(\mu(B_n))\,r_{B_n}^{(d)} \leq t_d\big)\,.
    \end{align}

    \noindent By hypothesis, there exists a function $u \in L^1(\mu)$ such that $Y$ is $u$-uniform (without loss of generality we assume $\int u \,\dd\mu = 1$). We consider $\mu_{v_n}$ the probability absolutely continuous with respect to $\mu$ and of density $v_n :=  \widehat{T^{\tau_n}}(\mathbf{1}_{B_n}/\mu(B_n))$ and write $(F_{B_n,v_n}^{[d]})_{d\geq 1}$ the family of distribution functions of $\gamma(\mu(B_n))\Phi_{B_n}$ drawn from $\mu_{v_n}$. 
    By \ref{cond_OtherLimits:good_density_after_tau_n}, for all $n\geq 1$, $v_n\in \mathcal{U}$.
    
    \begin{lem}
        \label{lem:same_limit_as_hitting_if_we_start_at_tau_n}
        We have the following asymptotical result
        \begin{align*}
            (\gamma(\mu(B_n)) \Phi_{B_n})_{\#}\mu_{v_n}
        \end{align*}
        and 
        \begin{align*}
            (\gamma(\mu(B_n))\Phi_{B_n})_{\#}\mu_v
        \end{align*}
        share the same limit, if it exists.
    \end{lem}

    \begin{proof}[Proof (of Lemma \ref{lem:same_limit_as_hitting_if_we_start_at_tau_n})]
        We denote $(F_{B_n,v_n}^{[d]})_{d\geq 1}$ and $(F_{B_n,v}^{[d]})_{d\geq 1}$ their respective family of distribution functions of the finite-dimensional marginals. By the Portmanteau theorem, the convergence in law is implied by the convergence for every bounded Lipschitz function. Thus, for every $d\geq 1$, consider a bounded Lipschitz function $\psi : \mathbb{R}^d\to \mathbb{R}^d$. We are going to show that 
    \begin{align*}
        \int \psi \circ (\gamma(\mu(B_n))\Phi^{[d]}_{B_n})\, u\,\dd\mu - \int \psi \circ (\gamma(\mu(B_n))\Phi^{[d]}_{B_n})\, u^*\,\dd\mu \xrightarrow[n\to +\infty]{} 0 \quad \text{uniformly in $u,u^* \in \mathcal{U}$.}
    \end{align*}
      Let $\varepsilon > 0$ and consider $M$ large enough so that the quantity in Lemma \ref{lem:convergence_Birkhoff_transfer_operator_compact} is smaller than $\varepsilon$. It yields,
    \begin{align*}
        \left|\int \psi \circ \big(\gamma(\mu(B_n))\Phi^{[d]}_{B_n}\big) \cdot \frac{1}{M}\sum_{j = 0}^{M - 1} \big(\widehat{T}^ju - \widehat{T}^ju^*\big)\,\dd\mu \right| & \leq \sup |\psi| \left\| \frac{1}{M}\sum_{j = 0}^{M - 1} \big(\widehat{T}^ju - \widehat{T}^ju^* \big) \right\|_{L^1(\mu)}\\
        & \leq \varepsilon \sup |\psi| .
    \end{align*}
    Furthermore, 
    \begin{align*}
        &\left| \int \psi \circ \big(\gamma(\mu(B_n))\Phi^{[d]}_{B_n}\big) \Big(u - \frac{1}{M}\sum_{j= 0}^{M-1}\widehat{T}^ju\Big)\,\dd\mu\right| \\
        & \qquad\leq \frac{1}{M} \sum_{j=0}^{M-1} \int \left| \psi \circ (\gamma(\mu(B_n))\Phi^{[d]}_{B_n}) - \psi \circ \big(\gamma(\mu(B_n))\Phi^{[d]}_{B_n}\big) \circ T^{j} \right|u\,\dd\mu \\
        & \qquad \leq \frac{1}{M} \sum_{j= 0}^{M-1} \left( 2 \sup |\psi| \int_{\{r_{B_n} \leq j\}} u\,\dd\mu + \lip(\psi)\,\gamma(\mu(B_n))j\right) \\
        & \qquad \leq 2\sup |\psi| \int_{\{r_{B_n}\leq M\}} u\,\dd\mu + \lip(\psi) \,\gamma(\mu(B_n))M \\
        & \qquad \leq \varepsilon, \; \text{for $n$ large enough and uniformly on $u\in \mathcal{U}$ since $\mathcal{U}$ is uniformly integrable.}
    \end{align*}
    
      Thus, by Portemanteau theorem, for every $d\geq 1$
    \begin{align}
        \label{eq:second_step_equivalence_convergence_cluster_chap5}
        F_{B_n,v_n}^{[d]}(t_1,\dots, t_d) - F_{B_n,v}^{[d]}(t_1,\dots, t_d) \xrightarrow[n\to +\infty]{} 0.
    \end{align}

    \end{proof}

    \begin{lem}
        \label{lem:independence_waiting_time_and_shift_process}
        Under the distribution $\mu_{B_n}$, the two random variables $\gamma(\mu(B_n))\Phi_{B_n} \circ T^{\tau_n}$ and $\gamma(\mu(B_n))\tau_n$ are asymptotically independent. This is equivalent to say that 
        \begin{align*}
            (\gamma(\mu(B_n))\Phi_{B_n} \circ T^{\tau_n}, \gamma(\mu(B_n))\tau_n) \xRightarrow[n\to +\infty]{\mu_{B_n}} (\Phi, W)
        \end{align*}
        where $\Phi$ and $W$ are as above and independent.
    \end{lem}

    \begin{proof}[Proof (of Lemma \ref{lem:independence_waiting_time_and_shift_process})]
        Let $d\geq 1$ and $t_1, \dots, t_d > 0$. Let $E$ be an element of a $\pi$-system generating $\mathbb{R}_+$ such that $\mathbb{P} (W \in \partial E) = 0$. Then, for all $n \geq 1$, we have 
        \begin{align*}
            &\mu_{B_n} \Big( \gamma(\mu(B_n))\Phi_{B_n}^{[d]} \circ T^{\tau_n} \leq (t_1, \dots, t_d), \; \gamma(\mu(B_n))\tau_n \in E \Big) \\
            &= \sum_{k \geq 0} \mu(B_n)^{-1} \int_{B_n} \mathbf{1}_{\{\gamma(\mu(B_n))\Phi_{B_n}^{[d]} \leq (t_1, \dots, t_d)\}}\circ T^{k}\cdot  \mathbf{1}_{\gamma(\mu(B_n))k \in E}\mathbf{1}_{\{\tau_n = k\}}\;\dd\mu \\
            & = \mu_{B_n}(\gamma(\mu(B_n))\tau_n\in E) \int \bigg(\mathbf{1}_{\{\gamma(\mu(B_n))\Phi_{B_n}^{[d]} \leq (t_1, \dots, t_d)\}} \times \\
            & \quad \sum_{\overset{\gamma(\mu(B_n))k\in E}{B_n \cap \{\tau_n = k\} \neq \emptyset}}\frac{\mu(B_n \cap \{\tau_n = k\})}{\mu(B_n \cap \{\gamma(\mu(B_n))\tau_n \in E\})} \widehat{T}^k \Big(\frac{\mathbf{1}_{B_n \cap \{\tau_n = k\}}}{\mu(B_n \cap \{\tau_n = k\})}\Big) \bigg) \,\dd\mu. 
        \end{align*}
        \ref{cond_OtherLimits:good_density_after_tau_n} ensures that for all $n \geq 1$, 
        \begin{align*}
            v'_n := \sum_{\overset{\gamma(\mu(B_n))k\in E}{B_n \cap \{\tau_n = k\} \neq \emptyset}}\frac{\mu(B_n \cap \{\tau_n = k\})}{\mu(B_n \cap \{\gamma(\mu(B_n))\tau_n \in E\})} \widehat{T}^k \Big(\frac{\mathbf{1}_{B_n \cap \{\tau_n = k\}}}{\mu(B_n \cap \{\tau_n = k\})}\Big) \in \mathcal{U}.
        \end{align*}
        Thus, 
        \begin{align*}
            & \mu_{B_n} \Big( \gamma(\mu(B_n))\Phi_{B_n}^{[d]} \circ T^{\tau_n} \leq (t_1, \dots, t_d), \; \gamma(\mu(B_n))\tau_n \in E \Big) \\
            & = \mu_{v'_n}\big( \gamma(\mu(B_n))\Phi_{B_n}^{[d]} \leq (t_1, \dots, t_d)\big)\mu_{B_n}(\gamma(\mu(B_n))\tau_n \in E)\\
            & \xrightarrow[n\to +\infty]{} \mathbb{P}\big(\Phi^{[d]} \leq (t_1, \dots, t_d)\big) \mathbb{P}(W \in E).
        \end{align*}    
    \end{proof}

    \begin{rem}
        \label{rem:condition_B_1_can_be_replaced}
        As seen in the proof of Lemma \ref{lem:independence_waiting_time_and_shift_process}, we could have replace \ref{cond_OtherLimits:good_density_after_tau_n} by the condition
        \begin{align*}
            v'_{n, E} := \widehat{T}^{\tau_n} \bigg(\frac{\mathbf{1}_{B_n \cap \{\gamma(\mu(B_n))\tau_n \in E\}}}{\mu(B_n \cap \{\gamma(\mu(B_n))\tau_n \in E\})}\bigg) \in \mathcal{U}
        \end{align*}
        for all $n \geq 1$ and $E$ in the $\pi$-system generating the topology such that $\mathbb{P}(W \in \partial E) = 0$. However, this is not very restrictive as in all our example \ref{cond_OtherLimits:good_density_after_tau_n} will hold.\\
        \noindent Furthermore, \ref{cond_OtherLimits:good_density_after_tau_n} is also not too restrictive in comparison with \ref{cond_CFPP:good_density_after_tau_n}, as it is also satisfied in all of our examples.
    \end{rem}

    \begin{cor}
        \label{cor:limit_of_the_process_sum_waiting_time_and_shifted}
        In particular,
        \begin{align*}
            \gamma(\mu(B_n))\Phi_{B_n} \circ T^{\tau_n} + \gamma(\mu(B_n))\tau_n \xRightarrow[n\to +\infty]{} \Phi + W := (\phi^{(1)} + W, \phi^{(2)} + W, \dots) 
        \end{align*}
        where $\Phi$ and $W$ are as above and independent.
    \end{cor}
    
    \noindent We can now continue with the proof of Theorem \ref{thm:sufficient_conditions_convergence_other_point_processes}. On $\{r_{B_n} > \tau_n\}$, we have $\gamma(\mu(B_n))\Phi_{B_n} = \gamma(\mu(B_n))\Phi_{B_n} \circ T^{\tau_n} +  \gamma(\mu(B_n))\tau_n$, so by \ref{cond_OtherLimits:cluster_compatible_tau_n_no_cluster_from_Q} and \ref{cor:limit_of_the_process_sum_waiting_time_and_shifted}, we get
    \begin{align*}
        \gamma(\mu(B_n))\Phi_{B_n}  \xRightarrow[n \to +\infty]{\mu_{B_n}} \Phi + W. 
    \end{align*}
    However, by hypothesis we also had 
    \begin{align*}
        \gamma(\mu(B_n))\Phi_{B_n} \xRightarrow[n\to +\infty]{\mu_{B_n}} \widetilde{\Phi}.
    \end{align*}
    Thus, we have 
    \begin{align}
        \label{eq:equality_in_law_non_0_waiting_time}
       \widetilde{\Phi} \overset{(law)}{=} \Phi + W  
    \end{align}
    and it remains to show that this equality uniquely determines the law of $\Phi$. We proceed as in Theorem \ref{thm:sufficient_conditions_convergence_compound_FPP} to show first that we get a unique law of $\Phi$ using \ref{thm:HTS-REPP_vs_RTS-REPP_infinite_measure_renormalized_measure}. We proceed by induction. For $d = 1$ and $s> 0$ we have the equation for the Laplace transform
    \begin{align*}
    1 - \frac{s^{\alpha}}{\Gamma(1+\alpha)}\mathbb{E}[e^{-s\phi_1}] = \mathbb{E}[e^{-s\widetilde{\phi}_1}] = \mathbb{E}[e^{-s\phi_1}]\,\mathbb{E}[e^{-sW}]
    \end{align*}
    and thus
    \begin{align*}
        \mathbb{E}[e^{-s\phi_1}] &= \bigg( \mathbb{E}[e^{-sW}] + \frac{s^{\alpha}}{\Gamma(1+\alpha)}\bigg)^{-1}, \quad \forall s > 0
    \end{align*}
    so $\phi_1 \overset{\mathrm{(law)}}{=} J_{\alpha}(\nu)$ and 
    \begin{align*}
        \mathbb{E}[e^{-s\widetilde{\phi}_1}] = \frac{s^{-\alpha}\Gamma(1+\alpha)\,\mathbb{E}[e^{-sW}]}{s^{-\alpha}\Gamma(1+\alpha)\,\mathbb{E}[e^{-sW}] +1},\quad \forall s > 0,
    \end{align*}
    so $\widetilde{\phi}_1 \overset{\mathrm{(law)}}{=} \widetilde{J}_{\alpha}(\nu)$. \\
    
    Now, assume this is true for $d-1\geq 1$. Recall that $\nu$ is the law of $W$. Let $F_1^{[k]}$ and $F_2^{[k]}$ be the distribution functions of $\Phi_1$ and $\Phi_2$ satisfying \eqref{eq:equality_in_law_non_0_waiting_time}. By induction hypothesis, $F_1^{[k]} = F_2^{[k]}$ for $1\leq k \leq d-1$. Let $0 < t_1 \leq \dots \leq t_d$ and set $s_2, \dots, s_d = t_2 - t_1, \dots, t_d - t_1$. By the equivalence between hittings and returns convergence, we have \eqref{eq:relation_hitting_return_infinite_measure} for $i = 1, 2$, that is to say
    \begin{align*}
        F_i^{[d]}(t_1, \dots, t_d) &= \alpha \int_0^{t_1} \bigg( \widetilde{F}_i^{[d-1]}(x+s_2, \dots, x+s_d) \\&\qquad - \widetilde{F}_i^{[d]}(x, x+s_2, \dots, x+s_d)\bigg)(t_1 - x)^{\alpha - 1}\,\dd x\,,
    \end{align*}
    and, by \eqref{eq:equality_in_law_non_0_waiting_time} it yields
    \begin{align*}
        F_i^{[d]}(t_1, \dots, t_d) &= \alpha \int_0^{t_1} \int_0^x \bigg(\bigg( F_i^{[d-1]}(x-y +s_2, \dots, x-y+s_d) \\
        &\qquad - F_i^{[d]}(x-y, x-y+s_2, \dots, x-y+s_d)\bigg)\nu(\dd y) \bigg)(t_1 - x)^{\alpha - 1}\,\dd x\,.
    \end{align*}
    Hence, with
    \begin{align*}
        h : x \mapsto F^{[d]}_1(x, x+s_2, \dots, x+s_d) - F_2^{[d]}(x, x+s_2, \dots, x+s_d), \quad x \geq 0,
    \end{align*}
    we have, for all $t_1 > 0$,
    \begin{align*}
        h(t_1) =- \alpha \int_0^{t_1} \bigg(\int_0^x h(x-y)\nu(\dd y)\bigg)(t_1 - x)^{\alpha - 1}\dd x.
    \end{align*}
    which can be reduced to
    \begin{align*}
        h = - \Gamma(1+\alpha)\, I^{\alpha}( h \star \nu).
    \end{align*}
    Since $h$ is locally integrable and bounded, we can go in the Laplace domain and the only solution is $h = 0$ (see \textit{e.g.} \cite{Book_FractionalCalculus} for more results on this topic) and thus $F_1^{[d]} = F_2^{[d]}$. \\

    Hence there can be at most one process whose distribution functions are fixed points of \eqref{eq:equality_in_law_non_0_waiting_time}. It remains to show the existence of at least one fixed point for this equation. Fortunately, the theory applies again to certain examples of null-recurrent Markov chains (Section \ref{section:Connection_with_Markov_Chains}) for which getting the result for the first return and hitting times is enough to get the convergence of the process due to the strong Markov property. In such scenarios, due to the case $d = 1$, the processes \(\Phi_{\RPP(\widetilde{J}_{\alpha}(\nu))}\) and \(\Phi_{\DRPP(J_{\alpha}(\nu), \widetilde{J}_{\alpha}(\nu))}\) emerge naturally and must satisfy \eqref{eq:equality_in_law_non_0_waiting_time}. This demonstrates that a fixed point indeed exists. Consequently, the fixed point is both unique and well-defined.
\end{proof}

\subsection{Proof of Theorem \ref{thm:HTS_infinite_points_that_never_come_back_statements}}
\label{section:proof_points_that_never_come_back}
Finally, we turn to the proof of Theorem \ref{thm:HTS_infinite_points_that_never_come_back_statements} which studies the third class of points: those that are only finitely recurrent with respect to their initial one cylinder. These are particularly interesting in the infinite measure setting, as they can exhibit behaviors that differ markedly from those observed for generic and periodic points treated in Section \ref{section:FPP_CFPP_infinite_recurrent_points}. However, Theorem \ref{thm:HTS_infinite_points_that_never_come_back_statements} states the result only in term of recurrence for the first letter $x_0$ of the point $x$ considered. However, we enlarge it to work for a wider class of letters $v$. This will be particularly useful when we will have to deal with images or preimages.

\begin{thm}
    \label{thm:HTS_infinite_points_that_never_come_back}
    Let $(\Omega, T, \mu)$ be a topologically mixing Countable Markov Shift associated to a Walters potential $\phi$ that is null recurrent. Assume that the (non-decreasing) normalizing sequence $(a_n)_{n\geq 1}$ belongs to $\RV(\alpha)$ for some $0 < \alpha < 1$. Let $x \in \Omega$ be such that
    \begin{align*}
        K := K_{x_0}(x) := |\mathcal{O}(x) \cap [x_0]| < +\infty
    \end{align*}
    and $v \in V$ be such that 
    \begin{align*}
        K_v := K_v(x) := \big| \mathcal{O}(x) \cap [v] \big| < +\infty.
    \end{align*}
    Set $j_v := j_v(x) := \max \{k \geq 0\;|\; T^kx \in [v]\}$ (with the convention $j_v = 0$ if $K_v = 0$). For the asymptotically rare events $B_n := [x_0^{n-1}]$, we have 
    \begin{align}
        \label{eq:convergence_law_tau_n_in_thm}
        \gamma(\mu(B_n))\,r_{[v]} \circ T^{j_v} \xRightarrow[n \to +\infty]{\mu_{B_n}} W \sim \nu
    \end{align}
    if and only if 
    \begin{align*}
         N_{B_n}^{\gamma} \xRightarrow[n\to +\infty]{\mu_{B_n}} \RPP(\widetilde{J}_\alpha(\nu))\quad \text{and} \quad N_{B_n}^{\gamma} \xRightarrow[n\to +\infty]{\mathcal{L}(\mu)} \DRPP(J_{\alpha}(\nu), \widetilde{J}_\alpha(\nu))
    \end{align*}
    where $J_{\alpha}(\nu)$ and $\widetilde{J}_{\alpha}(\nu)$ are defined in Definition \ref{defn:waiting_random_variables_J_alpha}. 
\end{thm}

\begin{proof}[Proof (of Theorem \ref{thm:HTS_infinite_points_that_never_come_back_statements})]
    This is a direct consequence of Theorem \ref{thm:HTS_infinite_points_that_never_come_back} with $v = x_0$.
\end{proof}

\begin{rem}
    The standard choice is to take $v = x_0$ and thus Theorem \ref{thm:HTS_infinite_points_that_never_come_back_statements} is what we will do in most situations. However, if we want to treat images and preimages of a point, it is sometime easier to take $v$ different from the first symbol of the considered point (see Section \ref{subsubsection:Images_and_preimages}).
\end{rem}

We begin the proof of Theorem \ref{thm:HTS_infinite_points_that_never_come_back} by a Lemma ensuring tightness of the rescaled delay times.

\begin{lem}
    \label{lem:tightness_gamma_tau_n}
    The sequence of measures $\big((\gamma(\mu(B_n))r_{[v]} \circ T^{j_v})_{*} \mu_{B_n} \big)_{n\geq 1}$ is tight.
\end{lem}

\begin{proof}[Proof (of Lemma \ref{lem:tightness_gamma_tau_n})]
    First, assume that $K_v(x) \geq 1$. Then, for all $t>0$, 
    \begin{align*}
        \mu_{B_n}(\gamma(\mu(B_n))\, r_{[v]} \circ T^{j_v} \geq t) &= \mu(B_n)^{-1}\mu(B_n \cap \{\gamma(\mu(B_n))\, r_{[v]} \circ T^{j_v}\geq t\})\\
        &\leq \mu(B_n)^{-1} \mu\big( T^{-j_v}([v] \cap \{\gamma(\mu(B_n))\, r_{[v]} \circ T^{j_v} \geq t\})\big)\\
        & \leq \mu(B_n)^{-1} \mu([v] \cap \{\gamma(\mu(B_n))\, r_{[v]} \circ T^{j_v} \geq t\}) \quad \text{by $T$-invariance of $\mu$.}
    \end{align*}
    Now, we take advantage of the regular variation hypothesis and the definition of $\gamma$ and Lemma \ref{lem:equivalence_queue_renormalisation_gamma} to get 
    \begin{align*}
        \mu_{B_n}(\gamma(\mu(B_n))\tau_n \geq t) \lesssim \frac{\sin(\pi \alpha)}{\pi\alpha}\, t^{-\alpha}.
    \end{align*}
    and hence the tightness for this family of random variables.\\
    Assume now that $K_v(x) = 0$. Then, by the topologically mixing hypothesis, there exists $y \in \Omega$ such that $y \in [v]$, $T^my = x$ and $K_v(y) = 1$, \textit{i.e.} $r_{[v]}(y) = +\infty$. Then, we are able to pass from $\mu_{[y_0^{m+n-1}]}$ to $\mu_{B_n}$. Indeed, for all $t > 0$,
    \begin{align*}
        \mu_{[y_0^{m+n-1}]}( \gamma(\mu(B_n))\, r_{[v]} \circ T^{m} \geq t) &= \mu[y_0^{m+n-1}]^{-1} \int \mathbf{1}_{[y_0^{n+m-1}]} \cdot \mathbf{1}_{\{\gamma(\mu(B_n))\,r_{[v]} \geq t\}} \circ T^m\,\dd\mu \\
        & = \mu[y_0^{m+n-1}]^{-1} \int \widehat{T}^{m}\mathbf{1}_{[y_0^{n+m-1}]} \cdot \mathbf{1}_{\{\gamma(\mu(B_n))\,r_{[v]} \geq t\}}\,\dd\mu,
    \end{align*}
    and we have 
    \begin{align}
        \label{eq:convergence_L_infty_before}
        \big\|\;\mu[y_0^{n+m-1}]^{-1}\,\widehat{T}^m\mathbf{1}_{[y_0^{n+m-1}]} - \mu[x_0^{n-1}]^{-1}\mathbf{1}_{[x_0^{n-1}]} \big\|_{L^{\infty}(\mu_{[x_0^{n-1}]})} \xrightarrow[n\to +\infty]{} 0.
    \end{align}
    as
    \begin{align*}
        \mu[y_0^{n+m-1}]^{-1} \widehat{T}^m\mathbf{1}_{[y_0^{n-1}]}(z) &= \mu[y_0^{n+m-1}]^{-1} L_{\phi_*}^m\mathbf{1}_{[y_0^{n+m-1}]}(z) \\
        & = \mu[y_0^{n+m-1}]^{-1} e^{S_m\phi_*(y_0^{m+n-1}z)} \\
        & = \mu[y_0^{n+m-1}]^{-1} e^{\pm W_{n}(\phi_*)} e^{S_m\phi_*(y_0^{n+m-1}z')} \\
        & = e^{\pm W_{n}(\phi_*)} \mu[y_0^{n+m-1}]^{-1}\,\widehat{T}^m\mathbf{1}_{[y_0^{n+m-1}]}(z')\,.
    \end{align*}
    Yet, we know that $\mu[y_0^{n-1}]^{-1} \widehat{T}^m\mathbf{1}_{[y_0^{n-1}]}$ is a probability density and its support is contained in $[x_0^{n-1}]$. Walter's condition for $\phi_*$ ensures that $W_{n}(\phi_*) \xrightarrow[n\to +\infty]{} 0$ and thus it gives \eqref{eq:convergence_L_infty_before}. Thus, we have for all $t > 0$,
    \begin{align*}
        \mu_{[y_0^{n+m-1}]}\big(\gamma(\mu(B_n))\,r_{[v]} \circ T^{m} \geq t\big) - \mu_{[x_0^{n-m-1}]} \big(\gamma(\mu(B_n))\,r_{[v]} \circ T^{j(x)} \geq t\big)\xrightarrow[n \to +\infty]{} 0.
    \end{align*}
    Finally, 
    \begin{align*}
        \mu_{[y_0^{n+m-1}]}\big(\gamma(\mu(B_n))\,r_{[v]} \circ T^{m} \geq t\big) & \leq \mu[y_0^{n+m-1}]^{-1} \mu([y_0^{n+m-1}] \cap \{r_{[v]} \geq t/\gamma(\mu(B_n))\}) \\
        & \leq \mu[y_0^{n+m-1}]^{-1} \mu([v] \cap \{r_{[v]} \geq t/\gamma(\mu(B_n))\}) \\
        & \leq C t^{-\alpha} \frac{\mu(B_n)}{\mu[y_0^{n+m-1}]} \quad \text{by Lemma \ref{lem:equivalence_queue_renormalisation_gamma}}.
    \end{align*}
    It concludes the proof of the tightness because $\mu(B_n)/\mu[y_0^{n+m-1}]$ is uniformly bounded for all $n \geq 1$.
\end{proof}

\noindent We are now ready to prove Theorem \ref{thm:HTS_infinite_points_that_never_come_back}. 

\begin{proof}[Proof (of Theorem \ref{thm:HTS_infinite_points_that_never_come_back})]
    We define $\tau_n := j_v + r_{[v]} \circ T^{j_v} |_{B_n}$. First, note that $B_n \subset [x_0]$ for all $n\geq 1$ and $[x_0]$ is a uniform set. We are going to prove that $B_n$ satisfies \ref{cond_OtherLimits:good_density_after_tau_n} and \ref{cond_OtherLimits:cluster_compatible_tau_n_no_cluster_from_Q}. $B_n$ is included in the same uniform set $[x_0]$ for all $n\geq 1$ (it is a direct consequence of Theorem \ref{thm:GRPF_null_recurrent} and the definition of $B_n$ as shrinking cylinders). \\

    \noindent For \ref{cond_OtherLimits:good_density_after_tau_n}, let $k \geq 1$ be such that $B_n \cap \{\tau_n = k\} \neq \emptyset$. By the Markov structure of the CMS, there exists a collection $I_k$ of admissible paths $[a_0^{k}]\subset B_n$ with $a_0 = x_0$ and $a_k = v$ such that 
    \begin{align*}
        B_n \cap \{\tau_n = k\} = \bigsqcup_{(a_0^k) \in I_k} [a_0^k].
    \end{align*}
    Now, by Lemma \ref{lem:bounded_distortion_CMS_living_in_compact}, 
    \begin{align*}
        \frac{1}{\mu[a_0^k]}\widehat{T}^k \mathbf{1}_{[a_0^k]} \in \mathcal{U}_{K, M}(v).
    \end{align*}
    for some $K, M$ depending only on $v$. Thus, we get 
    \begin{align*}
        \frac{1}{\mu(B_n \cap \{\tau_n = k\})} \widehat{T}^k \mathbf{1}_{B_n \cap \{\tau_n = k\}} = \sum_{(a_0^k)\in I_k} \frac{\mu[a_0^k]}{\mu(B_n \cap \{\tau_n = k\})} \widehat{T}^k \bigg(\frac{\mathbf{1}_{[a_0^{k}]}}{\mu[a_0^k]}\bigg)
    \end{align*}
    belongs to a compact set of $L^1(\mu)$ as a (potentially countable) convex combination of elements of $\mathcal{U}_{K, M}(v)$.\\
    
    \noindent Finally, we check \ref{cond_OtherLimits:cluster_compatible_tau_n_no_cluster_from_Q}. For this condition, we split the proof between two cases.
    \begin{itemize}
        \item If $|\mathcal{O}(x) \cap [v]| \geq 1$. In particular, $e_{[v]}(x) < +\infty$. When $n > j_v + e_{[v]}(x)$, we have $r_{B_n} \geq \tau_n - e_{[v]}(x)$ on $B_n$ because $r_{B_n} > j_v$ and  $B_n \subset [x_0] \cap T^{-e_{[v]}(x)}[v]$. Moreover, a return to $B_n$ implies a return to $[x_0^{e_{[v]}(x)}]$. Thus, we have 
    \begin{align*}
        \mu_{B_n}(r_{B_n}\leq \tau_n) &\leq \mu_{B_n}(j_v + r_{[v]} \circ T^{j_{v}} - e_{[v]}(x) \leq r_{B_n} \leq r_{[v]} \circ T^{j_v} + j_v) \\
        &\leq \mu_{B_n}(r_{B_n}^{[x_0^{e_{[v]}(x)}]} \leq K_v + e_{[v]}(x))        
    \end{align*}
    where we considered the induced map on $[x_0^{e_{[v]}(x)}]$. Since the induced map on $[x_0^{e_{[v]}(x)}]$ has good properties ($[x_0^{e_{[v]}(x)}]$ is a Darling-Kac set and the induced map is $\psi$-mixing),
    we have $\mu(B_n)r_{B_n}^{[x_0^{e_{[v]}(x)}]} \xRightarrow[n \to +\infty]{\mu_{B_n}} \mathcal{E}$ by Theorem \ref{thm:all-point_REPP_positive_recurrent_CMS}, which ensures that 
    \begin{align*}
        \mu_{B_n}(r_{B_n} \leq \tau_n) \leq \mu_{B_n}\Big(r^{[x_0^{e_{[v]}(x)}]}_{B_n} \leq K_v + e_{[v]}(x)\Big) \xrightarrow[n \to +\infty]{} 0,
    \end{align*}
    showing \ref{cond_OtherLimits:cluster_compatible_tau_n_no_cluster_from_Q}.
    \item If $\mathcal{O}(x) \cap [v] = \emptyset$. Since the TMS is topologically mixing, there exist some $p\geq 1$ and $(a_0^{p})$ admissible such that $a_0 = x_0$ and $a_p = v$. For $n \geq j_{v}(x)$, we have $r_{B_n} \geq n$. In particular, it gives
    \begin{align*}
        \mu_{B_n}(r_{B_n} \leq \tau_n) &= \mu_{B_n}(r_{B_n} \leq r_{[v]})\\
        & \leq \mu_{B_n}(r_{B_n} \leq r_{[a_0^p]})\\
        &\leq \mu_{B_n}(r_{B_n}^{[x_0]} \leq r_{[a_0^p]}^{[x_0]})\\
        & \leq \mu_{B_n}(r_{B_n}^{[x_0]} = K) + \mu_{B_n}\big(T_{[x_0]}^{-K} \{r_{B_n}^{[x_0]} \leq r_{[a_0^p]}^{[x_0]}\}\big) \\
        & \leq \mu_{B_n}(r_{B_n}^{[x_0]} = K) + C\mu_{[x_0]}\big(r_{B_n}^{[x_0]} \leq r_{[a_0^p]}^{[x_0]}\big)\,,
    \end{align*}
    where for the last inequality, we use the bounded distortion estimate for the induced map. Finally, for the induced map, by Theorem \ref{thm:all-point_REPP_positive_recurrent_CMS} we know that $\mu(B_n)\,r_{B_n}^{[x_0]} \xRightarrow[n \to +\infty]{\mu_{B_n}\; \& \;\mu_{[x_0]}} \mathcal{E}$ which ensures that both terms term on the right part of the inequality converge to $0$ and hence it shows \ref{cond_OtherLimits:cluster_compatible_tau_n_no_cluster_from_Q}.
    \end{itemize}

    \noindent Thus, to apply Theorem \ref{thm:sufficient_conditions_convergence_other_point_processes} the only condition remaining is \ref{cond_OtherLimits:tau_n_small_enough}, that is to say
    \begin{align*}
        \gamma(\mu(B_n))\,\tau_n \xRightarrow[n \to +\infty]{\mu_{B_n}} W.
    \end{align*}
    But, by definition of $\tau_n$, we have
    \begin{align*}
        \gamma(\mu(B_n))\,\tau_n - \gamma(\mu(B_n))\, r_{[v]}\circ T^{j_v} = j_v\gamma(\mu(B_n)) \xrightarrow[n \to +\infty]{} 0. 
    \end{align*}
    Hence \ref{cond_OtherLimits:tau_n_small_enough} is equivalent to \eqref{eq:convergence_law_tau_n_in_thm}. So, if \eqref{eq:convergence_law_tau_n_in_thm} is verified, so is \ref{cond_OtherLimits:tau_n_small_enough} and thus we have the left implication. On the other side, Lemma \ref{lem:tightness_gamma_tau_n} shows the tightness of $\gamma(\mu(B_n))r_{[v]}\circ T^{j_v}$ and, since the law of $J_{\alpha}(\nu)$ uniquely defines the law of $W$, it ensures that the only possible limit is $W$ and thus \eqref{eq:convergence_law_tau_n_in_thm} is satisfied.
\end{proof}

\subsection{The special case of images and preimages}
\label{subsubsection:Images_and_preimages}

\noindent However, if one point has a non zero limit for the delay time, then the behavior of the REPP for its preimages is much more constrained. We start by defining our density transitions to pass from a point $x$ to a point $y$. For all $x, y \in \Omega$ such that $T^m y = x$, we set 
\begin{align}
    \label{eq:definition_Q_x_m_y}
    Q_{x, m}(y) := e^{S_m\phi_*(y)} = h(x)^{-1}\, h(y)e^{S_m\phi - mP_G(\phi)}.
\end{align}

\begin{rem}
    \label{rem:Q_x,m_is_a_probability_density}
We choose such a notation $Q_{x,m}(y)$ because $Q_{x,m}$ can be interpreted as a probability distribution on $T^{-m}\{x\}$. Indeed, we have 
\begin{align*}
    \sum_{y \in T^{-m}x} Q_{x, m}(y) = \sum_{T^my = x} e^{S_m\phi_*(y)} = L_{\phi_*}1 = 1, \quad \text{by Remark \ref{rem:definition_phi_*}.} 
\end{align*}
\end{rem}

\begin{lem}
    \label{lem:preimages_CMS_delay_time}
    Let $x \in \Omega$ be such that $K(x) <+\infty$. Assume that 
    \begin{align*}
        \gamma(\mu[x_0^{n-1}])\,r_{[x_0]} \circ T^{j_{x_0}(x)} \xRightarrow[n \to +\infty]{\mu_{[x_0^{n-1}]}} W.
    \end{align*}
    Then, for all $m \geq 0$,
    \begin{itemize}
        \item If $y \in T^{-m}\{x\}$, 
        \begin{align*}
            \gamma(\mu([y_0^{n-1}]))\,r_{[x_0]}\circ T^{j_{x_0}(y)} \xRightarrow[n \to +\infty]{\mu_{[y_0^{n-1}]}} Q_{x,m}(y)^{1/\alpha}\,W.
        \end{align*}
        \item If $y = T^m x$,
        \begin{align*}
            \gamma(\mu([y_0^{n-1}]))\,r_{[x_0]}\circ T^{j_{x_0}(y)} \xRightarrow[n \to +\infty]{\mu_{[y_0^{n-1}]}} Q_{y,m}(x)^{-1/\alpha}\,W.
    \end{align*}
    \end{itemize}
\end{lem}

\begin{proof}[Proof (of Lemma \ref{lem:preimages_CMS_delay_time})]
    Let $y$ be such that $T^my = x$. We have $j_{x_0}(y) = j_{x_0}(x) + m$. By bounded distortion estimates, we are able to pass from $\mu_{B_n}$ to $\mu_{[x_0^{n-1}]}$. Indeed, for all $t> 0$, 
    \begin{align}
         \nonumber
         \MoveEqLeft \mu_{[y_0^{n-1}]}(\gamma(\mu[y_0^{n-1}])\,r_{[x_0]} \circ T^{j_{x_0}(x) +m} > t) \\
        &= \mu[y_0^{n-1}]^{-1} \int \mathbf{1}_{[y_0^{n-1}]} \cdot \mathbf{1}_{\{r_{[x_0]} \circ T^{j(x)} > t/\gamma(\mu[y_0^{n-1}])\}} \circ T^m\,\dd\mu \nonumber\\
        & = \mu[y_0^{n-1}]^{-1} \int \widehat{T}^m\mathbf{1}_{[y_0^{n-1}]}\cdot \mathbf{1}_{\{r_{[x_0]} \circ T^{j(x)} > t/\gamma(\mu[y_0^{n-1}])\}} \,\dd\mu. \label{eq:computation1_preimages}
    \end{align}
    \noindent Moreover, we have 
    \begin{align}
        \label{eq:convergence_L_infty_preimages}
        \Big\|\;\mu[y_0^{n-1}]^{-1}\,\widehat{T}^m\mathbf{1}_{[y_0^{n-1}]} - \mu[x_0^{n-m-1}]^{-1}\mathbf{1}_{[x_0^{n-m-1}]} \Big\|_{L^{\infty}(\mu_{[x_0^{n-m-1}]})} \xrightarrow[n\to +\infty]{} 0.
    \end{align}    
    The argument is the same as the one we used to prove \ref{cond_CFPP:Compatibility_Geometric_law} for periodic points in Theorem \ref{thm:all-point_REPP_positive_recurrent_CMS} and the argument for Lemma \ref{lem:tightness_gamma_tau_n}. Indeed, for $n > m$ and for all $z, z' \in [x_0^{n-m-1}]$, we have 
    \begin{align*}
        \mu[y_0^{n-1}]^{-1} \widehat{T}^m\mathbf{1}_{[y_0^{n-1}]}(z) &= \mu[y_0^{n-1}]^{-1} L_{\phi_*}^m\mathbf{1}_{[y_0^{n-1}]}(z) \\
        & = \mu[y_0^{n-1}]^{-1} e^{S_m\phi_*(y_0^{m-1}z)} \\
        & = \mu[y_0^{n-1}]^{-1} e^{\pm W_{n - m}(\phi_*)} e^{S_m\phi_*(y_0^{m-1}z')} \\
        & = e^{\pm W_{n - m}(\phi_*)} \mu[y_0^{n-1}]^{-1}\,\widehat{T}^m\mathbf{1}_{[y_0^{n-1}]}(z')\,.
    \end{align*}
    Yet, we know that $\mu[y_0^{n-1}]^{-1} \widehat{T}^m\mathbf{1}_{[y_0^{n-1}]}$ is a probability density and its support is contained in $[x_0^{n-m-1}]$. Walter's condition for $\phi_*$ ensures that $W_{n - m}(\phi_*) \xrightarrow[n\to +\infty]{} 0$ and thus it gives \eqref{eq:convergence_L_infty_preimages}. Thus, going back to \eqref{eq:computation1_preimages}, we have for all $t > 0$,
    \begin{align}
        \label{eq:from_y_to_x_preimages}
        \mu_{[y_0^{n-1}]}\big(\gamma(\mu[y_0^{n-1}])\,r_{[x_0]} \circ T^{j(x) +m} > t\big) - \mu_{[x_0^{n-m-1}]} \big(\gamma(\mu[y_0^{n-1}])\,r_{[x_0]} \circ T^{j(x)} > t\big)\xrightarrow[n \to +\infty]{} 0.
    \end{align}
    Furthermore, we need to take care about the $\gamma$ renormalization and compute the limit of the ratio
    \begin{align*}
        \gamma(\mu[y_0^{n-1}])/ \gamma(\mu[x_0^{n-m-1}]) &= \gamma(\mu[y_0^{n-1}]) / \gamma(\mu[y_m^{n-1}]).
    \end{align*}
    For all $z \in [y_0^{n-1}]$, we have
    \begin{align*}
        \widehat{T}^m \mathbf{1}_{[y_0^{n-1}]}(z) &= L_{\phi_*}^m \mathbf{1}_{[y_0^{n-1}]}(z) = e^{S_m\phi_*(y_0^{m-1}z)}\\
        & = e^{\pm \var_{n} S_m\phi_*}e^{S_m\phi_*(y)} \quad \text{because $y \in [y_0^{n-1}]$}\\
        & = e^{\pm W_{n-m}(\phi_*)}e^{S_m\phi_*(y)}.
    \end{align*}
    Hence, we obtain 
    \begin{align*}
        \mu[y_0^{n-1}] & = \int \widehat{T}^{m}\mathbf{1}_{[y_0^{n-1}]} \,\dd\mu \\
        & = e^{\pm W_{n-m}(\phi_*)} \int e^{S_m\phi_*(y)} \mathbf{1}_{[y_m^{n-1}]}(z)\,\dd\mu(z) \\
        & = e^{\pm W_{n-m}(\phi_*)} e^{S_m\phi_*(y)}\mu[y_m^{n-1}].
    \end{align*}
    Thus, since $W_{n-m}(\phi_*) \xrightarrow[n \to +\infty]{} 0$, we have 
    \begin{align}
        \label{eq:computation_measure_cylindre_images}
        \lim_{n \to +\infty} \frac{\mu[y_0^{n-1}]}{\mu[x_0^{n-m-1}]} = e^{S_m\phi_*(y)} = Q_{x,m}(y).
    \end{align}
    
    \noindent Finally, by the regular variation hypothesis on the system, we have $\gamma \in RV(1/\alpha)$ and thus, 
    \begin{align*}
        \frac{\gamma(\mu[y_0^{n-1}])}{\gamma(\mu[x_0^{n-m-1}])} \xrightarrow[n \to +\infty]{} Q_{x,m}(y)^{1/\alpha}.
    \end{align*}
    Using the convergence of the delay time for the point $x$ and \eqref{eq:from_y_to_x_preimages}, it yields 
    \begin{align*}
        \gamma(\mu[y_0^{n-1}])\,r_{[x_0]} \circ T^{j(x) + m} \xRightarrow[n \to +\infty]{\mu_{[y_0^{n-1}]}} Q_{x,m}(y)^{1/\alpha} W.
    \end{align*}
    \noindent The case $y = T^m x$ can be proven similarly.
\end{proof}

\noindent From Lemma \ref{lem:preimages_CMS_delay_time}, we can deduce the following Proposition allowing us to quantify the limits of REPP associated to preimages of a certain point for which we are able to characterize the limit.

\begin{prop}
    \label{prop:convergence_REPP_preimages}
    Let $x \in \Omega$ be a point such that $K(x) < +\infty$. Assume that for some probability measure $\nu$ on $\mathbb{R}_+$,
    \begin{align*}
        \gamma(\mu(B_n))\,r_{[x_0]} \circ T^{j(x)} \xRightarrow[n \to +\infty]{\mu_{[x_0^{n-1}]}} W \sim \nu.
    \end{align*}
    \noindent For all $m \geq 0$ and $y \in T^{-m}\{x\}$ denote $\nu_y$ the probability distribution of $Q_{x,m}(y)^{1/\alpha}W$. Then, for $B_n := [y_0^{n-1}]$,
    \begin{align*}
        N_{B_n}^{\gamma} \xRightarrow[n \to +\infty]{\mu_{B_n}} \RPP(\widetilde{J}_{\alpha}(\nu_y)) \quad \text{and} \quad N_{B_n}^{\gamma} \xRightarrow[n \to +\infty]{\mathcal{L}(\mu)} \DRPP(J_{\alpha}(\nu_y), \widetilde{J}_{\alpha}(\nu_y)).
    \end{align*}
\end{prop}

\begin{proof}[Proof (of Proposition \ref{prop:convergence_REPP_preimages})]
    If $K(y) = |\mathcal{O}(y_0) \cap [y_0]| = +\infty$ then we also have $K_{y_0}(x) = |\mathcal{O}(x) \cap [y_0]| = +\infty$ and $W$ is the null random variable by application of Theorem \ref{thm:HTS_infinite_FPP_infinitely_recurrent} and \ref{thm:HTS_infinite_points_that_never_come_back} together for $x$. Theorem \ref{thm:HTS_infinite_FPP_infinitely_recurrent} for $y$ also ensures the equivalence. So we can assume that $K(y) < +\infty$. Since $K(x) <+\infty$, we also have $K_{x_0}(y) < +\infty$. Thus, Theorem \ref{thm:HTS_infinite_points_that_never_come_back} can be applied to $y$ with the delay time law being $Q_{x,m}(y)^{\alpha^{-1}} W$ by Lemma \ref{lem:preimages_CMS_delay_time}.
\end{proof}

\noindent The converse is also true, that it to say we can pass from the behavior of a point to the behavior of one of its images.

\begin{prop}
    \label{prop:convergence_REPP_images}
    Let $x \in \Omega$ be a point such that $K(x) < +\infty$. Assume that 
    \begin{align*}
        \gamma(\mu(B_n))\, r_{[x_0]} \circ T^{j(x)} \xRightarrow[n \to +\infty]{\mu_{[x_0^{n-1}]}} W.
    \end{align*}
    
    For all $m \geq 0$, $y = T^mx$, denote $\nu_y$ the probability distribution of $Q_{y,m}(x)^{-1/\alpha}W$. Then, for the sequence of asymptotically rare events $B_n := [y_0^{n-1}]$, we have 
    \begin{align*}
        N_{B_n}^{\gamma} \xRightarrow[n \to +\infty]{\mu_{B_n}} \RPP(\widetilde{J}_{\alpha}(\nu_y)) \quad \text{and} \quad N_{B_n}^{\gamma} \xRightarrow[n \to +\infty]{\mathcal{L}(\mu)} \DRPP(J_{\alpha}(\nu_y), \widetilde{J}_{\alpha}(\nu_y))
    \end{align*}
    where $B_n := [y_0^{n-1}]$.
\end{prop}

\begin{proof}[Proof (of Proposition \ref{prop:convergence_REPP_images})]
    If $K(y) = +\infty$, then $K_{y_0}(x) = +\infty$ and $W$ is necessarily the null random variable. In both case, we can apply Theorem \ref{thm:HTS_infinite_FPP_infinitely_recurrent} giving the equivalence. So we can assume that $K(y) < +\infty$.  In this case, we have $K_{x_0}(y) < +\infty$ (because $Y = T^mx$ and $K(x) < +\infty$). We just need to find the limit behavior of $\gamma(\mu(B_n))\,r_{[x_0]} \circ T^{j_{x_0}(y)}$ under $\mu_{B_n}$. This is handled by Lemma \ref{lem:preimages_CMS_delay_time}.
\end{proof}

Lemma \ref{lem:preimages_CMS_delay_time}, Proposition \ref{prop:convergence_REPP_preimages} and Proposition \ref{prop:convergence_REPP_images} together with the next Section characterizing the possible delay times will be a key tool to prove another sufficient condition for a point $x$ to have the convergence of its REPP towards a fractional Poisson process.

\section{Proofs for results of section \ref{section:possible_limit_laws}}
\label{section:proof_possible_limit_laws}
This section is dedicated to classify all the achievable random variables and explain the explicit examples where they actually appear. Thus we will here prove Propositions \ref{prop:only_possible_limits_for_waiting_times_in_G} and \ref{prop:Every_waiting_time_achievable_when_CMS_well_chosen}.

Recall that we defined $\mathcal{G}_{\alpha}$ by the following set of distribution (see Figure \ref{fig:enter-label} for a graphical representation).
\begin{align*}
    \mathcal{G}_{\alpha} := \bigg\{ \nu \in &\,\mathbb{P}(\mathbb{R}_+) \;|\; \forall s>t>0, \; \nu(]t, +\infty[) \leq 1 \wedge \frac{\sin(\pi \alpha)}{\pi\alpha}\, t^{-\alpha}\; \\
    &\text{and}\; \nu(]t, s]) \leq \frac{\sin(\pi\alpha)}{\pi\alpha}\big(t^{-\alpha} - s^{-\alpha}\big)\bigg\}.
\end{align*}

This is a property that can be easily define through the tail distribution function $\overline{F_{\nu}} : s \mapsto 1 - F_{\nu}(s) := 1 - \nu([0, s])$. To alleviate notation, we will sometimes identify the probability $\nu$ and its tail distribution function and write $\overline{F_{\nu}} \in \mathcal{G}_{\alpha}$.\\

We start with the proof of Proposition \ref{prop:only_possible_limits_for_waiting_times_in_G}. This is a direct consequence of the regular variation hypothesis.

\begin{proof}[Proof (of Proposition \ref{prop:only_possible_limits_for_waiting_times_in_G})]
    Lemma \ref{lem:equivalence_queue_renormalisation_gamma} ensures that, for all $0 < t < s \leq +\infty$,
    \begin{align*}
        \mu_{B_n} \big( t < \gamma(\mu(B_n))\, r_{[x_0]} \circ T^j \leq s\big) &= \mu(B_n)^{-1} \mu\big( B_n \cap \{t < \gamma(\mu(B_n))\, r_{[x_0]} \circ T^j \leq s\}\big) \\
        & \leq \mu(B_n)^{-1}\mu\big([x_0] \cap T_{[x_0]}^{-(j-1)} \{t < \gamma(\mu(B_n))\, r_{[x_0]} \leq s\}\big)\\
        & \leq \mu(B_n)^{-1} \mu\big([x_0] \cap \{t < \gamma(\mu(B_n))\, r_{[x_0]} \leq s\}\big), \; \text{by $T_{[x_0]}$-invariance}\\
        & \xrightarrow[n\to +\infty]{} \frac{\sin(\pi \alpha)}{\pi\alpha} (t^{-\alpha} - s^{-\alpha}).
    \end{align*}
    Since $\mu_{B_n}$ is a probability measure, it gives the upper bound
    \begin{align*}
        \mu_{B_n}\big(\gamma(\mu(B_n))\, r_{[x_0]} \circ T^j > t\big) \lesssim 1 \wedge \frac{\sin(\pi \alpha)}{\pi\alpha}t^{-\alpha}.
    \end{align*}
    This is enough to ensure that, if the limit law exists, then we have $\nu \in \mathcal{G}_{\alpha}$.
\end{proof}

\noindent Now, we turn to Proposition \ref{prop:Every_waiting_time_achievable_when_CMS_well_chosen}. Our construction is based on a family of null-recurrent Markov Chains (recall that Markov chains are within the CMS setting, see Section \ref{section:Connection_with_Markov_Chains}) that allow us to keep the regular variation hypothesis while enabling us to build a large variety of examples. Thus, we start by constructing the frame for our future examples. It is a combination of a House of Cards and a renewal-type structure.

\begin{lem}
    \label{lem:construction_HoC/renewal_given_probabilities}
    Let $(p_k)_{k\geq 1}$ be a sequence of non negative numbers such that $\sum_{k \geq 1} p_k = 1$. Let $(c_{k,n})_{k\geq 1, n \geq 1} \in [0,1]^{\mathbb{N}^2}$ be such that for all $k \geq 1$, $(c_{k,n})_{n \geq 1}$ is a non increasing sequence, for all $k \geq 1$, $c_{k,1} = 1$ and for all $n > k$, $c_{k,n} = 0$. Then, there exists a Markov Chain on the countable phase space $V := \mathbb{N} \cup i\mathbb{N}$ (see Figure \ref{fig:HoC-Renewal_structure} for the graph) and a Markov measure $\mu$ such that
    \begin{align*}
        \mu( B_n \cap \{r_{[0]} = k\}) = c_{k,n}\, p_k \quad \forall k\geq 1
    \end{align*}
    with $B_n = [0 i(2i)\cdots ((n-1)i)]$.
\end{lem}

\begin{rem}
    In particular, we have 
    \begin{align*}
        \mu[0] = 1 \quad \text{and} \quad \mu([0] \cap \{r_{[0]} = k\}) = p_k \quad \forall k \geq 1,
    \end{align*}
    and
    \begin{align*}
        \mu(B_n) = \sum_{k \geq 0} c_{k,n}\, p_k = \sum_{k \geq n} c_{k,n}\, p_k \quad \text{because}\; c_{k,n} = 0 \; \text{for} \; k < n.
    \end{align*}
\end{rem}

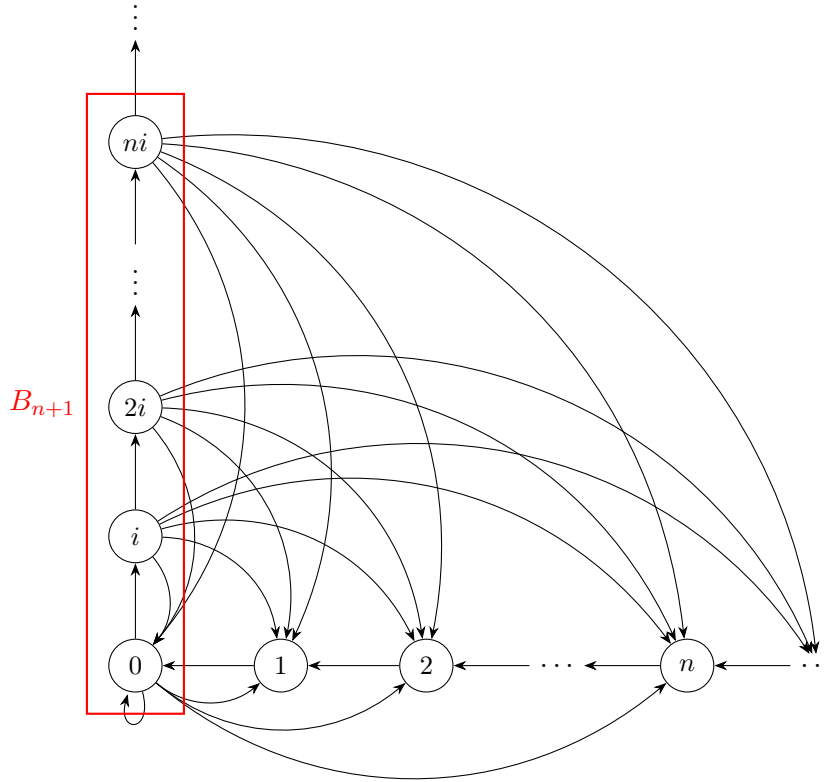
\begin{figure}[h]
\centering
\begin{tikzpicture}[
  state/.style={circle, draw, minimum size=7mm, inner sep=1pt, font=\small},
  >=Stealth,
  node distance=10mm and 12mm
]

\node[state] (H0) {0};
\node[state, right=of H0] (H1) {1};
\node[state, right=of H1] (H2) {2};
\node[right=1cm of H2] (Hdots) {\(\ldots\)};
\node[state, right=1cm of Hdots] (Hn) {\(n\)};
\node[right=1cm of Hn] (Hdots2) {\(\ldots\)};

\node[state, above=of H0] (V1) {\(i\)};
\node[state, above=of V1] (V2) {\(2i\)};
\node[above=1cm of V2] (Vdots) {\(\vdots\)};
\node[state, above=1cm of Vdots] (Vn) {\(ni\)};
\node[above=1cm of Vn] (Vdots2) {\(\vdots\)};

\foreach \x/\y in {H0/H1, H1/H2, H2/Hdots, Hdots/Hn, Hn/Hdots2} {
  \draw[->] (\y) -- (\x);
}

\draw[->, bend left=40] (V1) to (H0);
\draw[->, bend left=40] (V1) to (H1);
\draw[->, bend left=40] (V1) to (H2);
\draw[->, bend left=40] (V1) to (Hn);
\draw[->, bend left=45] (V1) to (Hdots2);

\draw[->, bend left=40] (V2) to (H0);
\draw[->, bend left=40] (V2) to (H1);
\draw[->, bend left=40] (V2) to (H2);
\draw[->, bend left=40] (V2) to (Hn);
\draw[->, bend left=45] (V2) to (Hdots2);

\draw[->, bend left=40] (Vn) to (H0);
\draw[->, bend left=40] (Vn) to (H1);
\draw[->, bend left=40] (Vn) to (H2);
\draw[->, bend left=40] (Vn) to (Hn);
\draw[->, bend left=45] (Vn) to (Hdots2);

\draw[->] (H0) edge[loop below]  (H0);
\draw[->, bend right=40] (H0) to (H1);
\draw[->, bend right=40] (H0) to (H2);
\draw[->, bend right=40] (H0) to (Hn);

\foreach \x/\y in {H0/V1, V1/V2, V2/Vdots, Vdots/Vn, Vn/Vdots2} {
  \draw[->] (\x) -- (\y);
}

\node[draw=red, thick, fit=(H0)(V1)(V2)(Vn), inner sep=8pt, label=left:{\textcolor{red}{$B_{n+1}$}}] {};


\end{tikzpicture}
\caption{The House of Cards-Renewal structure. The targets are cylinders shrinking towards the point $x = (ki)_{k\geq 0}$.}
\label{fig:HoC-Renewal_structure}
\end{figure}

\begin{proof}[Proof (of Lemma \ref{lem:construction_HoC/renewal_given_probabilities})]
    The idea is to use a House of cards construction for climbing on the imaginary axis and a renewal when we arrive on the real axis. Hence, we consider the set $E$ of arrows is the following 
    \begin{align*}
        E := \{(ni, (n+1)i), \; n\geq 0\} \cup \{(ni, \ell), \; n,\ell\geq 0\} \cup \{(k+1, k), \; k\geq 0\}.
    \end{align*}
    We endow it with the transition kernel $P$ defined by
    \begin{align*}
        P(k+1, k) &= 1\\
        P(ni, k) &= \bigg(\sum_{j\geq 0} c_{j, n+1}\, p_j \bigg)^{-1}(c_{n+k+1, n+1} - c_{n+k+1, n+2})\, p_{n+k+1} \\
        P( (n-1)i, ni) &= \bigg(\sum_{j\geq 0} c_{j, n+1}\, p_j \bigg) \bigg( \sum_{j \geq 0} c_{j,n}\, p_j\bigg)^{-1}.
    \end{align*}
    This Markov chain verifies for all $k\geq 0$
    \begin{align*}
        & \mu(B_n \cap \{r_{[0]} = n +k\}) = \sum_{m = 0}^k \mu[0i\cdots((n-1)i)\cdots ((n + m - 1)i)(k-m)(k-m-1)\cdots 10]\\
        & = \sum_{m = 0}^k \prod_{\ell = 1}^{n+m-1} P((\ell-1) i, \ell i) \cdot P((n+m-1)i, k-m) \\
        & = \sum_{m = 0}^k \prod_{\ell = 1}^{n+m-1} \frac{\sum_{j\geq 0} c_{j, \ell+1}\, p_j}{\sum_{j\geq 0} c_{j,\ell}\, p_j} \cdot \bigg(\sum_{j\geq 0} c_{j, n+m} p_j \bigg)^{-1}(c_{n + k, n + m} - c_{n + k, n+m+1})\, p_{n+k} \\
        & = \sum_{m = 0}^k (c_{n + k, n + m} - c_{n + k, n+m+1})\, p_{n+k}\\
        & = \big(c_{n + k, n} - c_{n+k, n+k+1}\big)\, p_{n+k}\\
        & = c_{n+k}\, p_{n+k} \quad \text{because} \; c_{n+k, n+k+1} = 0.
    \end{align*}
    Since $c_{k, n} = 0$ for $k < n$ by hypothesis, we have
    \begin{align*}
        \mu(B_n \cap \{r_{[0]} = k\}) = c_{k,n}\, p_k \quad \forall k \geq 1
    \end{align*}
\end{proof}

\noindent We can now proceed to the proof of Proposition \ref{prop:Every_waiting_time_achievable_when_CMS_well_chosen}.

\begin{proof}[Proof (of Proposition \ref{prop:Every_waiting_time_achievable_when_CMS_well_chosen})]
    Let $W$ be a random variable such that $1 - F_W \in \mathcal{G}_{\alpha}$. By definition of $\mathcal{G}_{\alpha}$, the only possible atom of $W$ is $0$. Let $\rho := \mathbb{P}(W = 0)$.
    Consider a decreasing sequence $(p_k)_{k \geq 1}$ of regular variation of parameter $-1 -\alpha$ and such that
    \begin{align*}
        \sum_{k \geq 1} p_k = 1
    \end{align*} 
    \noindent We are going to build a Markov chain on a countable state space \(V := \mathbb{N} \cup i\mathbb{N}\) such that 
    \begin{align*}
        \mu([0] \cap \{r_{[0]} = k\}) = p_k \quad \forall k \geq 1. 
    \end{align*}
    We also denote $q_k := \sum_{j \geq k} p_j$, that is to say we are going to have
    \begin{align*}
        \mu([0] \cap \{r_{[0]} \geq k\}) = q_k \quad \forall k \geq 1
    \end{align*}
    and $(q_k)_{k \geq 1} \in RV(-\alpha)$. We can take a continuation $q : \mathbb{R}_+ \to \mathbb{R}_+ \in RV(-\alpha)$.\\

    \noindent Let $(\varepsilon_n)_{n \geq 1}$ be a sequence converging to $0$ that will be designed afterwards. Set $(s_n)_{n \geq 1}$, $(M_n := n+1)_{n\geq 1}$, $(m_n := 2^{-n})_{n\geq 1}$ and $(u_n)_{n\geq 1}$ such that, for all $n \geq 1$,
    \begin{align}
        \label{eq:hyp_on_u_n_m_n_s_n}
        u_n \geq n, \;
        q_{m_ns_n} \leq \Big(\frac{1 - \rho}{\rho} \wedge 1\Big) p_{u_n} \; \text{and} \; 1 \geq q_{u_{n+1}} / q_{M_ns_n} \xrightarrow[n \to +\infty]{} 0.
    \end{align}
    In particular, it implies $u_{n} \leq m_ns_n \leq M_ns_n \leq u_{n+1}$.
    The regular variation implies that 
    \begin{align*}
        \frac{q(tn)}{q(n)} \xrightarrow[n \to +\infty]{} t^{-\alpha}
    \end{align*}
    and the convergence is uniform on every set of the form $[a, +\infty)$ where $a > 0$ (see \cite[Theorem 1.5.2]{Bingham89_RegularVariation}), so, up to choosing $s_n$ even larger, we can also build it such that $t\geq m_n = 2^{-n}$, 
    \begin{align}
        \label{eq:uniform_regular_variation_queue}
        t^{-\alpha} - \varepsilon_n \leq q(ts_n)/q(s_n) \leq t^{-\alpha} + \varepsilon_n.
    \end{align}
    Note that the dependencies are the following : $s_n = s_n(u_n, m_n, \varepsilon_n)$ and $u_{n+1} = u_{n+1}(s_n,M_n)$ but there is no dependencies needed between $\varepsilon_n$ and $m_n$ or $M_n$.\\
    
    Let $h : t \mapsto \frac{\sin(\pi \alpha)}{\pi\alpha}t^{-\alpha}$. For all $k \in E_n := \{-2^n + 1, \dots, n2^n - 1\}$, define $0 \leq \eta_{k,n}$ such that 
    \begin{align}
        \label{eq:definition_eta_k,n}
        F_W(1 + k2^{-n}) - F_W(1+(k+1)2^{-n}) = \eta_{k,n}(h(1 + k2^{-n}) - h(1 + (k+1)2^{-n})). 
    \end{align}
    Because $1 - F_W \in \mathcal{G}_{\alpha}$, we have $\eta_{k,n} \leq 1$. 
    For all $p \geq 0$ and $k \in E_p$, set 
    \begin{align*}
        A_{p,k} := [(1 + k2^{-p})s_p, (1 + (k+1)2^{-p})s_p[ \,\cap\, \mathbb{N}.
    \end{align*}
    With such a definition, we have 
    \begin{align*}
        [m_ps_p, M_ps_p] \cap \mathbb{N} = \bigsqcup_{k\in E_p} A_{p,k}. 
    \end{align*}
    For all $n \geq 0$, we set 
    \begin{align*}
        c_{m, n} := \eta_{k,p} \; \text{if} \; m\in A_{p,k} \; \text{for some} \; p\geq n, \, k \in E_p.
    \end{align*}
    Then, for $p \geq n$ 
    \begin{align}
        \label{eq:definition_c_u_n_n}
        c_{u_p, n} := \frac{\rho}{p_{u_p}(1 - \rho)}\sum_{k \in E_p}\sum_{m \in A_{p,k}} \eta_{k,p}p_m.
    \end{align}
    Finally, we set $c_{m,n} := 0$ for all $m \in \Big(\bigcup_{p\geq n} \{u_p\} \cup ([m_ps_p, M_ps_p] \cap \mathbb{N})\Big)^c$.

    \noindent The defined sequence $(c_{k,n})_{k,n\geq 1}$ verifies the hypothesis of Lemma \ref{lem:construction_HoC/renewal_given_probabilities}. Indeed, for all $k, p \geq 0$, $\eta_{k,p} \leq 1$ and for all $p \geq n \geq 1$, 
    \begin{align*}
        c_{u_p, n} \leq \frac{\rho}{(1 - \rho)p_{u_p}} \sum_{k\geq m_ps_p} p_k \leq \frac{\rho}{(1 - \rho)p_{u_p}} q_{m_ps_p} \leq \frac{\rho}{1 - \rho}\frac{1 - \rho}{\rho} = 1 \; \text{by} \; \eqref{eq:hyp_on_u_n_m_n_s_n}.
    \end{align*}
    Furthermore, by construction $c_{m,n} = 0$ for $m < u_n$ and since $u_n \geq n$ by \eqref{eq:hyp_on_u_n_m_n_s_n}, $c_{m,n} = 0$ for all $m < n$. Finally, the construction imposes $(c_{k,n})_{n\geq 0}$ for all $k\geq 1$ as for each $k$, $c_{k,n}$ can take only two values and once it is $0$ it remains $0$.\\
    
    Thus, by Lemma \ref{lem:construction_HoC/renewal_given_probabilities} we can construct a Markov Chain associated to $(p_k)_{k \geq 1}$ and $(c_{k,n})_{n\geq 0}$. We now show that for this special Markov chain and the point $x = (ki)_{n \geq 0}$ we have
    \begin{align}
        \label{eq:convergence_waiting_time_in_construction_everything_possible}
        \gamma(\mu(B_n))r_{[0]} \xRightarrow[n \to +\infty]{\mu_{B_n}} W
    \end{align}
    where $B_n = [x_0^{n-1}] = [0i \cdots (n-1)i]$. We start by computing the measure of $B_n$.
    
    \begin{lem}
    \label{lem:connection_between_s_n_and_mu_B_n}
    We have the following asymptotic equivalences
    \begin{align*}
        \mu(B_n) \isEquivTo{n \to +\infty} \frac{\pi\alpha}{\sin(\pi\alpha)}q_{s_n} \quad \text{and} \quad
        \gamma(\mu(B_n))^{-1} \isEquivTo{n \to +\infty} s_n.
    \end{align*}
\end{lem}

\begin{proof}[Proof (of Lemma \ref{lem:connection_between_s_n_and_mu_B_n})]
   We have 
\begin{align*}
    \mu(B_n) & = \sum_{k \geq 0} c_{k,n}p_k = c_{u_n,n}p_{u_n} + \sum_{m_ns_n \leq k \leq M_ns_n} c_{k,n}p_k + \sum_{j > n}\Big( c_{u_j,n}p_{u_j} + \sum_{m_js_j \leq k \leq M_js_j} c_{k,n}p_k\Big) \\
    & = c_{u_n,n}p_{u_n} + \sum_{m_ns_n \leq k \leq M_ns_n} c_{k,n}p_k + O\Big( \sum_{k \geq u_{n+1}} p_k\Big)\\
    & = c_{u_n,n}p_{u_n} + \sum_{m_ns_n \leq k \leq M_ns_n} c_{k,n}p_k + O(q_{u_{n+1}})\\
    & = c_{u_n,n}p_{u_n} + \sum_{m_ns_n \leq k \leq M_ns_n} c_{k,n}p_k + o(q_{M_ns_n}) \; \text{by} \;\eqref{eq:hyp_on_u_n_m_n_s_n} \\
    & = c_{u_n,n}p_{u_n} + \sum_{k \in E_n} \sum_{m \in A_{n,k}} \eta_{k,n} p_m  + o(q_{M_ns_n})\\
    & = c_{u_n,n}p_{u_n} + \sum_{k \in E_n} \eta_{k,n} \sum_{m \in A_{n,k}} p_m + o(q_{M_ns_n})\\
    & = c_{u_n,n}p_{u_n} + \sum_{k \in E_n} \eta_{k,n} \mu\big([0] \cap \{(1 + k2^{-n})s_n \leq r_{[0]} < (1 + (k+1)2^{-n})s_n\}\big) + o(q_{M_ns_n}) \\
    & = c_{u_n, n}p_{u_n} + \sum_{k \in E_n} \bigg( \eta_{k,n} \mu([0] \cap \{r_{[0]} \geq s_n\})\\
    &  \qquad \qquad \times \frac{\mu\big([0] \cap \{(1 + k2^{-n})s_n \leq r_{[0]} < (1 + (k+1)2^{-n})s_n\}\big)}{\mu([0] \cap \{r_{[0]} \geq s_n\})}\bigg) + o(q_{M_ns_n})\\ 
    & = c_{u_n,n}p_{u_n} + \sum_{k \in E_n} \eta_{k,n} q_{s_n} \frac{q_{(1+k2^{-n})s_n} - q_{(1+(k+1)2^{-n})s_n}}{q_{s_n}} + o(q_{M_ns_n})\\
    & = c_{u_n, n}p_{u_n}+ q_{s_n}\sum_{k \in E_n} \eta_{k,n} \Big((1+k2^{-n})^{-\alpha} - (1+(k+1)2^{-n})^{-\alpha}\Big)\\
    & \qquad \qquad + O( q_{s_n}|E_n|\varepsilon_n) + o(q_{M_ns_n}) \quad \text{by} \; \eqref{eq:uniform_regular_variation_queue},\\
    & = c_{u_n,n}p_{u_n} + q_{s_n} \frac{\pi\alpha}{\sin(\pi \alpha)} \big(\overline{F_W}(m_n) - \overline{F_W}(M_n)\big) + O(q_{s_n}|E_n|\varepsilon_n) + o(q_{M_ns_n}) \quad \text{by}\; \eqref{eq:definition_eta_k,n}.
\end{align*}
Furthermore, we have 
\begin{align*}
    c_{u_n,n}p_{u_n} &= \frac{\rho}{1 - \rho} \sum_{k \in E_n}\sum_{m \in A_{n,k}} \eta_{k,n}p_m \quad \text{by} \; \eqref{eq:definition_c_u_n_n}\\
    & = \frac{\rho}{1 - \rho} \frac{\pi\alpha}{\sin(\pi \alpha)} q_{s_n} (\overline{F_W}(m_n) - \overline{F_W}(M_n)) + O(q_{s_n}|E_n| \varepsilon_n)
\end{align*}
and 
\begin{align*}
    \frac{\rho}{1 - \rho}(\overline{F_W}(m_n) - \overline{F_W}(M_n)) \xrightarrow[n \to +\infty]{} \frac{\rho}{1- \rho}\mathbb{P}(W > 0) = \rho.
\end{align*}
Hence, we obtain
\begin{align*}
    \mu(B_n) & = \rho \frac{\pi\alpha}{\sin(\pi \alpha)} q_{s_n}(1 + o(1)) + q_{s_n}\frac{\pi\alpha}{\sin(\pi \alpha)}(1 - \rho) (1+o(1)) + O(q_{s_n}|E_n|\varepsilon_n) + o(q_{M_ns_n}) \\
    & = \frac{\pi\alpha}{\sin(\pi \alpha)} q_{s_n} + o(q_{s_n}) + o(q_{s_n}|E_n|\varepsilon_n)
\end{align*}
Since $|E_n| = n2^n$, we can choose, for example, $\varepsilon_n = 3^{-n}$ to ensure 
\begin{align}
    \label{eq:equivalence_mu_B_n_q_s_n}
    \mu(B_n) \isEquivTo{n \to +\infty} \frac{\pi\alpha}{\sin(\pi\alpha)}q_{s_n},
\end{align}
\noindent and thus
\begin{align*}
     q_{s_n} \isEquivTo{n \to +\infty} \frac{\sin(\pi\alpha)}{\pi\alpha}a(s_n)^{-1}.
\end{align*}
Putting it into \eqref{eq:equivalence_mu_B_n_q_s_n}, we obtain
\begin{align*}
    \gamma(\mu(B_n))^{-1} = a^{\leftarrow}(\mu(B_n)^{-1}) \isEquivTo{n \to +\infty} a^{\leftarrow}\bigg(\frac{\sin(\pi \alpha)}{\pi\alpha} q_{s_n}^{-1}\bigg) \isEquivTo{n \to \infty} a^{\leftarrow}(a(s_n)) \isEquivTo{n \to +\infty} s_n.
\end{align*} 
\end{proof}

    \noindent Now, let $q \geq 1$ and $k \in E_q$ and set $t = 1 + k2^{-q}$. Using Lemma \ref{lem:connection_between_s_n_and_mu_B_n}, we obtain 
    \begin{align*}
        \mu_{B_n}\big( \gamma(\mu(B_n)r_{[0]} \geq t\big) &= \mu(B_n)^{-1} \sum_{k \geq t/\gamma(\mu(B_n))} c_{k,n}p_k\\ 
        & = \frac{\sin(\pi\alpha)}{\pi\alpha} q_{s_n}^{-1}(1+o(1)) \bigg(\sum_{k \geq ts_n} c_{k,n}p_k + O\big(\big|q(t\gamma(B_n)^{-1}) - q(ts_n)\big|\big)\bigg)\\
        & = \frac{\sin(\pi \alpha)}{\pi\alpha}q_{s_n}^{-1} \sum_{k \geq ts_n} c_{k,n}p_k + o(1).
    \end{align*}
    \noindent The study of the sum is similar to the computation of $\mu(B_n)$ in Lemma \ref{lem:connection_between_s_n_and_mu_B_n}, the only difference being that the sum starts only at $ts_n$ instead of $0$. The choice of $t = 1 + k2^{-q} = 1 + k_n2^{-n}$ for some $k_n \in E_n$ and $n \geq q$, gives
    \begin{align*}
        \sum_{t \geq ts_n} c_{k,n}p_k &= \sum_{ts_n \leq k \leq M_ns_n} c_{k,n}p_k + o(q_{M_ns_n}) \\
        & = q_{s_n}\sum_{\overset{k \in E_n}{k \geq ts_n}} \eta_{k,n} \Big((1+k2^{-n})^{-\alpha} - (1+(k+1)2^{-n})^{-\alpha}\Big) + O( q_{s_n}|E_n|\varepsilon_n) + o(q_{M_ns_n}) \\
        & = q_{s_n} \frac{\pi\alpha}{\sin(\pi \alpha)}(\overline{F_W}(t) - \overline{F_W}(M_n)) + o(q_{s_n}).
    \end{align*}
    Thus, we obtain
    \begin{align*}
        \mu_{B_n}(\gamma(\mu(B_n)r_{[0]} \geq t) = \overline{F_W}(t) - \overline{F_W}(M_n) + o(1)
    \end{align*}
    proving the convergence for all $t > 0$ such that $t \in \{1 + k2^{-q}, \; q \geq 1 \; \text{and} \; k \in E_q\}$ which is dense in $\mathbb{R}$. Since $\overline{F_W} \in \mathcal{G}_{\alpha}$ is continuous on $\mathbb{R}_{> 0}$, this is enough to prove the convergence of \eqref{eq:convergence_waiting_time_in_construction_everything_possible}. Then, a direct application of Theorem \ref{thm:HTS_infinite_points_that_never_come_back} concludes the proof of Proposition \ref{prop:Every_waiting_time_achievable_when_CMS_well_chosen}.
    
\end{proof}

\section{Proof of results from section \ref{section:all-point_REPP_for_examples}}
\label{section:proofs_all-point_REPP_examples}
\subsection{Another sufficient condition for convergence towards fractional Poisson processes}

Before diving into the proof for the examples, we start by Proposition \ref{prop:sufficient_condition_for_0_delay_time_finitely_recurrent} which gives another sufficient condition for the convergence towards the fractional Poisson process. It will be particularly useful in the $\mathbb{Z}$-extension case. The proof relies on a contradiction argument using the results from Section \ref{subsubsection:Images_and_preimages} and Proposition \ref{prop:only_possible_limits_for_waiting_times_in_G}.

\begin{proof}[Proof (of Proposition \ref{prop:sufficient_condition_for_0_delay_time_finitely_recurrent})]
    By Lemma \ref{lem:tightness_gamma_tau_n}, we know that the sequence of measures $(\gamma(\mu(B_n)) \allowbreak \,r_{[x_0]} \circ T^{j(x)})_*\mu_{[x_0^{n-1}]}$ is tight. By contradiction, if it does not converge weakly towards the Dirac mass at $0$, we can assume that there exists a subsequence (without loss of generality we assume the convergence of the whole sequence) and a random variable $W$ with $\mathbb{P}(W > \varepsilon) \geq \delta > 0$ for some $\varepsilon, \delta > 0$ such that 
    \begin{align*}
        \gamma(\mu([x_0^{n-1}]))\, r_{[x_0]} \circ T^{j(x)} \xRightarrow[n \to +\infty]{\mu_{[x_0^{n-1}]}} W.
    \end{align*}
    Then, for all $m\geq 0$, Proposition \ref{prop:convergence_REPP_images} ensure that we have the convergence of the REPP for $T^mx$ towards the renewal process associated to the delay time $Q_{T^mx, m}(x)^{-1/\alpha}\, W$ and thus by the equivalence in Theorem \ref{thm:HTS_infinite_points_that_never_come_back}, we have 
    \begin{align*}
        \gamma(\mu[x_m^{m+n-1}])\, r_{[x_m]} \circ T^{j(T^mx)} \xRightarrow[n \to +\infty]{\mu_{[x_m^{m+n-1}]}} e^{-\alpha^{-1}S_m\phi_*(x)} W = W'_m.
    \end{align*}
    \noindent Thus, we have 
    \begin{align}
        \label{eq:pas_das_G_contradiction}
        \mathbb{P}(W'_m > e^{-S_m\phi_*(x)/\alpha}\varepsilon) \geq \delta. 
    \end{align}
    However, by Proposition \ref{prop:only_possible_limits_for_waiting_times_in_G}, $W'_m \in \mathcal{G}_{\alpha}$ for all $m \geq 0$. This contradicts \eqref{eq:pas_das_G_contradiction} because $e^{-\alpha^{-1}S_m\phi_*(x)} \allowbreak \to +\infty$ when $m\to +\infty$ by hypothesis. Thus, $W = 0$ and, by Theorem \ref{thm:HTS_infinite_points_that_never_come_back}, the REPP associated to $x$ exhibits a fractional Poisson behavior. 
\end{proof}

\subsection{Proof for the examples}

\paragraph{House of Cards.} We start with the proof for the House of Cards type null-recurrent CMS and prove Theorem \ref{thm:all_points_HoC_structure} which establishes the \textit{all-point REPP} property for those systems and identifies the limiting point processes.

\begin{proof}[Proof (of Theorem \ref{thm:all_points_HoC_structure})]
    If $x \notin \mathcal{D}$, this is a direct application of Theorems \ref{thm:HTS_infinite_FPP_infinitely_recurrent} and \ref{thm:HTS_infinite_CFPP_periodic_points}. For points in $\mathcal{D}$, we start by analyzing $x^{up}$. Fortunately, due to the House of Cards structure, we have a nice characterization of $B_n = [(x^{up})_0^{n-1}] = [0 \cdots (n-1)]$ and
    \begin{align*}
        B_n = [0] \cap \{r_{[0]} \geq n\}.
    \end{align*}
    Thus, the limit law of $\gamma(\mu(B_n)) r_{[0]}$ is a direct consequence of the regular variation hypothesis as for all $t \geq 0$, we have
    \begin{align*}
        \mu_{B_n}(\gamma(\mu(B_n))\,r_{[0]} \geq t) & = \mu(B_n)^{-1} \mu(B_n \cap \{\gamma(\mu(B_n))\,r_{[0]}\geq t\})\\
        & = \mu(B_n)^{-1} \mu([0] \cap \{r_{[0]} \geq t/\gamma(\mu(B_n)) \vee n\})
    \end{align*}
    However, by definition of $\gamma$ and Lemma \ref{lem:equivalence_queue_renormalisation_gamma}, we have
    \begin{align*}
        \mu_{B_n}(\gamma(\mu(B_n))\,r_{[0]} \geq t) \xrightarrow[n \to +\infty]{} 1 \wedge \frac{\sin(\pi\alpha)}{\pi\alpha}t^{-\alpha} = \mathbb{P}(W \geq t).
    \end{align*}
    Hence, Theorem \ref{thm:HTS_infinite_points_that_never_come_back} gives the result for $x^{up}$. \\

    \noindent The remaining points can be treated with Proposition \ref{lem:preimages_CMS_delay_time}. For the preimages of $x^{up}$, this is a direct application of Proposition \ref{prop:convergence_REPP_preimages}. Finally, if $y = T^mx^{up}$, the House of Cards structure gives $T^{-m}\{y\} = \{x^{up}\}$. Since $Q_{y,m}$ is a probability density on $T^{-m}\{y\}$ (see Remark \ref{rem:Q_x,m_is_a_probability_density}), we have $Q_{y,m}(x^{up}) = 1$ and the asymptotic behavior of the REPP for $y$ is the same as the one for $x^{up}$. 
\end{proof}

\paragraph{$\mathbb{Z}$-extensions.}
The goal of this section is to prove Theorem \ref{thm:all_points_Z_extension} establishing the \textit{all-point REPP} property for $\mathbb{Z}$-extension over strong positive recurrent CMS. \\

Let us define the category of TMS associated with a Gibbs measure.

\begin{defn}[Gibbs measure] \index{Gibbs measure}
    Let $(\Omega, \sigma)$ be a topologically mixing TMS and $\phi : X \to \mathbb{R}$ be a Walters potential. We say that a measure $\mu$ on $\Omega$ is Gibbs for $\phi$ if it is $\sigma$-invariant and there exist constants $K > 1$ and $P > 0$ such that for all $x\in \Omega$ and $n\geq 1$,
    \begin{align*}
        K^{-1} \leq \frac{\mu[x_0^{n-1}]}{\exp(S_n\phi(x) - nP)} \leq K.
    \end{align*}
\end{defn}

The first result proving the existence of Gibbs measure for subshift of finite type (with a Hölder potential) comes from R. Bowen \cite{Bowen75}. Now, for Walter's potentials, it is possible to give necessary and sufficient conditions for the existence of Gibbs measures. For that we first need to define the big image and preimage property.

\begin{defn}[BIP property] \index{BIP property}
    \label{defn:BIP_property}
    Let $(\Omega, \sigma)$ be a topologically mixing TMS. We say that it has the big images and preimages (BIP) property if there exists there exists a finite set of states $(v_i)_{1\leq i \leq N} \subset V^N$ such that for all $v \in V$, there exists $1\leq i,j\leq N$ such that $v_i \to v$ and $v\to v_j$.
\end{defn}

Then, it is known that there exists necessary and sufficient conditions for a topologically mixing TMS to have Gibbs measures.

\begin{thm}[\textup{\cite[Theorem 1]{Sarig03_Gibbs}, see also \cite[Theorem 4.9]{tdfsurvey}}]
    \label{thm:equivalence_Gibbs_measure}
    Let $(\Omega, \sigma)$ be a topologically mixing TMS and let $\phi$ be a Walters potential. Then, we have the following equivalence between:
    \begin{itemize}
        \item[$(i)$] $(\Omega, \sigma)$ has the BIP property, $P_G(\phi) < +\infty$ and $\var_1(\phi) <+\infty$.
        \item[$(ii)$] $(\Omega, \sigma)$ admits a Gibbs measure.
    \end{itemize}
    In such case, $\phi$ is positive recurrent, its associated measure is $\mu$ and $P = P_G(\phi)$ in the definition of Gibbs measure.
\end{thm}

The BIP property and condition $\var_1(\phi) < +\infty$ allows to strengthen the bounded distortion estimates Lemmas \ref{lem:bounded_distortion_CMS_living_in_compact} and \ref{lem:bounded_distortion_CMS}. This will be key to prove Theorem \ref{thm:all_points_Z_extension} below. The proof is similar to the proof of Lemma \ref{lem:bounded_distortion_CMS} but we take advantage of the additional property $\var_1(\phi) < +\infty$.

\begin{lem}
\label{lem:improved_bounded_distortion_CMS}
    Let $(\Omega, \sigma, \mu)$ be a topologically mixing TMS associated with a Walters potential and a Gibbs measure $\mu$. Then, there exists some constant $C$ such that for all admissible word $(a_0^{n-1}b_0^{j-1})$, we have 
    \begin{align*}
        \mu[a_0^{n-1}b_0^{j-1}] \leq C\mu[a_0^{n-1}]\mu[b_0^{j-1}].
    \end{align*}
\end{lem}

\begin{proof}[Proof (of Lemma \ref{lem:improved_bounded_distortion_CMS})]
    Condition $\var_1(\phi) < +\infty$ ensures that $W_0(\phi) = \sup_{n\geq 1} [\var_n S_n \phi] < +\infty$. Indeed, let $n\geq 1$ and $x, y \in \Omega$ such that $x_0^{n-1} = y_0^{n-1}$. Then, we have 
    \begin{align*}
        S_n\phi(x) - S_n\phi(y) &= S_{n-1}\phi(x) - S_{n-1}\phi(y) + \phi(T^{n-1}(x)) - \phi(T^{n-1}(y)) \\
        & \leq \var_{n} S_{n-1}\phi + \var_1\phi \leq W_1(\phi) + \var_1(\phi).
    \end{align*}
    
    Note that $\phi_* = \phi + \log h - \log h\circ T  - P_G(\phi)$ also satisfies $W_0(\phi_*) < +\infty$ because $\var_1\log h  < +\infty$ (see Remark \ref{rem:regularity_log_h}).\\
    
    For all $n\geq 1$, $(a_0^{n-1})$ admissible and $y,z \in T[a_{n-1}]$, we have 
    \begin{align*}
        \widehat{T}^n \mathbf{1}_{[a_0^{n-1}]}(y) & = L_{\phi_*}^n \mathbf{1}_{[a_0^{n-1}]}(y) = e^{S_n \phi_* (a_0^{n-1}y)}\\
        & = e^{\pm W_0(\phi_*)}e^{S_n\phi_*(a_0^{n-1}z)} = e^{W_0(\phi_*)} \widehat{T}^n \mathbf{1}_{[a_0^{n-1}]}(z).
    \end{align*}
    In particular, it yields
    \begin{align*}
        \widehat{T}^n \mathbf{1}_{[a_0^{n-1}]}(y) = e^{\pm W_0(\phi_*)} \frac{1}{\mu(T[a_{n-1}])} \int_{T[a_{n-1}]} \mathbf{1}_{[a_0^{n-1}]}(z)\,\dd\mu = e^{\pm W_0(\phi_*)} \frac{\mu[a_0^{n-1}]}{\mu(T[a_{n-1}])}. 
    \end{align*}
    
    Denote $\kappa := \inf_{1 \leq i \leq N} \mu[v_i]$ where $(v_i)_{1\leq i\leq N}$ is given by the BIP property. Thus, for all $(a_0^{n-1}b_0^{j-1})$ admissible, we obtain
    \begin{align*}
        \mu[a_0^{n-1}b_0^{j-1}] & = \int \widehat{T}^n\mathbf{1}_{[a_0^{n-1}]} \cdot \mathbf{1}_{[b_0^{j-1}]}\,\dd\mu \leq e^{\pm W_0(\phi_*)} \frac{\mu[a_0^{n-1}]\mu[b_0^{j-1}]}{\mu(T[a_{n-1}])} \\
        & \leq \kappa^{-1}e^{\pm W_0(\phi_*)}\mu[a_0^{n-1}]\mu[b_0^{j-1}]. 
    \end{align*}
\end{proof}

From Definition \ref{defn:Z_extension}, it is not clear that a $\mathbb{Z}$-extension over a countable Markov shift can itself be seen as a countable Markov shift. When the jump function $h$ is locally Lipschitz, we show the folklore result that it is indeed the case an we present a construction.\\

Thus, we consider the special case where $(\Omega, \sigma, \nu)$ is a topologically mixing TMS over the alphabet $\mathcal{A}$ and such that $\nu$ is a Gibbs measure for a Walters potential $\phi$ (in particular $P_G(\phi) <+\infty$ by Theorem \ref{thm:equivalence_Gibbs_measure}). Furthermore, we consider a locally Lipschitz jump function $h$ such that $\int h\,\dd\nu = 0$. To simplify the study, we also assume that $h$ is aperiodic. In particular, by uniform continuity and since $h$ only take discrete values, there exist a depth $L$ such that $h$ is constant on every element of $\mathcal{C}^{\Omega}(L)$ the (at most countable) set of $L$-cylinders for the TMS $(\Omega, \sigma, \nu)$. We call $(X, T,\mu)$ the system thus defined as a $\mathbb{Z}$-extension.\\

We build a CMS (with a countable alphabet, even in the case $\mathcal{A}$ finite) that is topologically conjugated to $(X, T)$.\\

We define the CMS $X'$ over the alphabet $V := \mathcal{C}^{\Omega}(L) \times \mathbb{Z}$ (this is a countable set because $\mathcal{C}^{\Omega}(L)$ is at most countable) and with transition matrix
\begin{align*}
    A( (v_0^{L-1}, z), ((v')_0^{L-1}, z')) = 1 \quad \text{if and only if} \quad v_1^{L-1} = (v')_0^{L-2} \;\text{and}\; z' - z = h|_{[v_0^{L-1}]}.  
\end{align*}

\begin{rem}
    Since $L$ is chosen such that $h$ is constant on elements of $\mathcal{C}^{\Omega}(L)$, we write, with a small abuse of notation, $h|_{[v_0^{L-1}]}$ for the unique value of $h$ taken on $[v_0^{L-1}]$.
\end{rem}

\begin{lem}
    \label{lem:topological_conjugation_CMS_Z_extension}
    The dynamical systems are $(X', T')$ and $(X, T)$ are topologically conjugated by a bi-continuous map $\Pi$.
\end{lem}

\begin{proof}[Proof (of Lemma \ref{lem:topological_conjugation_CMS_Z_extension})]
    We define the map
    \begin{align*}
        \begin{array}{cccc}
            \Pi :& X & \to & X'  \\
             & (\omega, z) & \mapsto & ((\omega_i^{i + L - 1}), z + S_i h(\omega))_{i\geq 0}.
        \end{array}
    \end{align*}
    It is well defined since for all $i \geq 0$, $\omega_{i+1}^{i+L-1} = \omega_{i + 1}^{(i+1) + L - 2}$ and $S_{i+1}h(\omega) - S_{i}h(\omega) = h \circ T^{i}(\omega) = h|_{[\omega_i^{i+L-1}]}$. The map is clearly one to one. Conversely, if $x' = ( ((a^{(i)})_0^{L-1}), z_i)_{i\geq 0} \in X'$, we consider the point $\omega = (a^{(0)})_0^{L-1}(a^{(L)})_0^{L-1}\cdots (a^{(kL)})_0^{L-1}$ and $z = z_0$. We show that $(\omega, z) \in X$ and $\Pi(\omega, z) = x'$. Since $x' \in X'$, for all $i\geq 0$, we have $(a^{(i)})_0^{L-1} \in C^{\Omega}(n)$ and thus $\omega_{i} = (a^{(i)})_0 \to (a^{(i)})_1 = w_{i+1}$ so $\omega \in \Omega$ and $(\omega,z) \in X$. Furthermore, we have $z_{i+1} - z_i = h|_{[(a^{(i)})_0^{L-1}]} = h|_{[\omega_i^{i+L-1}]} = h\circ T^i(\omega)$ and thus, since $z_0 = z$, we have $z + S_ih(\omega) = z_i$ which ensure that $\Pi(\omega, z) = x'$.\\

    \noindent We prove that $\Pi$ is bi-continuous. Let $x = (\omega, z)$ and $x' = (\omega', z)$ be such that $d(\omega, \omega')\leq \eta^{L + p}$ (\textit{i.e.}, $\omega_0^{L+ p-1} = (\omega')_0^{L+ p-1}$). Then, we have
    \begin{align*}
        (\omega_i^{i+L-1}, z + S_ih(\omega)) = ((\omega')_i^{i+L-1}, z + S_ih(\omega')) \quad \forall i \leq p.
    \end{align*}
    and, hence, $d(\Pi(x),\Pi(x')) \leq \eta^{p}$. This is even an equivalence, thus proving the bi-continuity.\\

    \noindent Finally, we show the dynamics compatibility, \textit{i.e.}, $\Pi \circ T = T' \circ \Pi$. Indeed, we have 
    \begin{align*}
        \Pi \circ T (\omega, z) &= \Pi (\sigma\omega, z + h(\omega))\\
        & = ( \omega_{i+1}^{i+L}, z + h(\omega) + S_ih(\sigma\omega))_{i\geq 0}\\
        & = (\omega_{i+1}^{i+L}, z + S_{i+1}(\omega))_{i\geq 0}\\
        & = T' \circ \Pi(\omega, z).
    \end{align*}
\end{proof}

\noindent We set $\mu' := \Pi_*\mu$ which is invariant under $T'$. We show that there exist a potential $\phi' : X' \to \mathbb{R}$ such that $\mu' = \mu_{\phi}$. For that, we need to compute the transfer operator associated to $\mu'$. It is associated with the transfer operator of $(X, T, \mu)$ that we first compute.

\begin{lem}
    \label{lem:computation_transfer_operator_Z_extension}
    The transfer operator $\widehat{T}$ is such that for all $f\in L^1(\mu)$ and $x\in X$, we have
    \begin{align*}
        \widehat{T}f(x) = \sum_{Ty = x} e^{\phi_* (\pr_{\Omega}(y))} f(y),
    \end{align*}
    where $\pr_{\Omega} : (\omega, z) \in X \mapsto \omega \in \Omega$.
\end{lem}

\begin{proof}[Proof (of Lemma \ref{lem:computation_transfer_operator_Z_extension})]
    For all $f \in L^1(\mu)$ and $g\in L^{\infty}(\mu)$, we have
    \begin{align*}
    \int f\cdot g\circ T\,\dd\mu & = \int f\cdot g\circ T\,\dd(\nu\otimes\mathfrak{m}) \\
    & = \sum_{p \in \mathbb{Z}} \int f(\omega, p)g(\sigma \omega, p + h(\omega))\,\dd\nu(\omega) \\
    & = \sum_{p\in \mathbb{Z}} \sum_{q' \in \mathbb{Z}} \int f(\omega, p)\mathbf{1}_{\{h(\omega) = q'\}} g(\sigma\omega, p+q')\,\dd\nu(\omega) \\
    & = \sum_{p\in \mathbb{Z}} \sum_{q \in \mathbb{Z}} \int f(\omega, p)\mathbf{1}_{\{h(\omega) = q - p\}} g(\sigma\omega, q)\,\dd\nu(\omega) \\
    & = \sum_{q\in \mathbb{Z}} \sum_{p\in \mathbb{Z}} \int L_{\phi_*}(f(\cdot, p)\mathbf{1}_{\{h(\cdot) = q - p\}})(\omega)\cdot g(\omega, q)\,\dd\nu(\omega)\\
    & = \sum_{q \in \mathbb{Z}} \int \bigg(\sum_{p\in \mathbb{Z}} \sum_{\sigma\omega' = \omega} e^{\phi_*(\omega')}f(\omega', p)\mathbf{1}_{\{h(\omega') = q' - p\}}\bigg)\, g(\omega, q)\,\dd\nu(\omega)\\
    & = \sum_{q\in \mathbb{Z}} \int \bigg(\sum_{Ty = (\omega, q)} e^{\phi_*(\pr_{\Omega}(y))} f(y)\bigg) g(\omega, q) \,\dd\nu(\omega)\\
    & = \int \bigg( \sum_{Ty = x} e^{\phi_* (\pr_{\Omega}(y))} f(y)\bigg) \, g(x) \,\dd\mu(x).
\end{align*}
    The identity characterizes the transfer operator and we deduce its formula.
\end{proof}

\noindent Thus, for all $f \in L^1(\mu')$ and $x \in X'$, we have 
\begin{align*}
    \widehat{T'}f(x') & = \widehat{T}(f \circ \Pi) \circ \Pi^{-1}(x') \\
    & = \sum_{T'y' = x'} e^{\phi_*(\pr_{\Omega}(\Pi^{-1}(y')))}f(y')\\
    & = L_{\phi'}f(x')
\end{align*}
where $\phi' = \phi_{*} \circ \pr_{\Omega} \circ \Pi^{-1}$.\\

\noindent Furthermore, $\phi'$ recovers the regularity properties of $\phi$ (with a shift of size $L$) as for all $x', y'\in X'$ with $d(x', y') \leq \eta^p$, we have $d(\pr_{\Omega}(\Pi^{-1}(x')), \pr_{\Omega}(\Pi^{-1}(y'))) \leq \eta^{p+L-1}$ and thus 
\begin{align*}
    \var_{n+k} S_n\phi' = \var_{n + k + L} S_n\phi_* \leq W_{k + L}(\phi_*).
\end{align*}

\noindent A ball around a point $x = (\omega, z) \in X$ corresponds to a cylinder (around $x' = \Pi(x)$) in $X'$ (again with a shift $L$) as we have 
\begin{align}
    \label{eq:cylinder_Z_extension_to_CMS_associated}
    B_X(x, \eta^{n + L - 1}) & =  [\omega_0^{n+L-1}] \times \{z\} \nonumber \\
    & = \Pi\Big( [(\omega_0^{L-1}, z)\cdots (\omega_{n}^{n+L-1}, z + S_nh(\omega))]\Big) 
\end{align}

\noindent Finally, we recall some regular variation properties for $\mathbb{Z}$-extension over TMS $(\Omega, T, \mu)$ associated to a weakly Hölder positive recurrent potential $\phi$ and with a Gibbs measure. It depends on the jump function $h$.

\begin{prop}[\textup{see \cite[Table 3.1]{Thomine_These}}]
    Let $h$ be a Lispchitz jump function such that $\int h\,\dd\nu = 0$. Then, we have
    \begin{itemize}
        \item[$(i)$] If $h \in L^2(\nu)$, then the $\mathbb{Z}$-extension $(X, T, \mu)$ is PDE with normalizing sequence $(a_n)_{n\geq 1} \in \RV(\alpha)$ with $\alpha = 1/2$.
        \item[$(ii)$] If $t \mapsto \nu(|h| > t) \in \RV(-\beta)$ with $\beta \in (1, 2]$, then the $\mathbb{Z}$-extension $(X, T, \mu)$ is PDE with normalizing sequence $(a_n)_{n\geq 1} \in \RV(\alpha)$ with $\alpha = 1 - 1/\beta$.
    \end{itemize}
    
\end{prop}

We are now ready to prove Theorem \ref{thm:all_points_Z_extension}.

\begin{proof}[Proof (of Theorem \ref{thm:all_points_Z_extension})]
    We always consider $n \geq L$ so that $B_n = [\omega_{0}^{n-1}] \times \{z\}$ is a cylinder around the point $x' = \Pi(x)$ because $\Pi(B_n) = [(x')_0^{n-L}]$ by \eqref{eq:cylinder_Z_extension_to_CMS_associated}. Thus, for $x$ periodic, this is a direct application of Theorem \ref{thm:HTS_infinite_CFPP_periodic_points} and the extremal index $\theta$ is 
    \begin{align*}
        \theta = e^{S_p\phi'(x')} = e^{S_p \phi_*(\pr_{\Omega}(x))} = e^{S_p\phi_*(\omega)} = e^{S_p\phi(\omega) - pP_G(\phi)}.
    \end{align*}
    
    \noindent If $x$ is non periodic and such that there exists an admissible word $(a_0^{L-1}) \in \mathcal{A}^L$ and $q \in \mathbb{Z}$ such that $|\mathcal{O}(x) \cap ([a_0^{L-1}] \times \{q\})| = +\infty$, then this is a direct consequence of Theorem \ref{thm:HTS_infinite_FPP_infinitely_recurrent}. \\

    \noindent Finally, it remains to show that for the remaining points we still recover the fractional Poisson behavior. This is will be an application of Proposition \ref{prop:sufficient_condition_for_0_delay_time_finitely_recurrent}. By Remark \ref{rem:another_sufficient_condition_delay_time_0_imeges}, this is enough to show that 
    \begin{align}
        \label{eq:in_proof_Z_extension_condition_to_0}
        \lim_{m \to +\infty} \lim_{n \to +\infty} \frac{\mu'[(x')_0^{n-1}]}{\mu'[(x')_m^{n-1}]} = 0.
    \end{align}
    Yet, for all $n > 0$, we have 
    \begin{align*}
        \frac{\mu'[(x')_0^{n+L}]}{\mu'[(x')_m^{n+L}]} & = \frac{\mu([\omega_0^{n-1}] \times \{z\})}{\mu([\omega_m^{n-1}] \times \{z + S_{n-1}h(\omega)\})} \\
        & = \frac{\nu[\omega_0^{n-1}]}{\nu[\omega_m^{n-1}]} \leq C \frac{\nu[\omega_0^{m}]\nu[\omega_m^{n-1}]}{\nu[\omega_m^{n-1}]} \leq C\nu[\omega_0^{m-1}]
    \end{align*}
    by bounded distortion Lemma \ref{lem:improved_bounded_distortion_CMS} for Gibbs CMS. Taking the limit $m\to +\infty$ gives \eqref{eq:in_proof_Z_extension_condition_to_0} and we can conclude.
\end{proof}

\paragraph{Multiple towers.} Here we prove Proposition \ref{prop:all_points_multiple_tower_structure}, which is the same as for Theorem \ref{thm:all_points_HoC_structure}.

\begin{proof}[Proof (of Proposition \ref{prop:all_points_multiple_tower_structure})]
    The proof is similar to the House of Cards case and Theorem \ref{thm:all_points_HoC_structure}. The only change is to study the returns from the points $x^{(i), up}$. However, in this case, we have $B_n = [0]\cap \{\ell_i \geq n\}$ and \eqref{eq:hypothesis_regular_variation_on_each_tower} give the additional scaling factor $p_i$. When $p_i = 0$, the process is the fractional Poisson process (of parameter $\Gamma(1+\alpha)$). 
\end{proof}

\paragraph{Tree house of cards.} Now, we prove Proposition \ref{prop:dichotomy_tree_HoC} establishing the \textit{all-point REPP} property with a dichotomy for null-recurrent Markov chain on the tree house of cards structure.

\begin{proof}[Proof (of Proposition \ref{prop:dichotomy_tree_HoC})]
    Consider the set of points 
    \begin{align*}
        \mathcal{D} := \{ x \in [0] \;|\; x_i \neq 0 \; \forall i\geq 1\}
    \end{align*}
    \textit{i.e.}, $\mathcal{D}$ is the set of points that keep climbing. For every point $x\in \Omega$ that is not in $\bigcup_{k\geq 1} T^k \mathcal{D} \cup \bigcup_{j \geq 0} T^{-j}\mathcal{D}$, the result is a direct consequence of Theorems \ref{thm:HTS_infinite_FPP_infinitely_recurrent} and \ref{thm:HTS_infinite_CFPP_periodic_points}. Furthermore, the extremal index follows from the formula $\theta = 1 - \exp(S_q\phi(x) - qP_G(\phi))$ and the definition of the potential $\phi$ for Markov chains (see Section \ref{section:Connection_with_Markov_Chains}). \\
    For the remaining points, we first show that every point in $\mathcal{D}$ has a $0$ delay limit. Fix some $x \in \mathcal{D}$. Thank to the symmetry of the transition kernel, for all $m\geq n$, we have
    \begin{align*}
        \mu([0] \cap \{r_{[0]}\geq m\}) = 2^n \mu([x_0^{n-1}]\cap \{r_{[0]}\geq m\}).
    \end{align*}
    Thus, for all $t > 0$, we have
    \begin{align*}
        \mu(B_n\cap \{r_{[0]}\geq t/\gamma(\mu(B_n))\}) &= 2^{-n} \mu([0] \cap \{r_{[0]}\geq t/\gamma(\mu(B_n))\}) \\
        & \isEquivTo{n \to +\infty} 2^{-n} t^{-\alpha}\mu(B_n) \quad \text{by Lemma \ref{lem:equivalence_queue_renormalisation_gamma}}\\
        & = o(\mu(B_n))
    \end{align*}
    Thus, the rescaled delay vanishes in the limit and thus point in $\mathcal{D}$ have the fractional Poisson process as a limit by Theorem \ref{thm:HTS_infinite_points_that_never_come_back}. We conclude the proof for the remaining points by Propositions \ref{prop:convergence_REPP_preimages} and \ref{prop:convergence_REPP_images}.
\end{proof}

\section{Proof of Theorem \ref{thm:REPP_subshift_CFPP}}
\label{section:proof_embedded_SFT}

This section is devoted to the proof of Theorem \ref{thm:REPP_subshift_CFPP}. Before diving into the proof, we start by computing the extremal index candidate.\\

As usual we consider $\phi_*$ the normalized potential such that $\widehat{T}_{\mu} = L_{\phi_*}$ (recall that, in particular, $L_{\phi_*}1 = 1$ and $P_G(\phi_*) = 0$). We denote $\phi_{*\Delta}$ the restriction of $\phi_*$ to $\Omega_{\Delta}$. 

\begin{lem}
    \label{lem:measure_Delta_n}
    We have
    \begin{align*}
        \mu(B_{n+1}) / \mu(B_n) \xrightarrow[n \to +\infty]{} e^{P_*}.
    \end{align*}
\end{lem}

\begin{proof}[Proof (of Lemma \ref{lem:measure_Delta_n})]
    Let $p\geq 1$. For all $x \in [\Delta^p]$ (\textit{i.e.}, $x_0^{p-1} \in \Delta^p$), we have
    \begin{align*}
        \widehat{T}^n \mathbf{1}_{[\Delta^n]}(x) & = L_{\phi_*}^n \mathbf{1}_{[\Delta^n]}(x)\\
        & = \sum_{T^n y = x} e^{S_n\phi_*(y)}\mathbf{1}_{[\Delta^n]}(y)\\
        & = \sum_{T^n y' = x'} e^{\pm \var_{n + p} S_n\phi_*} e^{S_n\phi(x')} \mathbf{1}_{[\Delta^n]}(y') \quad \forall x' \in \Omega_{\Delta} \cap [x_0^{p-1}] \\
        & = e^{\pm W_p(\phi_*)} L_{\phi_{*\Delta}}^n 1(x') \quad \forall x' \in \Omega_{\Delta} \cap [x_0^{p-1}].
    \end{align*}
    Since $\Omega_{\Delta}$ is a Subshift of Finite Type, the potential $\phi_{*\Delta} : \Omega \to \mathbb{R}$ is positive recurrent ($P_G(\phi_{*\Delta}) < P_G(\phi_*) <+\infty$ and $\phi_{*\Delta}$ inherits Walters property from $\phi_*$). By the Generalized Ruelle-Perron-Frobenius Theorem \ref{thm:GRPF}, there exist a measure $\nu_{\phi_{*\Delta}}$ and a density $h_{\phi_{*\Delta}}$ such that
    \begin{align*}
        e^{-nP_G(\phi_{*\Delta})} L_{\phi_{*\Delta}}^n 1 \xrightarrow[n \to +\infty]{} h_{\phi_{*\Delta}} \quad \text{pointwise and uniformly on compact sets}.
    \end{align*}
    
    Thus, since $\Omega_{\Delta}$ is compact, we obtain
    \begin{align*}
        \widehat{T}^n \mathbf{1}_{[\Delta^n]}(x) & = e^{\pm W_{p}(\phi_*)} \frac{e^{nP_G(\phi_{*\Delta})}}{\nu_{\phi_{*\Delta}}[x_0^{p-1}]} \int e^{-nP_G(\phi_{*\Delta})}L_{\phi_{*\Delta}}^n 1 \cdot \mathbf{1}_{[x_0^{p-1}]}\,\dd\nu_{\phi_{*\Delta}} \\
        & = e^{\pm W_{p}(\phi_*)} e^{nP_G(\phi_{*\Delta})} \frac{\mu_{\phi_{*\Delta}}[x_0^{p-1}]}{\nu_{\phi_{*\Delta}}[x_0^{p-1}]}(1 + o_n(1))
    \end{align*}
    and hence,  
    \begin{align*}
        \mu(B_{n+p}) = \int \mathbf{1}_{[\Delta^{n+p}]}\,\dd\mu & = \int \widehat{T}^n \mathbf{1}_{[\Delta^{n}]}\cdot \mathbf{1}_{[\Delta^p]}\,\dd \mu \\
        & = \sum_{a_0^{p-1} \in \Delta^p} \int \widehat{T}^n \mathbf{1}_{[\Delta^{n}]}\cdot \mathbf{1}_{[a_0^{p-1}]}\,\dd \mu\\
        & = \sum_{a_0^{p-1} \in \Delta^p}  e^{\pm W_{p}(\phi_*)} e^{nP_G(\phi_{*\Delta})} \frac{\mu_{\phi_{*\Delta}}[a_0^{p-1}]}{\nu_{\phi_{*\Delta}}[a_0^{p-1}]}(1 + o_n(1)) \mu[a_0^{p-1}],
    \end{align*}
    while 
    \begin{align*}
        \mu(B_{n+p+1}) & = \int \widehat{T}^{n+1} \mathbf{1}_{[\Delta^{n+1}]}\cdot \mathbf{1}_{[\Delta^{p}]}\,\dd \mu \\
        & = \sum_{a_0^{p-1} \in \Delta^p}  e^{\pm W_{p}(\phi_*)} e^{(n+1)P_G(\phi_{*\Delta})} \frac{\mu_{\phi_{*\Delta}}[a_0^{p-1}]}{\nu_{\phi_{*\Delta}}[a_0^{p-1}]}(1 + o_n(1)) \mu[a_0^{p-1}].
    \end{align*}
    
     Thus, it yields that 
     \begin{align*}
         e^{- 2W_{p}(\phi_*)}e^{P_G(\phi_{*\Delta})} \leq \liminf_{n\to +\infty} \frac{\mu(B_{n+1})}{\mu(B_n)} \leq \limsup_{n\to +\infty} \frac{\mu(B_{n+1})}{\mu(B_n)} \leq e^{2W_{p}(\phi_*)} e^{P_G(\phi_{*\Delta})}
     \end{align*}
     and when we take the limit $p \to +\infty$, we get
     \begin{align*}
         \frac{\mu(B_{n+1})}{\mu(B_n)} \xrightarrow[n \to +\infty]{} e^{P_G(\phi_{*\Delta})}.
     \end{align*}
     Finally, $\phi_{*\Delta}$ is the restriction of $\phi_{*}$ to $\Omega_{\Delta}$, \textit{i.e.},
     \begin{align*}
         \phi_{*\Delta} = \phi_{\Delta} + \log h|_{\Omega_{\Delta}} - \log h\circ T|_{\Omega_{\Delta}} - P_G(\phi).
     \end{align*}
     and thus $P_G(\phi_{*\Delta}) = P_*$.
\end{proof}

We are now ready to prove Theorem \ref{thm:REPP_subshift_CFPP}.

\begin{proof}[Proof (of Theorem \ref{thm:REPP_subshift_CFPP})]
    As usual, we are going to apply Theorem \ref{thm:sufficient_conditions_convergence_compound_FPP} and take advantage of the fact that $B_n$ is a union of $n$-cylinders. The set $B_1 = [\Delta] := \bigcup_{v \in \Delta}[v]$ is uniform (in fact Darling-Kac) because the induced map on $[\Delta]$ is $\psi$-mixing and by \cite[Lemma 3.7.4]{Aar97} so $B_n$ remains in a fixed uniform set for all $n\geq 1$. We set $U(B_n) = B_{n+1}$ and $Q(B_n) = B_n \backslash B_{n+1}$. Lemma \ref{lem:measure_Delta_n} gives \ref{cond_CFPP:extremal_index} with $\theta := e^{P_*}$.\\

    \noindent We can show that \ref{cond_CFPP:Compatibility_Geometric_law} is satisfied. Let $B_n := \bigcup_{(a_0^{n-1}) \in\Delta^n} [a_0^{n-1}]$. We have
    \begin{align*}
        \widehat{T_{B_n}}\mathbf{1}_{B_{n+1}} &= \widehat{T}\mathbf{1}_{B_{n+1}}  = L_{\phi_*}(\mathbf{1}_{B_{n+1}}) = \sum_{(a_0^{n}) \in \Delta^{n+1}} L_{\phi_*}(\mathbf{1}_{[a_0^{n}]}) = \sum_{(a_0^{n}) \in \Delta^{n+1}} e^{\pm W_n(\phi_*)} L_{\phi_*}(\mathbf{1}_{[a_0^{n}]}) \\
        & = e^{\pm W_{n}(\phi_*)} L_{\phi_*}(\mathbf{1}_{B_{n+1}}).
    \end{align*}
    Thus, since the support of $L_{\phi_*}(\mathbf{1}_{B_{n+1}})$ is $B_n$, we obtain \ref{cond_CFPP:Compatibility_Geometric_law}. \\

    \noindent Take $\tau_n := r_{[\Delta]} \circ T^{n-1} + n - 1$. With this choice, we immediately obtain
    \begin{align*}
        \mu_{U(B_n)}(\tau_n < r_{B_n}) = \mu_{U(B_n)}(\tau_n < 1) = 0
    \end{align*}
    and thus \ref{cond_CFPP:cluster_compatible_tau_n_cluster_from_U}.\\
    It remains to check \ref{cond_CFPP:good_density_after_tau_n}, \ref{cond_CFPP:tau_n_small_enough} and \ref{cond_CFPP:cluster_compatible_tau_n_no_cluster_from_Q}. The proof is similar to the one use for shrinking cylinders in Theorem \ref{thm:HTS_infinite_CFPP_periodic_points}.\\

    \noindent For \ref{cond_CFPP:good_density_after_tau_n}, the proof is also similar to the one use for shrinking cylinders, the compact set should depend on the symbol $v \in \Delta$ that appears at time $r_{[\Delta]} \circ T^{n-1} + n - 1$ for the cylinder considered. However, since $|\Delta| < +\infty$, the set $\bigcup_{v \in \Delta} \mathcal{U}_{K_v, M_v}(v)$ is also compact and thus so is the closure of its convex hull and it is enough to check \ref{cond_CFPP:good_density_after_tau_n}.\\

    \noindent For \ref{cond_CFPP:tau_n_small_enough}, consider $t > 0$. Lemma \ref{lem:measure_Delta_n} implies that $\mu(B_n) \leq K\kappa^{n}$ for some $0 <\kappa < 1$ and $K > 0$ (in fact this is true for every $\kappa > e^{P_*})$. Thus we have 
    \begin{align*}
        \mu_{B_n}(\gamma(\mu(B_n))\tau_n \geq t) & = \mu_{B_n}(\gamma(\mu(B_n))(r_{[\Delta]} \circ T^{n-1} + n-1)\geq t)\\
        & \lesssim \mu_{B_n}\big(r_{[\Delta]} \circ T^{n-1} \geq t/\gamma(\mu(B_n))\big) \\
        & \lesssim C \mu_{[\Delta]}\big(r_{[\Delta]} \geq t/\gamma(\mu(B_n))\big) \xrightarrow[n\to +\infty]{} 0,
    \end{align*}
    where we used the bounded distortion estimate Lemma \ref{lem:bounded_distortion_CMS}.\\

    \noindent For \ref{cond_CFPP:cluster_compatible_tau_n_no_cluster_from_Q}, the proof is the same as in Theorem \ref{thm:HTS_infinite_CFPP_periodic_points} and we do not do it again here.
\end{proof}

\end{appendices}

\end{document}